\documentclass[11pt]{article}
\usepackage[margin=1in,left=1in]{geometry}
\usepackage{amsmath}
\usepackage{mathpazo}
\usepackage{epigraph}
\usepackage{amsfonts}
\usepackage{amssymb}
\usepackage{amsthm}
\usepackage{eucal}
\usepackage{booktabs}
\usepackage{enumitem}
\usepackage{mathtools}
\usepackage{caption}
\usepackage{graphicx}
\usepackage{color}
\usepackage{url}
\usepackage[table]{xcolor}
\usepackage[colorlinks]{hyperref}
\usepackage{stmaryrd}
\usepackage{tikz-cd}
\usepackage[all]{xy}

\usepackage{array}
\newcolumntype{L}[1]{>{\raggedright\let\newline\\\arraybackslash\hspace{0pt}}m{#1}}
\newcolumntype{C}[1]{>{\centering\let\newline\\\arraybackslash\hspace{0pt}}m{#1}}
\newcolumntype{R}[1]{>{\raggedleft\let\newline\\\arraybackslash\hspace{0pt}}m{#1}}

\usepackage[utf8]{inputenc} % Required for inputting international characters
\usepackage[T1]{fontenc}

\numberwithin{section}{part}
\numberwithin{equation}{section}

\definecolor{aqua}{rgb}{0, 1.0, 1.0}
\definecolor{fuschia}{rgb}{1.0, 0, 1.0}
\definecolor{gray}{rgb}{0.502, 0.502, 0.502}
\definecolor{lime}{rgb}{0, 1.0, 0}
\definecolor{maroon}{rgb}{0.502, 0, 0}
\definecolor{navy}{rgb}{0, 0, 0.502}
\definecolor{olive}{rgb}{0.502, 0.502, 0}
\definecolor{purple}{rgb}{0.502, 0, 0.502}
\definecolor{silver}{rgb}{0.753, 0.753, 0.753}
\definecolor{teal}{rgb}{0, 0.502, 0.502}
\definecolor{cite}{RGB}{143, 113, 218}
\definecolor{url}{RGB}{218, 113, 136}

\theoremstyle{plain}
\newtheorem{theorem}{Theorem}[section]
\newtheorem{lemma}[theorem]{Lemma}
\newtheorem*{lemma*}{Lemma}

\newtheorem{proposition}[theorem]{Proposition}
\newtheorem{corollary}[theorem]{Corollary}

 \newtheoremstyle{TheoremNum}
        {\topsep}{\topsep}              %%% space between body and thm
        {\itshape}                      %%% Thm body font
        {}                              %%% Indent amount (empty = no indent)
        {\bfseries}                     %%% Thm head font
        {.}                             %%% Punctuation after thm head
        { }                             %%% Space after thm head
        {\thmname{#1}\thmnote{ \bfseries #3}}%%% Thm head spec
\theoremstyle{TheoremNum}
\newtheorem{thmn}{Theorem}

\theoremstyle{definition}
\newtheorem{definition}[theorem]{Definition}
\newtheorem{example}[theorem]{Example}

\newtheorem{remark}[theorem]{Remark}

\renewcommand{\owedge}{\varowedge}

\newcommand{\spi}{\pi^\mathrm{st}}
\newcommand{\dslash}{/\!\!/}

\providecommand{\keywords}[1]{{\small{\textit{Keywords:}} #1}}

\hypersetup{
  citecolor =  cite,
  urlcolor = url,
  linkcolor = red
  }

\begin{document}

%-------------------------------------------------------------------
%

\title{Strict algebraic models for rational parametrised spectra II}
\author{V.~Braunack-Mayer\footnote{Mathematics, Division of Science, New York University Abu Dhabi, UAE.\newline\indent\indent
\emph{Email address:} \href{mailto://v.braunackmayer@gmail.com}{\tt v.braunackmayer@gmail.com}}}
\date{2020}
\maketitle

\begin{abstract}
In this article, we extend Sullivan's PL de Rham theory to obtain simple algebraic models for the rational homotopy theory of parametrised spectra.
This simplifies and complements the results of \emph{Strict algebraic models for rational parametrised spectra I}, which are based on Quillen's rational homotopy theory.

According to Sullivan, the rational homotopy type of a nilpotent space $X$ with finite Betti numbers is completely determined by a commutative differential graded algebra $A$ modelling the cup product on rational cohomology.
In this article we extend this correspondence between topology and algebra to parametrised stable homotopy theory: for a space $X$ corresponding to the cdga $A$, we prove an equivalence between specific rational homotopy categories for parametrised spectra over $X$ and for differential graded $A$-modules.
While not full, the rational homotopy categories we consider contain a large class of parametrised spectra.
The simplicity of the approach that we develop enables direct calculations in parametrised stable homotopy theory using differential graded modules.

To illustrate the usefulness of our approach, we build a comprehensive dictionary of algebraic translations of topological constructions; providing algebraic models for base change functors, fibrewise stabilisations, parametrised Postnikov sections, fibrewise smash products, and complexes of fibrewise stable maps.
\end{abstract}
\keywords{rational homotopy theory, parametrized spectra, differential graded algebra, minimal models}

\tableofcontents

\section{Introduction}
In this work we build upon Sullivan's rational homotopy theory to produce strict algebraic models for rational parametrised spectra.

Parametrised spectra are continuously varying families of spectra that are parametrised by a topological space. 
First introduced to study the Becker--Gottlieb transfer \cite{clapp_duality_1981}, parametrised spectra naturally arise throughout much of algebraic topology as the classifying objects of twisted homology and cohomology theories.
In applications, parametrised spectra are frequently either invoked explicitly, such as in recent treatments of twisted $K$-theory \cite{ando_twists_2010, hebestreit_homotopical_2020} and generalised Thom spectra \cite{ando_parametrized_2018}, or else they provide a useful, though often implicit, contextual backdrop; for example the Eilenberg--Moore and Atiyah--Hirzebruch spectral sequences are fundamentally statements about parametrised spectra, as are many other classical results.

The usefulness of parametrised spectra stems from the fact that they encode a subtle mixture of both stable and unstable homotopy theory.
For a given space $X$, an $X$-parametrised spectrum is a homotopically-coherent spectrum-valued local system $P\colon x\mapsto P_x$ on $X$.
It is fruitful to think of these data as providing a stable representation of the homotopy type of the base space $X$.
This is true in at least two ways: at each point $x$ of $X$ there is a \lq\lq Serre'' or \lq\lq holonomy'' action of $\Omega_x X$ on the fibre spectrum $P_x$; alternatively, by identifying the basepoints of the fibre spectra $P_x$ at each $x$ we obtain an ordinary (unparametrised) spectrum $X_! P$ equipped with a natural $X_+$-coaction.
That an $X$-spectrum encodes both of these notions of representation is a topological analogue of Koszul duality\footnote{Specifically, it is an analogue of the Koszual duality between Lie algebras and cocommuative coalgebras.}.

For each space $X$, the $X$-spectra are naturally organised into a stable $\infty$-category $\mathrm{Sp}_X$.
While we do not expect the reader to be fluent in the language of higher category theory, we occasionally  argue with stable $\infty$-categories in order to streamline certain arguments.
For instance, since $\mathrm{Sp}_X$ is a stable $\infty$-category, for any $X$-spectra $A, B$ there is a spectrum $X_\ast F_X(A,B)$ of \emph{fibrewise stable maps} from $A$ to $B$, with stable homotopy groups\footnote{Beware! The grading convention here differs by a sign from that of \cite{braunack-mayer_combinatorial_2020, braunack-mayer_strict_2020}.}
\[
\{A,B\}_X^{-k}  := \spi_{k} \big(X_\ast F_X(A,B)\big)\,.
\]
A map of $X$-spectra $f\colon A\to B$ is an equivalence precisely if the induced map of $\mathbb{Z}$-graded abelian groups $\{B, E\}_X^\ast\to \{A, E\}^\ast_X$ is an isomorphism for all $X$-spectra $E$.
This is equivalent to the condition that $f$ induces equivalences $f_x\colon A_x\to B_x$ on all fibre spectra.
Working modulo torsion, $f$ is a \emph{rational equivalence} if for all $X$-spectra $E$ the induced map
\[
\{B, E\}_X^\ast\otimes_\mathbb{Z}\mathbb{Q}
\longrightarrow 
\{A, E\}^\ast_X\otimes_\mathbb{Z}\mathbb{Q}
\]
is an isomorphism of $\mathbb{Z}$-graded rational vector spaces.
Equivalently, $f$ is a rational equivalence if each of the induced maps  $f_x\colon A_x\to B_x$ is a rational equivalence of fibre spectra.
In this article we are interested in studying the rational homotopy theory of $X$-spectra; that is, the localisation of $\mathrm{Sp}_X$ at the class of rational equivalences.

In the case that the parameter space $X$ is simply connected, Quillen \cite{quillen_rational_1969} shows that the rational homotopy type is completely determined by relatively simple algebraic data.
A natural question to ask, therefore, is whether the rational homotopy theory of parametrised spectra over a simply connected space also admits a description in terms of simple algebraic categories.
Previously in Part I \cite{braunack-mayer_strict_2020} we answered this question in the affirmative by proving a series of equivalences between the rational homotopy category of parametrised spectra over a simply connected space and algebraic homotopy categories of differential graded Lie representations and comodules.
These equivalences are highly structured and give a means by which to realise many constructions with parametrised spectra in terms of algebra; for example the fibrewise smash product of parametrised spectra becomes either (i) the derived tensor product of Lie representations or (ii) the derived cotensor product of comodules.
The trouble with these results is their complexity.
Building on Quillen's original arguments, we obtained equivalences of rational homotopy categories by deriving a long zig-zag of Quillen equivalences of model categories.
This complicates the issue of translating freely between algebra and topology, especially since tractable concrete models are difficult to lay hands on.

In this article we pursue a dual approach to the problem of producing strict algebraic models for rational parametrised spectra.
We work with Sullivan's approach to rational homotopy theory, which is based on a single pair of adjoint functors relating the category of simplicial sets and the category of rational commutative differential graded algebras (cdgas).
This adjunction induces an equivalence between the homotopy category of nilpotent spaces with finite Betti numbers and a certain homotopy category of cdgas.
The result is dual to Quillen's in the sense that the algebraic data involved are algebras computing cohomology rather than coalgebras computing homology.
While Sullivan's rational homotopy theory does require certain finiteness conditions, it is the vastly simpler of the two approaches.
This simplicity, along with its ability to deal with nilpotent $\pi_1$-actions, is what makes the Sullivan approach suitable for a very wide range of applications.

According to Sullivan, then, the rational homotopy type of a nilpotent space $X$ with finite Betti numbers may be identified with an equivalence class $[A]$  of cdgas up to quasi-isomorphism. 
For any cdga $A$ in this class, the cohomology algebra $H^\bullet(A)$ computes the cup product structure of $H^\bullet(X;\mathbb{Q})$ and, conversely, the rational homotopy groups of $X$ can be recovered from a minimal model of $A$.
By carefully choosing representatives $X$ and $A$ of the rational homotopy and quasi-isomorphism types, respectively, we are able to produce an adjunction
\begin{equation}
\label{eqn:TheAdjunction}
\begin{tikzcd}
\mathrm{Sp}_X
\ar[rr, shift left = 2, "\mathfrak{M}_A"]
\ar[rr, shift left = -2, leftarrow, "\bot", "\mathfrak{P}_A"']
&&
A\mathrm{-Mod}^\mathrm{op}
\end{tikzcd}
\end{equation}
relating $X$-spectra to differential graded $A$-modules.
We establish this adjunction using the combinatorial models for parametrised spectra studied in \cite{braunack-mayer_combinatorial_2020} and demonstrate that the adjunction is Quillen for the stable model structure on sequential parametrised spectra.
For any $X$-spectrum $P$, the $X_+$-coaction on the collapse spectrum $X_! P$ gives rise to an action of $H^\bullet(X;\mathbb{Q})$ on $H^\bullet(X_!P;\mathbb{Q})$.
If $P$ is a cofibrant sequential $X$-spectrum, the cohomology of the $A$-module $\mathfrak{M}_A(P)$ computes this $H^\bullet(X;\mathbb{Q})$-action.
On the other hand, the right adjoint $\mathfrak{P}_A$ provides a means by which to realise differential graded modules over cdgas directly in topology.

The main result of this article is that the adjoint functors of \eqref{eqn:TheAdjunction} induce inverse equivalences between specific rational homotopy categories of parametrised spectra and of differential graded modules:
\begin{thmn}[\ref{thm:RatParamHomThry}]
Let $X$ be a nilpotent space with finite Betti numbers and let $A$ be a cofibrant connected cdga modelling the rational homotopy type of $X$.
There is an adjoint equivalence of categories
\[
\begin{tikzcd}
Ho\big(\mathrm{Sp}_{X}\big)^\mathbb{Q}_{\mathrm{f.t.,nil,bbl}}
\ar[rr, shift left =2, "\mathbf{L}\mathfrak{M}_A"]
\ar[rr, leftarrow, shift left =-2, "\simeq", "\mathbf{R}\mathfrak{P}_A"']
&&
Ho\big(A\mathrm{-Mod})^\mathrm{op}_\mathrm{f.h.t.}
\end{tikzcd}
\]
between:
\begin{itemize}
  \item The rational homotopy category of nilpotent $X$-spectra whose fibrewise stable homotopy groups are bounded below and of degreewise finite rank; and 
  
  \item The homotopy category of $A$-modules of finite homotopical type.
\end{itemize}
\end{thmn}
\noindent Let us briefly explain the various qualifications needed in this theorem since,
unlike the results of \cite{braunack-mayer_strict_2020}, here we must impose some finiteness conditions in order to get an equivalence of rational homotopy theories.
On the topological side we work with $X$-spectra $P$ for which the fibrewise stable homotopy groups are (i) bounded below and (ii) of degreewise finite rank.
Since we allow nilpotent parameter spaces $X$ we also work with \emph{nilpotent $X$-spectra}, which are $X$-spectra  $P$ such that the Serre action of $\pi_1X$ on each fibrewise stable homotopy group $\spi_k P_x$ is nilpotent.
Despite these restrictions we are still left with a relatively large class of parametrised spectra containing, for instance, the fibrewise suspension spectra of nilpotent fibrations over $X$. 

The conditions that we impose on the algebraic side are a little more complicated to describe.
For any bounded below $A$-module $M$ there is a quasi-isomorphism of $A$-modules $A\otimes V\to M$  for which the underlying graded module of the domain is free on some graded rational vector space $V$ and the differential on $A\otimes V$ is of a certain prescribed form.
In this case we say that $A\otimes V$ is a \emph{minimal model} of $M$.
A bounded below $A$-module is of \emph{finite homotopical type} if it has a minimal model of the form $A\otimes V$ where $V$ is degreewise finite dimensional.
Under the equivalence of Theorem \ref{thm:RatParamHomThry}, the condition that an $A$-module be of finite homotopical type ensures that the fibrewise stable homotopy groups of the $X$-spectrum $\mathbf{R}\mathfrak{P}_A(M)$ are of degreewise finite rank, and conversely.

An important consequence of Theorem \ref{thm:RatParamHomThry} is the fact that any algebraic construction in terms of differential graded modules over a cdga has a topological translation.
The final pages of this article are given over to exploring various aspects of this rational homotopy theory dictionary between topology and algebra;
the results are summarised in Table \ref{table}.
We find straightforward algebraic identifications of many important constructions from  parametrised stable homotopy theory: fibrewise stabilisation and destabilisation; pushforwards and pullbacks along a map of base of spaces; connective covers and Postnikov sections.
Extending classical results of Eilenberg and Moore, we show that  Theorem \ref{thm:RatParamHomThry} identifies the fibrewise smash product of $X$-spectra with the derived tensor product of $A$-modules in rational homotopy theory. 
Generalising the main result of \cite{felix_fibrewise_2010}, we show that rational homotopy classes of fibrewise stable maps between $X$-spectra are computed as $\mathrm{Ext}$-groups of $A$-modules and we use this result to derive spectral sequences for computing fibrewise stable maps.

\paragraph*{Organisation.}
This article is organised as follows:

In Section \ref{S:ParamSpec} we give a brief $\infty$-categorical overview of the central constructions of parametrised stable homotopy theory: parametrised spectra, base change functors, fibrewise stabilisations of retractive spaces, and connective covers and Postnikov sections.
We often need to work with point-set models for parametrised spectra, for which we recall the \emph{sequential parametrised spectra} of \cite{braunack-mayer_combinatorial_2020}.
Using these combinatorial models for parametrised spectra we prove a useful technical result characterising fibrewise Eilenberg--Mac Lane spectra parametrised over a connected space $X$ in terms of $\mathbb{Z}[\pi_1 X]$-modules.
In Section \ref{SS:NilParamSpec} we study a novel class of nilpotent parametrised spectra.
An $X$-spectrum over connected $X$ is nilpotent if $\pi_1 X$ acts nilpotently on its fibrewise stable homotopy groups; for an $X$-spectrum $P$ with bounded below fibrewise stable homotopy groups we show that $P$ is nilpotent if and only if it can be obtained from the trivial $X$-spectrum by a sequence of untwisted fibrewise Eilenberg--Mac Lane spectra over $X$.
The class of nilpotent bounded below $X$-spectra is important to us for two reasons: it contains the fibrewise stabilisations of nilpotent fibrations over $X$ and it is the topological setting for our rational homotopy theory results of Sections \ref{S:FibStabRat} and \ref{S:Dictionary}.

In Section \ref{S:RHT} we recall the Sullivan--de Rham equivalence theorem as proven by Bousfield and Gugenheim in \cite{bousfield_rational_1976}.
This equivalence of rational homotopy theories is induced by a Quillen adjunction between the model categories of simplicial sets and of commutative differential graded algebras (or, more specifically, the opposite category).
Passing to slice categories and taking augmentation ideals we get Quillen adjunctions between model categories of retractive spaces and connective differential graded modules.
We give a topological interpretation of the module associated to a retractive space in terms of a natural cohomology action on the quotient and prove a key relation between suspensions of retractive spaces and their associated modules.

Section \ref{S:FibStabRat} contains the proof of our main theorem.
Stabilising the adjunctions of Section \ref{S:RHT}, for each cdga $A$ we prove a Quillen adjunction between (the opposite of) the model category of unbounded $A$-modules and the model category of sequential spectra parametrised over the Sullivan spatial realisation of $A$.
After proving several convenient properties of these adjunctions, we formulate and prove our main result: Theorem \ref{thm:RatParamHomThry}.
The method of proof essentially boils down to an induction over extensions by untwisted fibrewise Eilenberg--Mac Lane spectra, for which the nilpotence condition is indispensable.

In the final Section \ref{S:Dictionary} we expand on our equivalence theorem by elucidating various aspects of a rational homotopy theory dictionary that relates algebraic constructions with modules to topological constructions with parametrised spectra.
The results are summarised in Table \ref{table}.
We give algebraic constructions modelling pushforwards and pullbacks of parametrised spectra, fibrewise stabilisation of retractive spaces, and connective covers and Postnikov sections of parametrised spectra.
We prove that fibrewise smash products (of nilpotent, bounded below parametrised spectra of finite type) are identified in rational homotopy theory with the derived tensor product of modules, generalising classical results of Eilenberg and Moore and situating them in the broader context of parametrised stable homotopy theory.
We conclude this discussion by demonstrating that rational homotopy classes of fibrewise stable maps between $X$-spectra are computed by forming $\mathrm{Ext}$-groups of the corresponding $A$-modules.
In favourable situations, we show that these $\mathrm{Ext}$-groups can be calculated by means of either a \lq\lq hyper-$\mathrm{Ext}$'' or \lq\lq minimal'' spectral sequence.

\paragraph*{Notation and conventions.}
\begin{itemize}
	\item We use $\infty$-categories in the first few pages of the article, but the reader is not required to be at all fluent in this language.
	The bulk of our work is carried out using explicit point-set models for homotopy types in the context of specific model categories.
	
	\item Our focus is rational homotopy theory and so, unless otherwise indicated, all homology and cohomology groups are with rational coefficients. For example $H^\bullet(-)$ is understood to mean $H^\bullet(-;\mathbb{Q})$ for both spaces and spectra.
	
	\item The category of rational cochain complexes is denoted by $Ch$.
	For a rational cochain complex $V\in Ch$ and an element $v\in V$ of homogeneous degree, we write $|v|$ for the degree of $v$.
	
	\item With a single exception, we adopt the terminology and notation for parametrised stable homotopy theory used in \cite{braunack-mayer_combinatorial_2020, braunack-mayer_strict_2020}.
	The exception is this: for $X$-spectra $P$ and $Q$ the $\mathbb{Z}$-graded abelian group of fibrewise stable maps from $P$ to $Q$ is given in degree $k$ by 
	\[
	\{P,Q\}_X^k = \mathrm{Sp}_X\big(P, \Sigma^k_X Q\big) \cong
	\mathrm{Sp}_X\big(\Sigma^{-k}_X P, Q\big)\,,
	\]
	which differs from \emph{loc.~cit.}~by a minus sign. 
	We make this change for two related reasons: upper indices in homotopy theory really ought to be \lq\lq cohomologically graded'', and we wish to be consistent with existing notation (e.g.~in \cite{crabb_fibrewise_1998, felix_fibrewise_2010}).
\end{itemize}

\paragraph*{Acknowledgements.}
This article is the third and final in a series \cite{braunack-mayer_combinatorial_2020, braunack-mayer_strict_2020} based on my 2018 PhD
thesis. I would like to thank Adam Rohrlach, Hisham Sati and Urs Schreiber for useful comments and encouragement during the writing of this article.

%%%%%%%%%%%%%%%%%%%%%%
%%%%%%%%%%%%%%%%%%%%%%

\section{Parametrised spectra}
\label{S:ParamSpec}
In this section we give a brief overview of parametrised spectra and provide a treatment of nilpotent parametrised spectra, a notion which appears to be new.

Intuitively, a parametrised spectrum is a spectrum-valued system of local coefficients over some base space $X$:
\begin{itemize}
  \item To each point $x$ in $X$ we assign a spectrum $P_x$ in a suitably continuous way,
  
  \item Each path $\gamma$ from $x$ to $y$ gives rise to a map of spectra $\gamma_\ast\colon P_x\to P_y$,
  
  \item A based homotopy of paths $h\colon \gamma\to \delta$ determines a homotopy between $\gamma_\ast$ and $\delta_\ast$,
\end{itemize}
and so on.
Such an object encodes a homotopically-coherent  Serre (or \lq\lq holonomy'') action of loops at $x\in X$ on the spectrum $P_x$, combining the unstable homotopy of $X$ (or, rather, of $\Omega_x X$) with the stable homotopy of $P_x$.

Regarding the base space $X$ as an $\infty$-groupoid, the most succinct definition of an \emph{$X$-parametrised spectrum} is as an $\infty$-functor $X\to \mathrm{Sp}$ valued in the stable $\infty$-category of spectra.
The $\infty$-category of $X$-parametrised spectra is thus
\[
\mathrm{Sp}_X := \mathrm{Fun}(X%^\mathrm{op}
, \mathrm{Sp})\,,
\]
the $\infty$-category of spectrum-valued presheaves on $X$%\footnote{Note that since $X$ is an $\infty$-groupoid there is a canonical equivalence $X^\mathrm{op}\cong X$.}
.
The suspension and loop functors prolong objectwise to auto-equivalences $\Sigma_X\colon (x\mapsto  P_x)\mapsto (x\mapsto \Sigma P_x)$ and $\Omega_X\colon (x\mapsto P_x) \mapsto (x\mapsto \Omega P_x)$ with respect to which $\mathrm{Sp}_X$ is a stable $\infty$-category.
The smash product of spectra furnishes $\mathrm{Sp}_X$ with a \emph{fibrewise smash product} $P \wedge_X Q \colon (x\mapsto P_x\wedge Q_x)$, which is closed symmetric monoidal (in the appropriate $\infty$-categorical sense).
For a map of base spaces $f\colon X\to Y$ there is an associated adjoint triple of exact functors
\begin{equation}
\label{eqn:BChangeooCat}
(f_!\dashv f^\ast\dashv f_\ast)\colon
\begin{tikzcd}
\mathrm{Sp}_X
\ar[rr, leftarrow, "\bot"]
\ar[rr, shift left =4]
\ar[rr, shift left =-4, "\bot"]
&&
\mathrm{Sp}_Y
\end{tikzcd}
\end{equation}
obtained respectively by left Kan extension, pullback, and right Kan extension along $f$.
The pullback functor $f^\ast$ is strongly closed monoidal, a consequence of which is the projection formula
\[
f_! (f^\ast A\wedge_X B) \cong A\wedge_Y f_! B
\]
for $A\in \mathrm{Sp}_Y$ and $B\in \mathrm{Sp}_X$.
\begin{remark}
\label{rem:ComoduleStructure}
For any space $X$, we write $X_! \colon \mathrm{Sp}_X \to \mathrm{Sp}$ for the pushforward along the terminal map $X\to \ast$.
For an $X$-spectrum $P$, the projection formula endows $X_!P$ with the structure of an $X_+$-comodule spectrum as follows.
Forming fibrewise smash products of $P$ with the component of the $(X_!\dashv X^\ast)$-unit at the fibrewise sphere spectrum $X^\ast S$ and applying $X_!$, we get
\[
\begin{tikzcd}[row sep = tiny]
X_! (X^\ast S\wedge_X P)
\ar[d, equals, "\wr"]
\ar[r]
&
X_!(X^\ast X_!X^\ast S\wedge_X P)
\ar[d, equals, "\wr"]
\\
X_!P
&
X_!X^\ast X
\wedge X_! P
\end{tikzcd}
\]
by the projection formula.
The pushforward $X_! X^\ast S \cong \Sigma^\infty_+ X$ is the suspension spectrum of $X$ (see Example \ref{exam:TypicalXSpec} (1)) and the resulting morphism of spectra $X_! P\to X_+ \wedge X_! P$ defines a coaction of $X_+$ (equivalently of $\Sigma^\infty_+ X$) on $X_! P$.
\end{remark}

For $A, B \in \mathrm{Sp}_X$ there is a spectrum of fibrewise stable maps from $A$ to $B$.
Explicitly, this spectrum is obtained by forming the spectrum $X_\ast F_X(A,B)$ of sections of the fibrewise internal hom $F_X(A,B)$.
The stable homotopy groups of this spectrum comprise the $\mathbb{Z}$-graded abelian group of \emph{fibrewise stable maps} from $A$ to $B$, which is given in dimension $k$ by
\[
\{A,B\}_X^k := \spi_{-k} \big(X_\ast F_X(A,B)\big)
\cong 
\mathrm{Sp}_X \big(A, \Sigma^k_X  B\big)\,.
\]
A map of $X$-spectra $A\to B$ is an equivalence in $\mathrm{Sp}_X$ precisely if for all $E\in \mathrm{Sp}_X$, the induced map of $\mathbb{Z}$-graded abelian groups $\{B,E\}_X^\ast \to \{A,E\}^\ast_X$ is an isomorphism.
Any point $x\colon \ast \to X$ gives rise to a pushforward functor $x_!\colon \mathrm{Sp}\to \mathrm{Sp}_X$ and, if $X$ is connected, the pushforward of the sphere spectrum $x_! S$ detects equivalences.
Note that if $x$ and $y$ are in the same path component of $X$, a choice of path from $x$ to $y$ determines an equivalence of functors $x_!\cong y_!$ and hence also between adjoints $x^\ast \cong y^\ast$.
In the case that $y=x$, the space $\Omega_x X$ acts by automorphisms on $x^\ast$ and we have the following
\begin{proposition}
\label{prop:FibModEquiv}
For any connected space $X$ there is an equivalence of stable $\infty$-categories between $\mathrm{Sp}_X$ and the $\infty$-category $\Omega X_+\mathrm{-Mod}$ of $\Omega X_+$-module spectra.
Under this equivalence, the base change adjunction $(x_!\dashv x^\ast)\colon \mathrm{Sp}\to \mathrm{Sp}_X$ induced by $x\colon \ast \to X$ is identified with the free-forgetful adjunction
\[
\begin{tikzcd}
\mathrm{Sp}
\ar[rr, shift left=2 , "\Omega X_+ \wedge -"]
\ar[rr, leftarrow, shift left =-2, "\bot"]
&&
\Omega X_+\mathrm{-Mod}\,.
\end{tikzcd}
\]
\end{proposition}
\begin{proof}[Sketch of proof.]
The monad of the adjunction $(x_!\dashv x^\ast)$ is $T\colon P\mapsto x^\ast x_! P$.
Since $x_!$ and $x^\ast$ preserve colimits, $T$ is determined by the image of the sphere spectrum $T(P)\cong T(S)\wedge P$.
One finds that $T(S)\cong \Sigma^\infty_+\Omega X$, with monadic structure corresponding to the concatenation of loops (this can be seen explicitly in certain model category presentations, e.g.~\cite[Example 2.20]{braunack-mayer_combinatorial_2020}).
The adjunction $(x_!\dashv x^\ast)$ thus factors as
\[
\begin{tikzcd}
\mathrm{Sp}
\ar[rr, shift left = 2, "\Omega X_+ \wedge -"]
\ar[rr, shift left=-2, leftarrow, "\bot", "\text{forgetful}"']
&&
\Omega X_+\mathrm{-Mod}
\ar[rr, shift left = 2, "L"]
\ar[rr, shift left = -2, leftarrow, "\bot", "x^\ast"']
&&
\mathrm{Sp}_X\,,
\end{tikzcd}
\]
where the existence of the left adjoint $L$ is a consequence of the adjoint functor theorem for $\infty$-categories.
The forgetful functor $\Omega X_+\mathrm{-Mod}\to \mathrm{Sp}$ preserves and reflects equivalences, as does $x^\ast\colon \mathrm{Sp}_X\to \mathrm{Sp}$ since $X$ is connected.
It follows that $x^\ast\colon \mathrm{Sp}_X\to \Omega X_+\mathrm{-Mod}$ preserves and reflects equivalences, so to show that the adjunction $(L\dashv x^\ast)$ is an equivalence it suffices to shows that the unit $M\to x^\ast L M$ is an equivalence for all $\Omega X_+$-module spectra $M$.
For this, we let $\mathcal{E}$ be the full subcategory on $Ho(\Omega X_+\mathrm{-Mod})$ on homotopy classes of $\Omega X_+$-module spectra for which $M \to x^\ast LM$ is an equivalence.
It is straightforward to verify that $\mathcal{E}$ is a localising subcategory.
The triangulated category $Ho(\Omega X_+\mathrm{-Mod})$ has weak compact generator (the homotopy class of ) $\Sigma^\infty_+ \Omega X$ and the component of the unit at this object is equivalent to the identity:
\[
\Sigma^\infty_+\Omega X \longrightarrow x^\ast L \Sigma^\infty_+\Omega X
\cong
x^\ast x_! S \cong \Sigma^\infty_+ \Omega X\,.
\]
Since $\mathcal{E}$ is localising and contains the weak compact generator, we have $\mathcal{E} = Ho(\Omega X_+\mathrm{-Mod})$ and thus $(L\dashv x^\ast)$ is an equivalence of stable $\infty$-categories. 
\end{proof}

\begin{corollary}
\label{cor:CompactGen}
For any connected space $X$,  $\mathrm{Sp}_X$ is compactly generated by the pushforward $x_! S$ of the sphere spectrum along any map $x\colon \ast \to X$. 
\end{corollary}

\begin{remark}
\label{rem_FibSmashProd}
The equivalence of stable $\infty$-categories $\mathrm{Sp}_X\cong \Omega X_+\mathrm{-Mod}$ identifies the fibrewise smash product of $X$-spectra $P\wedge_X Q$ with the smash product of fibre spectra $x\mapsto x^\ast P\wedge x^\ast Q$ equipped with the diagonal $\Omega X_+$-action (see  \cite[Section 2.4]{braunack-mayer_strict_2020} for an argument using model categories).
\end{remark}

\begin{remark}
All results stated for connected base spaces throughout this section are true for non-connected base spaces with some minor adjustments---one simply works separately over each connected component.
It is simpler to state results for connected spaces and so this is what we do.
\end{remark}

The main goal of this article is to provide algebraic models for rational parametrised spectra.
A map of $X$-spectra $A\to B$ is a \emph{rational equivalence} precisely if for all $E\in \mathrm{Sp}_X$, the induced map $\{B, E\}_X^\ast\otimes_\mathbb{Z}\mathbb{Q}\to \{A, E\}_X^\ast\otimes_\mathbb{Z}\mathbb{Q} $ is an isomorphism of graded rational vector spaces.
Localising at the class of rational equivalences we get a reflective subcategory
\[
\begin{tikzcd}
\mathrm{Sp}_X^\mathbb{Q}
\ar[rr, hookrightarrow, shift left =-2, "\bot"]
\ar[rr, leftarrow, shift left =2]
&&
\mathrm{Sp}_X
\end{tikzcd}
\]
of \emph{rational parametrised spectra}, with localisation functor given by smashing with the Eilenberg--Mac Lane spectrum $H\mathbb{Q}$ at each fibre.
\begin{remark}
\label{rem:RatEquiv}
The following conditions on a map of $X$-spectra $f\colon A\to B$ are equivalent:
\begin{itemize}
  \item $f$ is a rational equivalence of $X$-spectra,
  \item $f\wedge_X X^\ast H\mathbb{Q}$ is an equivalence of $X$-spectra,
  \item $x^\ast f\wedge H\mathbb{Q}$ is an equivalence of spectra for all $x\in X$, and 
  \item $f$ induces an isomorphism of graded rational vector spaces $\spi_\ast x^\ast A\otimes_\mathbb{Z}\mathbb{Q}\to \spi_\ast x^\ast B\otimes_\mathbb{Z}\mathbb{Q}$ for all $x\in X$.
\end{itemize}
In practice, it suffices to check the latter two conditions at a single point in each connected component of $X$.
\end{remark}

The first examples of parametrised spectra arise by forming fibrewise suspension spectra.
A \emph{retractive space} over $X$ is a space $Z$ with  maps $i\colon X\to Z$ and $p\colon Z \to X$ exhibiting $X$ as a retract (so that $p\circ i = \mathrm{id}_X$).
Retractive spaces over $X$ form an $\infty$-category $\mathcal{S}_{\dslash X}$, which by Lurie's straightening/unstraightening construction is equivalent to the $\infty$-category $\mathrm{Fun}(X%^\mathrm{op}
, \mathcal{S}_\ast)$ of presheaves on $X$ valued in pointed spaces.
Given a retractive space $Z$ over $X$ realised as a functor $\underline{Z}\colon X%^\mathrm{op}
\to \mathcal{S}_\ast$, the corresponding \emph{fibrewise suspension spectrum} is the composite
\[
\Sigma^\infty_X \underline{Z}
\colon 
\begin{tikzcd}
X%^\mathrm{op}
\ar[r, "\underline{Z}"]
&
\mathcal{S}_\ast
\ar[r, "\Sigma^\infty"]
&
\mathrm{Sp}
\end{tikzcd}
\]
given by stabilising the fibres of $Z\to X$.
Stabilising Lurie's straightening/unstraightening construction gives  an equivalence of $\infty$-categories $\mathrm{Sp}_X\simeq \mathrm{Stab}(\mathcal{S}_{\dslash X})$.
It follows that fibrewise suspension spectra generate $\mathrm{Sp}_X$ under colimits and finite limits.
\begin{example}
\label{exam:TypicalXSpec}
There are two especially useful examples of $X$-spectra to keep in mind:
\begin{enumerate}[label=(\arabic*)]
  \item The fibrewise stabilisation of $X\times  S^0 \cong X\coprod X$, which is equivalent to the pullback $X^\ast S$ of the sphere spectrum along the terminal map $X\to \ast$.
  For connected $X$, under the equivalence of Proposition \ref{prop:FibModEquiv} $X^\ast S$ corresponds to the sphere spectrum $S$ with trivial $\Omega X_+$-action.
  
  \item The fibrewise stabilisation of the wedge sum $X\vee_x S^k$ of the $k$-sphere at a point $x\in X$, which is equivalent to the pushforward $x_! \Sigma^k S \cong \Sigma^k_X x_! S$.
  For connected $X$, $\Sigma^k_X x_! S$ corresponds under Proposition \ref{prop:FibModEquiv} to the $k$-shifted suspension spectrum $\Sigma^k \Sigma^\infty_+\Omega_x X$ regarded as an $\Omega_x X$-module spectrum in the obvious way.
\end{enumerate}
\end{example}

\begin{example}
\label{exam:FibrewiseSuspensionSpectra}
Any map of spaces $p\colon Y\to X$ gives rise to a retractive space $Y_{X+}$
\[
\begin{tikzcd}
X
\ar[r]
&
\displaystyle{X\coprod Y}
\ar[r, "\mathrm{id}_X + p"]
&
X\,
\end{tikzcd}
\]
which is the result of freely adjoining a basepoint to each of the fibres of $p$.
We write $\Sigma_{X+}^\infty Y$ for the corresponding fibrewise suspension spectrum;
the fibre spectrum of $\Sigma_{X+}^\infty Y$ at  $x\in X$ the suspension spectrum $\Sigma^\infty_+ F_x$ of the fibre of $p$ at $x$. 

If $X$ is connected, the fibre $F$ of $p$ has a natural $\Omega X$-action.
Under Proposition \ref{prop:FibModEquiv}, the fibrewise suspension spectrum $\Sigma_{X+}^\infty Y$ corresponds to $\Sigma^\infty_+ F$ with the stabilised $\Omega X_+$-action:
\[
\Sigma^\infty_+( \Omega X\times F \to F)\cong (\Omega X_+\wedge \Sigma^\infty_+ F\to \Sigma^\infty_+ F)\,.
\]
\end{example}

Analogously to the unparametrised setting, we can form connective covers and Postnikov sections of parametrised spectra.
While the basic ideas are straightforward, we did not find the details set down in the literature and so we record them here.
\begin{definition}
For an integer $k$, an $X$-spectrum $P$ is \emph{$k$-connective} if the spectrum $x^\ast P$ is $k$-connective for any $x\colon \ast \to X$.
Similarly, $P$ is \emph{$k$-coconnective} if $x^\ast P$ is $k$-coconnective for any $x\colon \ast \to X$.
The full subcategories of $\mathrm{Sp}_X$ spanned by the $k$-connective and $k$-coconnective spectra are $\mathrm{Sp}_X^{\geq k}$ and $\mathrm{Sp}_X^{\leq k}$ respectively.
\end{definition}
\begin{remark}
If the base space $X$ is connected, it suffices to check for $k$-(co)connectivity at a single point of $X$.
\end{remark}
\begin{proposition}
\label{prop:kconnectiveXspec}
Fix $k\in \mathbb{Z}$, a connected space $X$, and a point $x\in X$.
The full subcategory $\mathrm{Sp}_X^{\geq k}$ spanned by $k$-connective $X$-spectra is generated by $x_! \Sigma^k S$ under colimits.
\end{proposition}
\begin{proof}[Sketch of proof]
Write $\langle x_! \Sigma^k S\rangle$ denote the full subcategory of $\mathrm{Sp}_X$ generated by $x_! \Sigma^k S$ under colimits.
The fibre of $x_! \Sigma^k S$ at $x$ is $\Sigma^k \Sigma^\infty_+ \Omega X$ (Example \ref{exam:TypicalXSpec} (2)), which is certainly $k$-connective.
The functor $x^\ast \colon \mathrm{Sp}_X\to \mathrm{Sp}$ computing fibre spectra at $x$ preserves colimits, hence any $X$-spectrum in $\langle x_! \Sigma^k S\rangle$ is $k$-connective.

Conversely, let $P$ be any $X$-spectrum.
The subcategory $\langle x_! \Sigma^k S\rangle$ determines a class of objects of the triangulated category $Ho(\mathrm{Sp}_X)$ which is closed under sums and suspensions, and hence there is an $X$-spectrum $P_{\geq  k} \in \langle x_! \Sigma^k S\rangle$ and a map of $X$-spectra $P_{\geq  k}\to P$ such that
\[
Ho(\mathrm{Sp}_X)(x_! S^p, P_{\geq k} )\longrightarrow 
Ho(\mathrm{Sp}_X)(x_! S^p, P)
\]
is an isomorphism for all $p\geq k$.
Since $x_! S^p$ corepresents the functor $P\mapsto \spi_p (x^\ast P)$, $P_{\geq k} \to P$ is an equivalence if $P$ is $k$-connective.
\end{proof}
\begin{corollary}
\label{cor:ConnCover}
For a connected base space $X$ and any $k\in \mathbb{Z}$, the subcategory of $k$-connective $X$-spectra is coreflective
\[
\begin{tikzcd}
\mathrm{Sp}_X^{\geq k}
\ar[rr, hookrightarrow, shift left =2]
\ar[rr, leftarrow, shift left=-2, "\bot", "\mathrm{cn}_k^X"']
&&
\mathrm{Sp}_X
\end{tikzcd}
\]
with coreflector $\mathrm{cn}_k^X\colon P\mapsto P_{\geq k}$.
The counit $P_{\geq k}\to P$ induces an isomorphism on fibrewise stable homotopy groups in dimensions $\geq k$.
\end{corollary}
\begin{proof}
This is an immediate consequence of the Proposition and the adjoint functor theorem for locally presentable $\infty$-categories.
\end{proof}
\begin{remark}
For integers $k > l$, the counit $P_{\geq k }\to P$ factors as $P_{\geq k }\to P_{\geq l} \to P$.
\end{remark}
\begin{corollary}
\label{cor:PfwdConnectivity}
For any map of connected spaces $f\colon X\to Y$ and $k\in \mathbb{Z}$, the pushforward $f_!\colon \mathrm{Sp}_X\to \mathrm{Sp}_Y$ sends $k$-connective $X$-spectra to $k$-connective $Y$-spectra.
\end{corollary}
\begin{proof}
Choosing $x\colon \ast \to X$, the pushforward of $x_! \Sigma^k S$ along $f$ is equivalent to $f(x)_! \Sigma^k S$.
Since $f_!$ preserves colimits, 
the claim follows from the Proposition.
\end{proof}

Now that we are equipped with functorial connective covers, we are in a position to define Postnikov sections of parametrised spectra.
For a connected space $X$ and $k\in \mathbb{Z}$, the \emph{$k$-th Postnikov section} of an $X$-spectrum $P$ is the cofibre 
\[
P_{\leq k}:=\mathrm{cofib} (P_{\geq (k+1)}\to P)\,.
\]
The map of $X$-spectra $P\to P_{\leq k}$ is readily seen to induce isomorphisms on fibrewise stable homotopy groups in dimensions $\leq k$; the fibrewise stable homotopy groups of $P_{\leq k}$ are trivial in dimensions $>k$.
We have thus essentially proven the following
\begin{proposition}
For a connected base space $X$ and any integer $k$, the subcategory of $(k+1)$-coconnective $X$-spectra is reflective
\[
\begin{tikzcd}
\mathrm{Sp}^{\leq k}_X
\ar[rr, hookrightarrow, shift left=-2, "\bot"]
\ar[rr, leftarrow, shift left =2, "\tau_{\leq k}^X"]
&&
\mathrm{Sp}_X
\end{tikzcd}
\] 
with reflector $\tau^X_{\leq k}\colon P\mapsto P_{\leq k}$.
\end{proposition}
\begin{proposition}
\label{prop:PostnikovPB}
For a map of connected spaces $f\colon X\to Y$, the pullback $f^\ast\colon \mathrm{Sp}_Y\to \mathrm{Sp}_X$ commutes with forming Postnikov sections; that is, for all $k\in \mathbb{Z}$ there is a natural equivalence $f^\ast \circ \tau^Y_{\leq k}\cong \tau_{\leq k}^X\circ f^\ast$.
\end{proposition}
\begin{proof}
For each $x\colon \ast \to X$ and any $A\in \mathrm{Sp}_Y$ there is an equivalence of spectra $x^\ast f^\ast A \cong f(x)^\ast A$.
It follows that $f^\ast \tau^Y_{\leq k}  A$ is a $(k+1)$-coconnective $X$-spectrum and hence, by the above result, that the counit $f^\ast A \to \tau^X_{\leq k} f^\ast A$ naturally factors as
\[
f^\ast 
A
\longrightarrow
f^\ast \tau^Y_{\leq k}A
\longrightarrow
\tau^X_{\leq k} A\,.
\]
The latter map induces an equivalence of fibrewise stable homotopy groups at all points of $X$, so is an equivalence.
\end{proof}

For each integer $k$, consider the  commuting diagram of pulllback-pushout squares in $\mathrm{Sp}_X$
\begin{equation}
\label{eqn:SliceSquares}
\begin{tikzcd}
P_{\geq (k+1)} 
\ar[r]
\ar[d]
&
0_X
\ar[d]
\\
P_{\geq k}
\ar[r]
\ar[d]
&
P_{=k}
\ar[d]
\ar[r]
&
0_X
\ar[d]
\\
P
\ar[r]
&
P_{\leq k}
\ar[r]
&
P_{\leq (k-1)}
\end{tikzcd}
\end{equation}
The $X$-spectrum $P_{=k}$ is called the \emph{$k$-th slice} of $P$.
The fibrewise stable homotopy groups of $P_{=k}$ are concentrated in dimension $k$, where they coincide with those of $P$.
Pullback functors are exact, so as an immediate consequence of Proposition \ref{prop:PostnikovPB} we have the following
\begin{corollary}
\label{cor:SliceFibre}
For a map of connected spaces $f\colon X\to Y$, the pullback $f^\ast\colon \mathrm{Sp}_Y\to \mathrm{Sp}_X$ commutes with forming slices; that is, for all $k\in \mathbb{Z}$ and $A\in \mathrm{Sp}_Y$ there is a natural equivalence of $X$-spectra $(f^\ast A)_{=k} \cong f^\ast (A_{=k})$.
\end{corollary}

\subsection{Parametrised sequential spectra}
We often need to lay hands on good point-set models in order to build examples, carry out explicit calculations, and illuminate certain properties.
There are by now a fair few modern approaches to parametrised stable homotopy theory---for instance the comprehensive \cite{may_parametrized_2006} and the more recent \cite{ando_parametrized_2018, hebestreit_multiplicative_2020, malkiewich_parametrized_2019}, to name a few.
We shall use the sequential parametrised spectra of \cite{braunack-mayer_combinatorial_2020}.
Sequential parametrised spectra provide a bare-bones combinatorial approach to parametrised stable homotopy theory; while many features of the theory (particularly fibrewise smash products) are obscured, this minimality is particularly well-suited for studying the rationalisation of parametrised stable homotopy theory.

Sequential parametrised spectra are based on simplicial objects.
For any simplicial set $X$, there is a combinatorial model category of \emph{retractive spaces} over $X$
\[
\mathrm{sSet}_{\dslash X} = (\mathrm{sSet}_{/X})^{\mathrm{id}_X/}
\cong (\mathrm{sSet}^{X/})_{/\mathrm{id}_X}
\] 
whose objects are commuting diagrams of simplicial sets
\[
\begin{tikzcd}
X
\ar[rr, bend left =-20, "\mathrm{id}_X"']
\ar[r]
&
Z
\ar[r]
&
X
\end{tikzcd}
\]
Objects of $\mathrm{sSet}_{\dslash X}$ are often written $Z$, with a particular choice of structure maps exhibiting $X$ as a retract to be understood from context.
A morphism in $\mathrm{sSet}_{\dslash X}$ is a map of simplicial sets $\psi\colon Z\to Z'$ commuting with the structure maps.
The model structure on $\mathrm{sSet}_{\dslash X}$ is detected by the forgetful functor $\mathrm{sSet}_{\dslash X}\to \mathrm{sSet}$, so that that a morphism $\psi\colon Z\to Z'$ of retractive spaces is a weak equivalence, cofibration, or fibration precisely if the underlying map is so in the Kan model structure on simplicial sets.

The category $\mathrm{sSet}_{\dslash X}$ is enriched, tensored, and cotensored over pointed simplicial sets.
The tensor of $K\in \mathrm{sSet}_\ast$ and $Z\in\mathrm{sSet}_{\dslash X}$ is the colimit (computed in $\mathrm{sSet}$)
\[
K\owedge_X Z :=
\mathrm{colim}
\left\{
\begin{tikzcd}[sep=small]
X
\ar[r]
\ar[d]
&
X\times K
\ar[dd]
\ar[dr]
\\
Z
\ar[dr]
\ar[rr, crossing over]
&&
Z\times_X (X\times K)
\\
&
X
\end{tikzcd}
\right\}
\]
equipped with the evident structure maps as a retractive space over $X$.
One shows that $\mathrm{sSet}_{\dslash X}$ is a $\mathrm{sSet}_\ast$-model category, so that in particular the suspension in $Ho(\mathrm{sSet}_{\dslash X})$ is modelled by tensoring cofibrant objects with the simplicial circle $S^1 = \Delta^1/\partial \Delta^1$.
This operation gives rise to a functor $\Sigma_X \colon Z\mapsto S^1\owedge_X Z$.

\begin{remark}
\label{rem:SuspensionPresentation}
The category $\mathrm{sSet}_{\dslash X}$ also admits tensors by unpointed simplicial sets. 
For a simplicial set $K$ and $Z\in \mathrm{sSet}_{\dslash X}$ we set $K\otimes_X Z := K_+\owedge_X Z$.
It is straightforward to check that $K\otimes_X Z$ coincides with the pushout of simplicial sets
\begin{equation}
\label{eqn:TensoringinUndercat}
\begin{tikzcd}
K\times X
\ar[r]
\ar[d]
&
X
\ar[d]
\\
K\times Z
\ar[r]
&
K\otimes_X Z\,,
\end{tikzcd}
\end{equation}
which is regarded as a retractive space over $X$ via the evident structure maps.
Note that since $S^1 = \Delta^1_+ /\partial \Delta^1_+$, the cofibre  of the map $\partial\Delta^1 \otimes_X Z \to \Delta^1 \otimes_X Z$ computes the suspension.
\end{remark}

\begin{remark}
\label{rem:ModelPresentsoo}
Any simplicial model category $\mathcal{M}$ gives rise to an $\infty$-category $\mathcal{M}^\infty$ by (i) passing to the full subcategory $\mathcal{M}^\circ\hookrightarrow \mathcal{M}$ on cofibrant-fibrant objects and (ii) applying the homotopy coherent nerve functor $N_\Delta\colon \mathrm{sSetCat}\to \mathrm{sSet}$ to obtain $\mathcal{M}^\infty = N_\Delta \mathcal{M}^\circ$.
In the case of $\mathrm{sSet}_{\dslash X}$, the $\infty$-category $\mathrm{sSet}_{\dslash X}^\infty$ produced by this recipe is equivalent to $\mathcal{S}_{\dslash X}$ mentioned in the previous section.
\end{remark}

For any map of simplicial sets $f\colon X\to Y$ there is a triple of adjoint functors
\begin{equation}
\label{eqn:BChangeRetSp}
(f_!\dashv f^\ast\dashv f_\ast)\colon
\begin{tikzcd}
\mathrm{sSet}_{\dslash X}
\ar[rr, leftarrow, "\bot"]
\ar[rr, shift left =4]
\ar[rr, shift left =-4, "\bot"]
&&
\mathrm{sSet}_{\dslash Y}
\end{tikzcd}
\end{equation}
in which $f_!$ and $f^\ast$ are respectively obtained by forming pushouts and pullbacks along $f$.
Both $f_!$ and $f^\ast$ preserve $\mathrm{sSet}_\ast$-tensors, and $(f_!\dashv f^\ast)$ is a $\mathrm{sSet}_\ast$-Quillen adjunction which is a Quillen equivalence whenever $f$ is a weak equivalence.
\begin{remark}
Since pullbacks of simplicial sets generally fail to preserve weak equivalences, the adjunction $(f^\ast\dashv f_\ast)\colon \mathrm{sSet}_{\dslash Y}\to \mathrm{sSet}_{\dslash X}$ is not generally Quillen.
However, if $f\colon X\to Y$ is a fibration $(f^\ast\dashv f_\ast)$ is a $\mathrm{sSet}_\ast$-Quillen adjunction, which is additionally a Quillen equivalence if $f$ acyclic.
\end{remark}

\begin{definition}
For a simplicial set $X$, a \emph{sequential $X$-spectrum} is a sequence of retractive spaces $A_0, A_1, A_2, \dots$ over $X$ together with maps $\sigma_n\colon \Sigma_X A_n \to A_{n+1}$ in $\mathrm{sSet}_{\dslash X}$ for each $n\geq 0$; these data are often abbreviated to $A$, with the underlying sequence of retractive spaces and structure maps to be understood.
A morphism of sequential $X$-spectra $\psi\colon A\to B$ is a sequence of maps of retractive spaces $\psi_n\colon A_n\to B_n$ commuting with the structure maps.
Sequential $X$-spectra organise into a category, which is denoted $\mathrm{Sp}^\mathbb{N}_X$.
\end{definition}

For each $k\geq 0$ there is an adjunction 
\begin{equation}
\label{eqn:StabAdj}
\begin{tikzcd}
\mathrm{sSet}_{\dslash X}
\ar[rr, shift left =2, "\Sigma^{\infty-k}_X"]
\ar[rr, shift left=-2, leftarrow, "\widetilde{\Omega}^{\infty-k}_X"', "\bot"]
&&
\mathrm{Sp}_X^\mathbb{N}
\end{tikzcd}
\end{equation}
where $\widetilde{\Omega}^{\infty-k}_X\colon A\mapsto A_k$ extracts the $k$-th retractive space of a sequential $X$-spectrum and $\Sigma^{\infty-k}_X$ sends the retractive space $Z$ to the sequential $X$-spectrum freely generated by $Z$ in degree $k$
\[
(\Sigma^{\infty-k}Z)_{n}
=
\begin{cases}
0_X & n<k\\
\Sigma^{n-k}_X Z & n\geq k
\end{cases}
\]
equipped with the evident structure maps.
In the above expression, $0_X= (X\to X\to X)$ is the zero object of $\mathrm{sSet}_{\dslash X}$.
Taken together, for $k\geq 0$ the adjunctions \eqref{eqn:StabAdj} can be used to equip $\mathrm{Sp}^\mathbb{N}_X$ with a \emph{projective model structure}, for which the fibrations and weak equivalences are jointly detected by the functors  $\widetilde{\Omega}^{\infty-k}_X$.

\begin{remark}
\label{rem:SeqSpecCofibGen}
The projective model structure on $\mathrm{Sp}^\mathbb{N}_X$ is cofibrantly generated.
We describe sets of generating cofibrations and generating acyclic cofibrations (cf.~\cite[Section 2.1.1]{braunack-mayer_combinatorial_2020}).
For a simplex $\sigma\colon \Delta^n \to X$ there are morphisms of retractive spaces over $X$
\[
i_n(\sigma_{X+})=
\left\{
\begin{tikzcd}[sep=small]
&X
\ar[dr]
\ar[dl]
&
\\
X\displaystyle{\coprod}  \partial\Delta^n 
\ar[dr, "\mathrm{id}_X + \sigma|_{\partial\Delta^n}"']
\ar[rr]
&&
X\displaystyle{\coprod}\Delta^n
\ar[dl, "\mathrm{id}_X + \sigma"]
\\
&X&
\end{tikzcd}
\right\}
\;\;
\mbox{ and }
\;\;
h^n_k(\sigma_{X+})=
\left\{
\begin{tikzcd}[sep=small]
&X
\ar[dr]
\ar[dl]
&
\\
X\displaystyle{\coprod}  \Lambda^n_k 
\ar[dr, "\mathrm{id}_X + \sigma|_{\Lambda^n_k}"']
\ar[rr]
&&
X\displaystyle{\coprod}\Delta^n
\ar[dl, "\mathrm{id}_X + \sigma"]
\\
&X&
\end{tikzcd}
\right\}
\]
induced by the inclusions of the boundaries $i_n\colon \partial\Delta^n \hookrightarrow \Delta^n$ and horns $h^n_k\colon \Lambda^n_k\hookrightarrow \Delta^n$ respectively.
A set of generating cofibrations of $\mathrm{Sp}^\mathbb{N}_X$ is 
\[
\mathcal{I}^\mathbb{N}_X =\big\{  \Sigma^{\infty-k}_X i_n(\sigma_{X+}) \,\big|\,  \sigma\colon \Delta^n \to X\,,\; n,k\geq 0 \big\}
\]
obtained by applying the shifted fibrewise suspension functors to the $i_n (\sigma_{X+})$.
A set of generating acyclic cofibrations of $\mathrm{Sp}^\mathbb{N}_X$ is 
\[
\mathcal{J}^\mathbb{N}_X =\big\{  \Sigma^{\infty-k}_X h^n_l(\sigma_{X+}) \,\big|\,  \sigma\colon \Delta^n \to X\,,\; 0\leq l\leq n\,,\; n,k\geq 0  \big\}
\]
obtained by applying shifted fibrewise suspension functors to horn inclusions of simplices in $X$.
\end{remark}

The \emph{stable model structure} on $\mathrm{Sp}^\mathbb{N}_X$ is obtained as the left Bousfield localisation of the projective model structure at the set of morphisms
\begin{equation}
\label{eqn:LocClass}
S_{\mathbb{N}, X}:=\big\{
\zeta_{k,X}(C)\colon 
\Sigma^{\infty-(k+1)}_X (\Sigma_X C)\to \Sigma^{\infty-k}_X(C)
\big\}\,,
\end{equation}
where $\zeta_{k,X}(Z)$ is the adjunct of the identity map $\Sigma_X Z \to \Sigma_X Z = \widetilde{\Omega}^{\infty-(k+1)}_X\Sigma^k_X Z$, $C$ ranges over retractive spaces over $X$ of the form $X\coprod\partial\Delta^n$ or $X\coprod\Delta^n$ (as in Remark \ref{rem:SeqSpecCofibGen}), and $k\geq 0$. 
The stable model structure has the following key properties (established in \cite[Section 2.1.1]{braunack-mayer_combinatorial_2020}):
\begin{itemize}
  \item The stable model structure on $\mathrm{Sp}^\mathbb{N}_X$ is a left proper combinatorial $\mathrm{sSet}_\ast$-model structure. 
  The suspension and looping auto-equivalences of $Ho(\mathrm{Sp}^\mathbb{N}_X)$ are modelled by the functors $\Sigma_X$ and $\Omega_X$ that respectively compute $\mathrm{sSet}_\ast$-tensors and cotensors with the simplicial circle $S^1$.
  The adjunction $(\Sigma_X\dashv \Omega_X)$ is a $\mathrm{sSet}_\ast$-Quillen auto-equivalence of the stable model structure on $\mathrm{Sp}^\mathbb{N}_X$, which is therefore a stable model category.
  
  \item The fibrant objects are precisely the \emph{fibrant $\Omega_X$-spectra}; that is, the sequential $X$-spectra $A$ for which (i) each $A_n$ is fibrant in $\mathrm{sSet}_{\dslash X}$ and (ii) the adjoint structure maps $\sigma_n^\vee\colon A_n \to \Omega_X A_{n+1}$ are weak equivalences for all $n\geq 0$.
  The stable weak equivalences between fibrant $\Omega_X$-spectra are precisely the levelwise weak equivalences.
  
  \item The stable model structure presents an $\infty$-category $(\mathrm{Sp}^\mathbb{N}_X)^\infty$ (cf.~Remark \ref{rem:ModelPresentsoo}) equivalent to the $\infty$-category $\mathrm{Sp}_X$ of $X$-parametrised spectra discussed previously.
\end{itemize}
For any map of simplicial sets $f\colon X\to Y$, the functors of \eqref{eqn:BChangeRetSp} prolong to functors of sequential parametrised spectra by levelwise application.
\begin{proposition}
For any map of simplicial sets $f\colon X\to Y$ there is a triple of adjoint functors 
\begin{equation}
\label{eqn:BChangeSeqSpec}
(f_!\dashv f^\ast\dashv f_\ast)\colon
\begin{tikzcd}
\mathrm{Sp}^\mathbb{N}_X
\ar[rr, leftarrow, "\bot"]
\ar[rr, shift left =4]
\ar[rr, shift left =-4, "\bot"]
&&
\mathrm{Sp}^\mathbb{N}_Y
\end{tikzcd}
\end{equation}
such that
$(f_!\dashv f^\ast)$ is a $\mathrm{sSet}_\ast$-Quillen adjunction for the stable model structures.
In addition:
\begin{itemize}
  \item If $f$ is a weak equivalence then $(f_!\dashv f^\ast)$ is a $\mathrm{sSet}_\ast$-Quillen equivalence.
  \item If $f$ is a fibration then $(f^\ast\dashv f_\ast)$ is a $\mathrm{sSet}_\ast$-Quillen adjunction for the stable model structures, which is moreover a Quillen equivalence if $f$ is acyclic.
\end{itemize}
\end{proposition}

\begin{remark}
A $\mathrm{sSet}$-Quillen adjunction of simplicial model categories $(L\dashv R)\colon \mathcal{M}\to \mathcal{N}$ induces an adjunction of $\infty$-categories $(L\dashv R)\colon \mathcal{M}^\infty\to \mathcal{N}^\infty$.
In this manner the $\mathrm{sSet}_\ast$-Quillen adjunctions of \eqref{eqn:BChangeSeqSpec} are identified with the adjunctions \eqref{eqn:BChangeooCat}.
\end{remark}

For a reduced simplicial set $X$ we use Kan's simplicial monoid $\mathbb{G}X$ to model the homotopy type of $\Omega X$.
By means of the $\mathrm{sSet}_\ast$-tensoring on sequential spectra we form a model category of sequential $\mathbb{G}X_+$-module spectra; that is, of sequential spectra with a $\mathbb{G}X_+$-action.
The resulting model category, denoted $\mathbb{G}X_+\mathrm{-Mod}^\mathbb{N}$, is isomorphic to the model category of sequential spectra in pointed $\mathbb{G}X$-spaces.
\begin{proposition}
\label{prop:SeqSpecFibvsModule}
For any reduced simplicial set $X$, there is a $\mathrm{sSet}_\ast$-Quillen equivalence
\[
\begin{tikzcd}
\mathbb{G}X_+\mathrm{-Mod}^\mathbb{N}
\ar[rr, shift left =2, "\mathfrak{b}_X"]
\ar[rr, shift left=-2, leftarrow, "\bot", "\mathfrak{s}_X"']
&&
\mathrm{Sp}^\mathbb{N}_{X}
\end{tikzcd}
\]
\end{proposition}
\begin{proof}[Sketch of proof.]
The adjunction in question is found in \cite[Remark 2.13]{braunack-mayer_combinatorial_2020}.
The proof that this adjunction is a Quillen equivalence is sketched in \emph{loc.~cit.}, alternatively apply the sequential stabilisation machine to \cite[Lemma 1.31]{braunack-mayer_combinatorial_2020}.
\end{proof}
\begin{remark}
\label{rem:SeqSpecFibvsModule}
The functor $\mathfrak{b}_X$ sends a $\mathbb{G}X_+$-module spectrum $E$ to the $X$-spectrum $x_! (E/\mathbb{G}X)$, with $x\colon \ast \to X$ the unique $0$-simplex.
The Quillen equivalence of Proposition \ref{prop:SeqSpecFibvsModule} thus identifies the pushforward functor $x_!\colon \mathrm{Sp}^\mathbb{N}\to \mathrm{Sp}^\mathbb{N}_X$ with the free functor $P\mapsto \mathbb{G}X_+\wedge P$.
In addition, for any fibrant $\Omega_X$-spectrum $A$, $\mathfrak{s}_X A$ has underlying spectrum  equivalent to $x^\ast A$, the fibre spectrum of $A$ at $x\in X$.
\end{remark}

The next result uses the combinatorial models to characterise fibrewise Eilenberg--Mac Lane spectra.
\begin{proposition}
\label{prop:HeartofSpX}
Let $X$ be a connected space with $x\in X$.
For any $k\in \mathbb{Z}$, the functor
\begin{align*}
\mathrm{Sp}_X
&\longrightarrow \mathbb{Z}[\pi_1 X]\mathrm{-Mod}
\\
P
&\longmapsto
\spi_k (x^\ast P)
\end{align*}
restricts to an equivalence on the full subcategory of $X$-spectra with fibrewise stable homotopy concentrated in dimension $k$. 
\end{proposition}
\begin{proof}
By applying fibrewise suspension or looping functors as necessary, we may suppose that $k=0$ without loss of generality.
We may also take $X$ to be a reduced simplicial set with unique $0$-simplex $x$, in which case Proposition \ref{prop:SeqSpecFibvsModule} shows that the $\infty$-category $\mathrm{Sp}_X$ is presented by the model category $\mathbb{G}X_+\mathrm{-Mod}^\mathbb{N}$ of sequential $\mathbb{G}X_+$-module spectra.
By Remark \ref{rem:SeqSpecFibvsModule}, the functor $P\mapsto \spi_n (x^\ast P)$ on $X$-spectra is modelled in terms of $\mathbb{G}X_+$-module spectra simply by computing $n$-th stable homotopy groups.

Let $E$ be a sequential $\mathbb{G}X_+$-module spectrum corresponding to the $X$-spectrum $P$, taking $E$ to be cofibrant-fibrant without loss of generality.
The structure map
\[
\rho_E\colon \mathbb{G}X_+\wedge E\longrightarrow E
\]
furnishes $\pi_0^\mathrm{st}E$ with the structure of an $\mathbb{Z}[\pi_1 X]$-module via the action map 
\[
\begin{tikzcd}
\mathbb{Z}[\pi_0\mathbb{G}X] \otimes \spi_0 E
\ar[r, "\cong"]
&
\pi_0^\mathrm{st}\mathbb{G}X\otimes \spi_0 E
\ar[r]
&
\pi_0^\mathrm{st}(\mathbb{G}X_+\wedge E)
\ar[r, "\spi_0\rho_E"]
&
\pi_0^\mathrm{st} E\,,
\end{tikzcd}
\]
using $\pi_0 \mathbb{G}X \cong \pi_1 X$.
Upon taking zeroth stable homotopy groups, morphisms of $\mathbb{G}X_+$-module spectra give rise to morphisms of $\mathbb{Z}[\pi_1 X]$-modules and thus the functor $\mathrm{Sp}_X\to \mathbb{Z}\mathrm{-Mod}$ that sends $P\mapsto \spi_0 (x^\ast P)$ factors via the functor $\mathbb{Z}[\pi_1X]\mathrm{-Mod}\to \mathbb{Z}\mathrm{-Mod}$ induced by the canonical augmentation $\mathbb{Z}[\pi_1 X]\to \mathbb{Z}$.

Suppose that the $X$-spectra $A$ and $B$ are ($0$-)connective and ($0$-)coconnective respectively.
By the above discussion there is a homomorphism
\[
\begin{tikzcd}
\{A,B\}^0_X
\ar[r, "\varpi_{A,B}"]
&
\mathrm{Hom}_{\mathbb{Z}[\pi_1 X]} \big(\spi_0 (x^\ast A), \spi_0(x^\ast B)\big)
\end{tikzcd}
\]
which we claim is an isomorphism.
To see this, consider the full subcategory $\mathcal{E}\hookrightarrow Ho(\mathrm{Sp}_X)$ of connective $X$-spectra $A$ for which $\varpi_{A,B}$ is an isomorphism for all coconnective $X$-spectra $B$.
The subcategory $\mathcal{E}$ contains $x_! S$, for $\spi_0 (x^\ast x_! S) = \mathbb{Z}[\pi_1 X]$ and thus for any coconnective $X$-spectrum $B$, $\varpi_{x_!S, B}$ is the evident isomorphism
\[
\{x_! S,B\}^0_X \cong \spi_0 (x^\ast B)
\longrightarrow
\mathrm{Hom}_{\mathbb{Z}[\pi_1 X]}\big(\mathbb{Z}[\pi_1 X], \spi_0 (x^\ast B)\big)\,.
\]
The category $\mathcal{E}$ is readily seen to be closed under forming sums and cofibre sequences, hence under all colimits.
By Proposition \ref{prop:kconnectiveXspec} we have that $\mathcal{E}$ is the full subcategory spanned by connective $X$-spectra.

An easy consequence of the above is that the assignment $P\mapsto \spi_0 (x^\ast P)$ becomes fully faithful when restricted to the homotopy category of $X$-spectra with fibrewise stable homotopy concentrated in dimension zero.
To show that this functor is also essentially surjective, we once more use objects of $\mathbb{G}X_+\mathrm{-Mod}^\mathbb{N}$ to model $X$-spectra.
For a $\mathbb{Z}[\pi_1X]$-module $M$, first form the sequential Eilenberg--Mac Lane spectrum $HM$ whose $n$-th space $HM_n = M[S^n]$ is the reduced $M$-linearisation of the simplicial $n$-sphere $S^n$, equipped with sequential spectrum structure maps 
\begin{align*}
S^1 \wedge M[S^n] &\longrightarrow M[S^{n+1}]
\\
x \wedge \Big(\sum_{i} m_i y_i\Big)
&
\longmapsto
\sum_i m_i (x\wedge y_i)\,.
\end{align*}
We use the $\mathbb{Z}[\pi_1 X]$-module structure of $M$ to equip each space $M[S^n]$ with $\mathbb{G}X_+$-structure maps
\begin{align*}
\mathbb{G}X_+\wedge M[S^n] &\longrightarrow M[S^n]
\\
\gamma \wedge \Big(\sum_{i} m_i y_i\Big)
&
\longmapsto
\sum_i ([\gamma]\cdot m_i) y_i\,,
\end{align*}
where $[\gamma]$ is the class of $\gamma$ in $\pi_0\mathbb{G}X\cong \pi_1 X$.
These $\mathbb{G}X_+$-structure maps are clearly compatible with the suspension spectrum structure maps, so that $HM$ is a (fibrant) sequential $\mathbb{G}X_+$-module spectrum.
By fibrancy, in order to compute zeroth stable homotopy groups of $HM$ it is sufficient to take unstable $\pi_0$ in dimension zero, where we recover $M$ with its given structure as a $\mathbb{Z}[\pi_1 X]$-module.
\end{proof}

\begin{remark}
\label{rem:NaturalityHeart}
Let $f\colon X\to Y$ be a map of connected spaces.
Under the equivalence of Proposition \ref{prop:HeartofSpX}, the restriction of $f^\ast\colon  \mathrm{Sp}_Y\to \mathrm{Sp}_X$ to full subcategories of parametrised spectra with fibrewise stable homotopy concentrated in dimension zero is identified with the base change functor
\[
\mathbb{Z}[\pi_1Y]\mathrm{-Mod}\longrightarrow \mathbb{Z}[\pi_1 X]\mathrm{-Mod}
\]
induced by $\pi_1 f \colon \pi_1 X \to \pi_1 Y$.
In the case that $Y =\ast$, this implies that any $X$-parametrised spectrum $P$ with fibrewise stable homotopy concentrated in degree zero for which $\pi_1 X$ acts trivially on $\spi_0 x^\ast P$ must be stably equivalent to $X^\ast HA$ for some abelian group $A$.
\end{remark}

\subsection{Nilpotent parametrised spectra}
\label{SS:NilParamSpec}
For connected $X$, we will call an $X$-parametrised spectrum $P$ \emph{nilpotent} if $\pi_1 X$ acts nilpotently on each fibrewise stable homotopy group $\spi_k( x^\ast P)$. 
This condition turns out to be equivalent to requiring that each stage Postnikov tower of $P$ factors as a finite sequence of extensions by untwisted Eilenberg--Mac Lane spectra.

Recall that an action $\alpha\colon \pi\to \mathrm{Aut}(G)$ of $\pi$ on a group $G$ is \emph{nilpotent} if there is a finite sequence of subgroups of $G$
\[
G = G_0 \geq G_1\geq \dotsb \geq  G_j \geq  \dotsb \geq  G_q = \ast
\]
such that for each $j$:
\begin{enumerate}[label=(\roman*)]
  \item $G_j$ is closed under the action $\alpha$,
  \item $G_{j+1}\trianglelefteq G_j$ is normal with abelian quotient, and 
  \item The induced action of $\pi$ on $G_j /G_{j+1}$ is trivial.
\end{enumerate}
We are frequently concerned with the case that $G = A$ is an abelian group, i.e.~a $\mathbb{Z}[\pi]$-module.
We shall need the following easy lemma:
\begin{lemma}
\label{lem:NilExactSeq}
If $\pi$ acts on a short exact sequence of groups 
\[
\begin{tikzcd}
\ast
\ar[r]
&
G'
\ar[r]&
G
\ar[r]
&
G''
\ar[r]&
\ast
\end{tikzcd}
\]
then the action of $\pi$ on $G$ is nilpotent if and only if the actions of $\pi$ on $G'$ and $G''$ are nilpotent.
\end{lemma}
\begin{remark}
A fibration with $p\colon E\to B$ with connected fibre $F$ is \emph{nilpotent} if the action of $\pi_1 E$ on each $\pi_k F$ is nilpotent.
A fibration $p\colon E\to B$ is nilpotent if and only if every stage of the Moore--Postnikov tower of $p$ can be factored as the composite of finitely many principal fibrations with fibres $K(\pi, n)$'s (see \cite[Section III.5]{bousfield_limits_1972}).

A connected space $X$ is nilpotent if the terminal map $X\to \ast$ is nilpotent; equivalently, if $\pi_1 X$ is a nilpotent group and the $\pi_1 X$-action on $\pi_n X$ is nilpotent for all $n\geq 2$.
\end{remark}

\begin{definition}
\label{defn:Nilpotent}
Let $X$ be a connected space with $x\in X$.
An $X$-parametrised spectrum $P$ with fibre spectrum $x^\ast P$ at $x$ is called \emph{nilpotent} if the action of $\pi_1 X$ on each $\spi_k (x^\ast P)$ is nilpotent.
We write $Ho(\mathrm{Sp}_X)_\mathrm{nil}\hookrightarrow Ho(\mathrm{Sp}_X)$ for the full subcategory on nilpotent parametrised spectra.
\end{definition}

\begin{example}
Any $X$-spectrum in the image of the pullback functor $X^\ast \colon Ho(\mathrm{Sp})\to Ho(\mathrm{Sp}_X)$ is nilpotent since $\Omega X_+$ acts trivially on the fibre spectrum (cf.~Remark \ref{rem:NaturalityHeart}).
\end{example}

\begin{proposition}
\label{prop:NilXSpecSeq}
For a connected space $X$ and fibre sequence of $X$-parametrised spectra $A\to B\to C$,
if any two of $A$, $B$, or $C$ is nilpotent then so is the third.
\end{proposition}
\begin{proof}
Taking fibre spectra yields a fibre sequence $x^\ast A \to x^\ast B \to x^\ast C$, which in turn induces a long exact sequence of stable homotopy groups
\[
\begin{tikzcd}
\dotsb
\ar[r]
&
\spi_{k+1} x^\ast C
\ar[r]
&
\spi_k x^\ast A
\ar[r]
&
\spi_k x^\ast B
\ar[r]
&
\spi_k x^\ast C
\ar[r]
&
\spi_{k-1} x^\ast A
\ar[r]
&
\dotsb
\end{tikzcd}
\]
in which each arrow is a homomorphism of $\mathbb{Z}[\pi_1 X]$-modules.
Using Lemma \ref{lem:NilExactSeq}, it follows that if any two of $\spi_\ast x^\ast A$, $\spi_\ast x^\ast B$, and $\spi_\ast x^\ast C$ are nilpotent $\mathbb{Z}[\pi_1 X]$-modules then so is the third.
\end{proof}

\begin{corollary}
For a connected space $X$, 
$Ho(\mathrm{Sp}_X)_\mathrm{nil}\hookrightarrow Ho(\mathrm{Sp}_X)$
is a triangulated subcategory.
\end{corollary}

We shall require the following useful characterisation of nilpotent $X$-spectra:
\begin{theorem}
\label{thm:Nil}
An $X$-spectrum $P$ is nilpotent if and only if for all $k$, the zero morphism $P_{=k}\to 0_X$ factors as a finite sequence of extensions by $X^\ast (\Sigma^k HA)$'s. 
\end{theorem}
\begin{proof}
For the \lq\lq if'' direction, suppose that $P$ is an $X$-spectrum such that for each integer $k$, the map $P_{=k}\to 0_X$ factors as a finite sequence of extensions by $X^\ast (\Sigma^k HA)$'s.
Note that fibre spectrum $x^\ast {P_{=k}}$ of $P_{=k}$ is equivalent to $(x^\ast {P})_{=k}$ (Corollary \ref{cor:SliceFibre}) and by hypothesis there is a diagram of $X$-spectra
\[
\begin{tikzcd}[row sep=small, column sep =tiny]
P_{=k} = (P_{=k})_{n_k}
\ar[r]
&
(P_{=k})_{n_k-1}
\ar[r]
&
\dotsb
\ar[r]
&
(P_{=k})_{j}
\ar[r]
&
(P_{=k})_{j-1}
\ar[r]
&
\dotsb 
\ar[r]
&
(P_{=k})_{1}
\ar[r]
&
0_X
\\
X^\ast \Sigma^kHA_{n_k}
\ar[u]
&
X^\ast \Sigma^kHA_{n_k-1}
\ar[u]
&
\dotsb
&
X^\ast \Sigma^kHA_{j}
\ar[u]
&
X^\ast \Sigma^kHA_{j-1}
\ar[u]
&
\dotsb 
&
X^\ast \Sigma^kHA_{1}
\ar[u]
\end{tikzcd}
\]
in which each $X^\ast \Sigma^k HA_j \to (P_{=k})_j\to (P_{=k})_{j-1}$ is a fibre sequence.
As each $X^\ast \Sigma^k HA_j$ is nilpotent, Proposition \ref{prop:NilXSpecSeq} implies that $P_{=k}$ is also. 
Each of the zig-zags $P\leftarrow P_{\geq k}\to P_{=k}$ induces an isomorphism of $\mathbb{Z}[\pi_1 X]$-modules $\spi_k x^\ast P \cong \spi_k x^\ast P_{=k}$ and hence $P$ is nilpotent.

Conversely, suppose $P$ is a nilpotent $X$-spectrum. 
Fixing an integer $k$, there is a filtration of $\spi_k x^\ast {P}$ by subgroups
\[
\label{eqn:PostStageHoFilt}
\tag{$\ast$}
\spi_k x^\ast {P} = \Gamma_0\geq
\Gamma_1\geq \dotsb \geq \Gamma_j \geq \dotsb \geq \Gamma_{n_k} = 0
\]
such that each $\Gamma_j \to \Gamma_{j-1}$ is an injective homomorphism of $\mathbb{Z}[\pi_1 X]$-modules for which the $\mathbb{Z}[\pi_1 X]$-action on  $\Gamma_{j-1}/\Gamma_j$ is trivial (that is, the action factors through the augmentation $\mathbb{Z}[\pi_1 X]\to \mathbb{Z}$).
Recall that the adjoint equivalence of $\infty$-categories $(\mathfrak{b}_X\dashv \mathfrak{s}_X) \colon \Omega X_+\mathrm{-Mod}\to  \mathrm{Sp}_X$ of Propositions \ref{prop:FibModEquiv} and \ref{prop:SeqSpecFibvsModule} identifies $P$ with the stable fibre $x^\ast P$ equipped with the Serre action of $\Omega X_+$.
Since the fibre spectrum $x^\ast {P_{=k}}$ of $P_{=k}$ is equivalent to $(x^\ast P)_{=k}$, Proposition \ref{prop:HeartofSpX} implies that for each $0\leq j \leq n_k$ there is a morphism of $\Omega X_+$-module spectra $\Sigma^k H\Gamma_j \to \mathfrak{s}_X P_{=k}$ which corresponds to the morphism $\mathbb{Z}[\pi_1 X]$-modules $\Gamma_j \to \spi_k x^\ast P$ by taking stable homotopy groups.
Extending to a cofibre sequence $\Sigma^k H\Gamma_j \to \mathfrak{s}_X P_{=k}\to Q_j$, we obtain a sequence of $\Omega X_+$-module spectra
\[
\begin{tikzcd}
\mathfrak{s}_X P_{=k} \cong Q_{n_k}
\ar[r]
&
Q_{n_k -1}
\ar[r]
&
\dotsb
\ar[r]
&
Q_j
\ar[r]
&
\dotsb
\ar[r]
&
Q_0 \cong 0
\end{tikzcd}
\]
where $Q_j \to Q_{j-1}$ induces the morphism of $\mathbb{Z}[\pi_1 X]$-modules $(\spi_k x^\ast P)/\Gamma_j \to (\spi_k x^\ast P)/\Gamma_{j-1} $ on homotopy.
Taking fibres of the morphisms $Q_j \to Q_{j-1}$, we get  fibre sequences of $\Omega X_+$-module spectra $\Sigma^k H (\Gamma_{j-1}/\Gamma_j)\to Q_j \to Q_{j-1}$ for which the $\Sigma^k H (\Gamma_{j-1}/\Gamma_j)$ are trivial $\Omega X_+$-module spectra.
By Remark \ref{rem:NaturalityHeart} the adjoint equivalence  of $\infty$-categories $(\mathfrak{b}_X\dashv \mathfrak{s}_X)\colon \Omega X_+\mathrm{-Mod}\to \mathrm{Sp}_X$ identifies the trivial $\Omega X_+$-module spectra $\Sigma^k H (\Gamma_{j-1}/\Gamma_j)$ with $X^\ast \Sigma^k H (\Gamma_{j-1}/\Gamma_j)$ 
so that we get a diagram of $X$-spectra
\[
\begin{tikzcd}[row sep=small, column sep =tiny]
P_{=k} \cong \mathfrak{b}_X Q_{n_k}
\ar[r]
&
\mathfrak{b}_X Q_{n_k-1}
\ar[r]
&
\dotsb
\ar[r]
&
%\mathfrak{b}_X Q_{j}
%\ar[r]
%&
%\mathfrak{b}_X Q_{j-1}
%\ar[r]
%&
%\dotsb 
%\ar[r]
%&
\mathfrak{b}_X Q_{1}
\ar[r]
&
0_X
\\
X^\ast \Sigma^kH \Gamma_{n_k-1}
\ar[u]
&
X^\ast \Sigma^kH (\Gamma_{n_k-2}/\Gamma_{n_k-1})
\ar[u]
&
\dotsb
&
%X^\ast \Sigma^kH(\Gamma_{j-1}/\Gamma_{j}
%\ar[u]
%&
%X^\ast \Sigma^kH(\Gamma_{j-2}/\Gamma_{j-1})
%\ar[u]
%&
%\dotsb 
%&
X^\ast \Sigma^k H(\Gamma_{0}/\Gamma_1)
\ar[u]
\end{tikzcd}
\]
in which each $X^\ast \Sigma^k H (\Gamma_{j-1}/ \Gamma_j)\to \mathfrak{b}_X Q_j\to \mathfrak{b}_X Q_{j-1}$ is a fibre sequence.
\end{proof}

\begin{corollary}
\label{cor:NilSum}
The class of nilpotent $X$-spectra is closed under finite sums.
\end{corollary}
\begin{proof}
Suppose that $P^1,\dotsc, P^n$ are nilpotent $X$-spectra.
For each integer $k$ we have $(\bigoplus_{i=1}^n P^i)_{=k} \cong \bigoplus_{i=1}^n P^i_{=k}$.
By Theorem \ref{thm:Nil}, for each $1\leq i\leq n$ there is an integer $m_i$ and sequence of abelian groups $HA^i_1, \dotsc, HA^i_{m_i}$ such that for each $1\leq j\leq m_i$ there are fibre sequences of $X$-spectra 
\[
X^\ast \Sigma^k HA^i_{j}
\longrightarrow
(P^i_{=k})_{j}
\longrightarrow
(P^i_{=k})_{j-1}
\]
with $(P^i_{=k})_{m_i} = P^i_{=k}$ and $(P^i_{=k})_{0} = 0$.
Letting $m:=\max\{m_1, \dotsb, m_n\}$, for each $i$ and $m_i< j\leq m$ we set $HA^i_j := 0$ and $(P^i_{=k})_{j} := P^i_{=k}$.
Then for each $1\leq j\leq m$ we have a fibre sequence of $X$-spectra
\[
\bigoplus_{i=1}^n X^\ast \Sigma^k HA^i_{j}
\cong 
X^\ast \Sigma^k H \bigg(\bigoplus_{i=1}^n A^i_{j}\bigg)
\longrightarrow
\bigoplus_{i=1}^n  (P^i_{=k})_{j}
\longrightarrow
\bigoplus_{i=1}^n  (P^i_{=k})_{j-1}
\]
with $\bigoplus_{i=1}^n  (P^i_{=k})_{0} = 0$ and $\bigoplus_{i=1}^n  (P^i_{=k})_{m} = \bigoplus_{i=1}^n P^i_{=k}$.
Theorem \ref{thm:Nil} now shows that $\bigoplus_{i=1}^n P^i$ is nilpotent.
\end{proof}

\begin{corollary}
\label{cor:NilTrunc}
An $X$-spectrum $P$ is nilpotent if and only if $P_{\leq k}$ is nilpotent for all $k$.
\end{corollary}
\begin{proof}
For any $X$-spectrum $P$ and integer $k$, by \eqref{eqn:SliceSquares} the $k$-th slice $P_{=k}$ is the fibre of the map of Postnikov sections $P_{\leq k}\to P_{\leq (k-1)}$.
For any $l\leq k$ the $l$-th Postnikov sections of $P$ and $P_{\leq k}$ coincide: $(P_{\leq k})_{\leq l} \cong P_{\leq l}$.
Hence $P_{=l} \cong (P_{\leq k})_{=l}$ for all $l\leq k$ and the result follows by Theorem \ref{thm:Nil}.
\end{proof}

We conclude this section with a result that provides a natural class of examples of nilpotent parametrised spectra.
\begin{proposition}
\label{prop:StabNilFib}
For a connected space $X$, the fibrewise stabilisation of a nilpotent fibration $E\to X$ is a nilpotent parametrised spectrum.
\end{proposition}
\begin{proof}
Let $p\colon E\to B$ be a nilpotent fibration with fibre $F$, and let $E\to E_{\leq k} \to X$ be the $k$-th stage of the Moore--Postnikov tower of $p$. 
Define $E^{(k)}$ as the pushout
\[
\begin{tikzcd}
E
\ar[r]
\ar[d]
&
E_{\leq k}
\ar[d]
\\
X
\ar[r]
&
E^{(k)}\,,
\end{tikzcd}
\]
so that $p$ factors as $E\to E_{\leq k}\to E^{(k)}\to X$.
Write $F_{(\leq k)}$ and $F^{(k)}$ for the fibres of $E_{\leq k}\to X$ and $E^{(k)}\to X$ respectively, and observe that we have a cofibre sequence of $X$-parametrised spectra $\Sigma^\infty_{X+}E \to \Sigma^\infty_{X+} E_{\leq k} \to \Sigma^\infty_{X+}E^{(k)}$.

For a fibre sequence $Z\to Y\to X$, the fibrewise stabilisation $\Sigma^\infty_{X+} Y$ has fibre spectrum $\Sigma^\infty_+ Z$ with $\pi_1 X$ acting via the Serre action on $\spi_k Z$ (cf.~Example \ref{exam:FibrewiseSuspensionSpectra}).
Upon passing to fibres, the cofibre (hence fibre) sequence of $X$-parametrised spectra $\Sigma^\infty_{X+}E \to \Sigma^\infty_{X+} E_{\leq k} \to \Sigma^\infty_{X+}E^{(k)}$ induces a fibre sequence of spectra $\Sigma^\infty_+ F \to \Sigma^\infty_+ F_{(\leq k)}\to \Sigma^\infty_+ F^{(k)}$ and hence a long exact sequence of stable homotopy groups
\[
\begin{tikzcd}[sep=small]
\dotsb
\ar[r]
&
\spi_{i+1} F^{(k)}
\ar[r]
&
\spi_i F
\ar[r]
&
\spi_i F_{(\leq k)}
\ar[r]
&
\spi_i F^{(k)}
\ar[r]
&
\spi_{i-1} F
\ar[r]
&
\dotsb
\end{tikzcd}
\]
where each arrow is a homomorphism of $\mathbb{Z}[\pi_1 X]$-modules.
But $X \to E^{(k)}$ induces an isomorphism on $\pi_i$ for $i<k$, from which it follows that the fibre $F^{(k)}$ is $(k-2)$-connected and hence $\spi_i F^{(k)} = 0$ for $i \leq k-2$.
From the long exact sequence, it follows that $\Sigma^{\infty}_{X+} E \to \Sigma^{\infty}_{X+} E_{\leq k}$ induces an isomorphism of fibrewise stable homotopy groups in dimensions $i<k-2$.

In order to show that $\Sigma^\infty_{X+} E$ is a nilpotent $X$-parametrised spectrum, it is now sufficient to show that $\Sigma^\infty_{X+} E_{\leq k}$ is so for each stage of the Moore--Postnikov tower $E\to E_{\leq k}\to X$.
Since $p\colon E\to X$ is nilpotent, $E_k\to X$ factors as a composite of finitely many principal fibrations with $K(\pi, n)$'s as fibres.
Lemmas \ref{lem:NilExtensionSGood} and \ref{lem:PrincpalKpinSGood} below show that $\Sigma^\infty_{X+} E_{\leq k}$ is a nilpotent $X$-spectrum for all $k$.
\end{proof}

Essentially the same argument also proves
\begin{proposition}
\label{prop:StabNilFib2}
Let $Z$ be a retractive space over connected $X$ for which the projection map $Z\to X$ is a nilpotent fibration.
Then the fibrewise stabilisation $\Sigma^\infty_X Z$ is a nilpotent parametrised spectrum.
\end{proposition}

\begin{corollary}
For a nilpotent space $X$ with $x\in X$, $x_! S$ is a compact generator of $Ho(\mathrm{Sp}_X)_\mathrm{nil}$.
\end{corollary}
\begin{proof}
For connected $X$, the fibrewise stabilisation of $x\colon \ast\to X$ is a weak compact generator of $\mathrm{Sp}_X$ (Corollary \ref{cor:CompactGen}), so is a weak generator of $Ho(\mathrm{Sp}_X)_{\mathrm{nil}}$ provided it is nilpotent.
For nilpotent $X$, the path fibration $PX \to X$ is nilpotent and hence so is the fibrewise stabilisation $\Sigma^\infty_{X+} PX \cong x_! S $ by the Proposition (cf.~Example \ref{exam:TypicalXSpec} (2)).
\end{proof}

We say a fibration $p\colon E\to B$ with connected fibre $F$ is \emph{$S$-good} if the Serre action of $\pi_1 B$ on $\spi_k F$ is nilpotent for all $k\geq 0$.

\begin{lemma}
\label{lem:NilExtensionSGood}
Let $f\colon Z\to Y$ and $g\colon Y\to X$ be nilpotent fibrations.
If $f$ and $g$ are $S$-good, then so is $g\circ f\colon Z\to X$.
\end{lemma}
\begin{proof}
Letting $F(f)$, $F(g)$, and $F(gf)$ be the fibres of $f$, $g$, and $g\circ f$ respectively, there is a fibre sequence
\[
F(f)\longrightarrow F(gf) \longrightarrow F(g)\,.
\]
There is a convergent generalised homology spectral sequence
$H_p (F(g); \spi_q F(f))\Rightarrow \spi_{p+q} F(gf)$ upon which $\pi_1 Z$ acts.
To show that $\pi_1 Z$ (hence $\pi_1 X$) acts nilpotently on each of the $\spi_k F(gf)$, by Lemma \ref{lem:NilExactSeq} it is sufficient to show that $\pi_1 Z$ acts nilpotently on each of the homology groups $H_p (F(g); \spi_q F(f))$.

Since $\pi_1 Z$ acts nilpotently on $\spi_q F(f)$, there is a filtration by subgroups
\[
\spi_q F(f) = \Gamma_0 \geq \Gamma_1\geq \dots \geq \Gamma_j \geq \dotsb \geq \Gamma_n = 0
\]
such that each $\Gamma_{j+1} \to \Gamma_{j}$ is an injective homomorphism of $\pi_1 Z$-modules and the quotient $\Gamma_j /\Gamma_{j+1}$ has trivial $\pi_1 Z$-action.
On the other hand, since $g$ is a nilpotent fibration $\pi_1 X$ (and hence $\pi_1 Z$) acts nilpotently on each $H_p F(g)$ \cite[Lemma II.5.4]{bousfield_limits_1972} and so it follows that $\pi_1 Z$ acts nilpotently on $H_p F(g)\otimes (\Gamma_j/\Gamma_{j+1})$ and $\mathrm{Tor}(H_{p-1} F(g), \Gamma_j /\Gamma_{j+1})$ for each $p$.
The universal coefficient theorem and Lemma \ref{lem:NilExactSeq} together imply that $\pi_1 Z$ acts nilpotently on each $H_p (F(g);  \Gamma_j /\Gamma_{j+1})$.
Working over coefficient sequences, it now follows that $\pi_1 Z$ acts nilpotently on each $H_p (F(g); \spi_q F(f))$.
\end{proof}

\begin{lemma}
\label{lem:PrincpalKpinSGood}
Let $p\colon E\to B$ be a principal fibration with fibre $K(A,n)$.
Then the Serre action of $\pi_1 B$ on $\spi_k K(A,n)$ is trivial. 
\end{lemma}
\begin{proof}
A principal fibration of the form $K(A,n)\to E\to B$ is simple.
Every element $\gamma$ of $\pi_1 B$ induces a self-homotopy equivalence $\rho(\gamma) \colon K(A,n)\to K(A,n)$ such that the action of $\pi_1 B$ on $\pi_n K(A,n)$ is identified with the composite
\[
\begin{tikzcd}
\pi_1 B 
\ar[r, "\rho"]
&
{[K(A,n), K(A,n)]}
\ar[r, "\pi_n"]
&
\mathrm{Hom}(A,A)\,,
\end{tikzcd}
\]
where the latter map is an isomorphism.
Since the fibration is simple, each $\gamma\in \pi_1 B$ acts trivially on $\pi_n K(A,n)= A$ and hence each self-homotopy equivalence $\rho(\gamma)\colon K(A,n)\to K(A,n)$ must be homotopic to the identity.
The action of $\pi_1 B$ on $\spi_k K(A,n)$ is identified with the composite
\[
\begin{tikzcd}
\pi_1 B 
\ar[r, "\rho"]
&
{[K(A,n), K(A,n)]}
\ar[r, "\spi_k"]
&
\mathrm{Hom}(\spi_k K(A,n),\spi_k K(A,n))\,,
\end{tikzcd}
\]
which must therefore be trivial.
\end{proof}

\section{Sullivan's rational homotopy theory}
\label{S:RHT}
We now turn to the algebraic categories in which our models for rational homotopy types shall live.
Following \cite{bousfield_rational_1976}, we begin with a careful recollection of the model category of commutative differential graded algebras, the home of Sullivan's rational models for topological spaces.

\begin{definition}
A \emph{commutative differential graded algebra} (\emph{cdga}) over $\mathbb{Q}$ is a non-negatively graded rational cochain complex $(A^\ast, d)$ equipped with an associative unital graded commutative product $\mu\colon A\otimes A\to A$ for which $d$ is a graded derivation (of cohomological degree $1$).
The data $(A^\ast, d, \mu)$ are frequently abbreviated to $A$, with $d_A := d$ and $\mu_A := \mu$ the usual notation for the differential and multiplication of $A$.
We also frequently use the shorthand $a\cdot b := \mu_A (a,b)$ for the multiplication in $A$.

A morphism of cdgas $f\colon A\to B$ is a map of cochain complexes $f\colon (A^\ast, d_A)\to (B^\ast, d_B)$ that is compatible with multiplication maps in the sense that $\mu_B \circ (f\otimes f) = f \circ \mu_A$.

The collection of cdgas (over $\mathbb{Q}$) organises into a category, which we denote by $\mathrm{cDGA}$.
\end{definition}
\begin{example}
\label{exam:Minimal}
There is a class of morphisms in $\mathrm{cDGA}$ that plays a particularly important role in rational homotopy theory.
A morphism $f\colon A\to B$ is \emph{semifree} if the following conditions are satisfied:
\begin{enumerate}[label=(\roman*)]
  \item The underlying graded commutative algebra of $B$ is $B^\ast = A\otimes \Lambda V$, with $\Lambda V$ the free graded-commutative algebra generated by a non-negatively graded rational vector space $V$.
  
  \item The graded rational vector space $V$ has a basis $\{v_\alpha\mid \alpha\in \mathcal{I}\}$ where $\mathcal{I}$ is a well-ordered set and the differential $d_B$ on $B^\ast = A\otimes \Lambda V$ is such that for all $\beta\in \mathcal{I}$, the differential  $d_B v_\beta$ lies in the subcomplex $A\otimes \Lambda V_{<\beta}$ with $V_{<\beta}=\mathrm{span}\{v_\alpha\mid \alpha<\beta\}$.
  
  \item The morphism $f\colon A\to B$ is given on underlying cochain complexes by the evident inclusion $A\to A\otimes \Lambda V$.
\end{enumerate}
A semifree morphism $f\colon A\to B$ is \emph{minimal} if the basis $\{v_\alpha\mid \alpha\in \mathcal{I}\}$ of $V$ satisfies the condition
\[
\alpha \leq \beta
\;\;
\Longrightarrow
\;\;
|v_\alpha| \leq |v_\beta|\,.
\]
A cdga $A$ is \emph{semifree} or \emph{minimal} if the unit map $\eta\colon \mathbb{Q}\to A$ is so.
\end{example}

There is a well-known model structure on the category $\mathrm{cDGA}$ that enables us to do homotopy theory with cdgas \cite[Section 4]{bousfield_rational_1976}.
A map of cdgas $f\colon A\to B$ is a \emph{weak equivalence} or \emph{fibration} if the underlying map of cochain complexes is a quasi-isomorphism or epimorphism respectively.
Cofibrations in $\mathrm{cDGA}$ are the morphisms that have the left lifting property with respect to acyclic fibrations. 

\begin{remark}
\label{rem:cDGACofibGen}
The category $Ch^{\geq 0}$ of non-negatively graded rational cochain complexes admits a cofibrantly generated model structure for which  weak equivalences are quasi-isomorphisms,
fibrations are degreewise epimorphisms, and
 cofibrations are monomorphisms in all positive degrees. 
Generating sets of cofibrations and acyclic cofibrations for $Ch^{\geq 0}$ are given in terms of the \emph{sphere} and \emph{disk} complexes $S(n)$ and $D(n)$.
The sphere complex $S(n)$ is generated by a single closed element of degree $n$.
The disk complex $D(n)$ is generated by elements $a$ and $b$ in degrees $n-1$ and $n$ respectively, with $da = b$.
A set of generating cofibrations for $Ch^{\geq 0}$ is
\[
\mathcal{I} =\{i_n \colon S(n)\to D(n)\mid n\geq 1\}\cup \{i_0\colon 0\to S(0) \}\,
\]
where $i_n$ is the evident inclusion; a set of generating acyclic cofibrations is
\[
\mathcal{J} = \{j_n \colon 0\to D(n)\mid n\geq 1\}\,.
\]
The aforementioned model structure on $\mathrm{cDGA}$ is equivalent to that produced by transferring the model structure of $Ch^{\geq 0}$ along the free-forgetful adjunction
\[
\begin{tikzcd}
Ch^{\geq 0}
\ar[rr, shift left = 2, "\Lambda"]
\ar[rr, shift left=-2, leftarrow, "\bot"]
&&
\mathrm{cDGA}\,.
\end{tikzcd}
\]
In particular, the model structure on $\mathrm{cDGA}$ is cofibrantly generated with generating set of cofibrations $\mathcal{I}_{\mathrm{cDGA}} = \{\Lambda (i_n) \mid n\geq 0\}$  and of acyclic cofibrations $\mathcal{J}_{\mathrm{cDGA}} =\{\Lambda (j_n) \mid n\geq 1\}$.
\end{remark}

\begin{remark}
\label{rem:cdgaMinMod}
By the previous remark, any cofibration in $\mathrm{cDGA}$ is a retract of a relative $\mathcal{I}_{\mathrm{cDGA}}$-cell complex and conversely.
In particular, minimal and semifree morphisms of cdgas are cofibrations.
The utility of minimal morphisms stems from the fact that any morphism of cdgas $f\colon A\to B$ can be factored as
\[
\begin{tikzcd}
A
\ar[rr, "f"]
\ar[dr, "i"']
&&
B
\\
&
A\otimes \Lambda V 
\ar[ur, "f'"']
&
\end{tikzcd}
\]
where $i$ is minimal and $f'$ is a quasi-isomorphism.
The quasi-isomorphism $f'$ is a \emph{minimal model} of $f$ and minimal models are unique up to isomorphism---both existence and uniqueness of minimal models are proven by a version of the small object argument, working degree-by-degree.
Finally, a \emph{minimal model} for a cdga $A$ is a minimal model for the unit $\eta\colon \mathbb{Q}\to A$.
\end{remark}

Sullivan's approach to rational homotopy theory begins by assigning to each simplicial set $X$ a cdga computing the rational cohomology algebra of $X$.
To the combinatorial $n$-simplex $\Delta^n$ we assign the cdga of \emph{polynomial differential forms}
\[
\mathfrak{A}_n:= \Lambda (t_0, \dotsc, t_n, y_0, \dotsc, y_n) \big/ J_n\,,
\]
where the generators satisfy $|t_i |=0$ and $dt_i = y_i$ for all $i$, and we quotient by the differential ideal $J_n$ generated by $\langle 1- \sum_{i=0}^n t_i$ and $\sum_{i=0}^n y_i$.
This construction gives rise to a simplicial cdga
\[
\mathfrak{A}_\bullet \colon \Delta^\mathrm{op}\longmapsto\mathrm{cDGA}
\]
with face and degeneracy maps given by
\begin{equation}
\label{eqn:SDRface}
\partial_i\colon \mathfrak{A}_n \longrightarrow \mathfrak{A}_{n-1}\,,
\qquad
\partial_i\colon t_k \longmapsto
\begin{cases}
t_k & k<i \\
0 & k=i\\
t_{k-1} & k>i
\end{cases}
\end{equation}
and
\begin{equation}
\label{eqn:SDRdegen}
s_i\colon \mathfrak{A}_n \longrightarrow \mathfrak{A}_{n+1}\,,
\qquad
s_i\colon t_k \longmapsto
\begin{cases}
t_k & k<i \\
t_k + t_{k+1} & k=i\\
t_{k+1} & k>i\,.
\end{cases}
\end{equation}
Viewing $\mathfrak{A}_\bullet$ as a cosimplicial object of $\mathrm{cDGA}^\mathrm{op}$, the cdga of \emph{piecewise linear (PL) de Rham forms} on a simplicial set $X$ is obtained as the coend (in $\mathrm{cDGA}^\mathrm{op}$)
\[
\mathfrak{A}(X) := \int^{[n]\in \Delta} X_n \times \mathfrak{A}_n \cong \mathrm{sSet}\big(X, \mathfrak{A}_\bullet\big)\,.
\]
This assignment determines a functor $\mathfrak{A}\colon \mathrm{sSet}\to \mathrm{cDGA}^\mathrm{op}$ which, according to the yoga of nerve and realisation, participates in the \emph{Sullivan--de Rham adjunction}
\begin{equation}
\label{eqn:SullivandeRham}
\begin{tikzcd}
\mathrm{sSet}
\ar[rr, shift left=2,"\mathfrak{A}"]
\ar[rr, shift left=-2, leftarrow, "\bot", "\mathfrak{S}"']
&&
\mathrm{cDGA}^\mathrm{op}\,.
\end{tikzcd}
\end{equation}
The right adjoint $\mathfrak{S}$ sends a cdga $A$ to its \emph{spatial realisation}, which is the simplicial set 
\[
\mathfrak{S}(A)\colon
[n]\longmapsto \mathrm{cDGA}(A, \mathfrak{A}_n)\,.
\]

\begin{remark}
\label{rem:StokesMap}
For any simplicial set $X$ there is a natural $A_\infty$-quasi-isomorphism $\mathfrak{A}(X)\rightsquigarrow C^\bullet(X)$ between the PL de Rham forms on $X$ and the rational cochain complex \cite[Proposition 3.3]{bousfield_rational_1976}.
The underlying quasi-isomorphism is the \emph{Stokes map}
$
\rho_X \colon \mathfrak{A}(X)\to C^\bullet (X)$ determined by 
\[
\langle \rho_X \omega, \sigma\rangle = \int_{\Delta^n} \omega(\sigma)
\]
for an $n$-simplex $\sigma\colon \Delta^n \to X$.
The Stokes map induces an isomorphism of cohomology algebras $H^\bullet(\mathfrak{A}(X))\to H^\bullet(X)$.
\end{remark}
\begin{example}
\label{exam:MinimalEMandS}
We consider minimal models for two important families of spaces:
\begin{itemize}
  \item {\bf Spheres.} For $n=2k+1$ odd, a minimal model of the cdga $\mathfrak{A}(S^n)$ is given by the free algebra $\Lambda \langle \omega_n\rangle$ spanned by a single closed generator $\omega_n$ of degree $n$.
  For $n=2k$ even, a minimal model of $\mathfrak{A}(S^n)$ is given by the semifree algebra $\Lambda\langle\omega_n, \omega_{2n-1}\rangle$ spanned by a closed generator $\omega_n$ of degree $n$ and a generator $\omega_{2n-1}$ of degree $2n-1$ such that $d\omega_{2n-1} = \omega_n \wedge \omega_n$. 
  
  \item {\bf Eilenberg--Mac Lane spaces.} For all $n\geq 1$, a minimal model of the cdga $\mathfrak{A}(K(\mathbb{Q},n))$ is given by the algebra $\Lambda\langle \omega_n\rangle$ spanned by a single closed generator in degree $n$.
\end{itemize}
For odd $n$, the minimal models for $\mathfrak{A}(S^n)$ and $\mathfrak{A}(K(\mathbb{Q}, n))$ coincide.
\end{example}

Associated to any pair of cdgas $A, B$ there is a simplicial \emph{function space}
\[
F(A,B) := \mathrm{cDGA}(A, \mathfrak{A}_\bullet\otimes B)
\]
of maps $A\to B$.
If $i\colon A\to B$ is a cofibration of cdgas and $p\colon X\to Y$ is a fibration, the evident map
\[
F(B,X)\longrightarrow F(A, X)\prod_{F(A,Y)} F(B, Y)
\]
is a fibration of simplicial sets which is moreover a weak equivalence if either of $i$ or $p$ is \cite[Proposition 5.3]{bousfield_rational_1976}.
Since $\mathfrak{S}(A) = F(A, \mathbb{Q})$, it follows that the Sullivan--de Rham adjunction \eqref{eqn:SullivandeRham} is a Quillen adjunction.

The Sullivan--de Rham adjunction restricts to an adjoint equivalence between simplicial and algebraic versions of the rational homotopy category.
A connected, nilpotent space $X$ is \emph{of finite rational type} if the following equivalent conditions are satisfied:
\begin{itemize}
  \item The rational vector spaces $H_1(X)$ and $\pi_n X\otimes_\mathbb{Z}\mathbb{Q}$ for $n\geq 2$ are finite dimensional.
  \item The rational vector spaces $H_n(X)$ for $n\geq 1$ are finite dimensional.
\end{itemize}
A connected space $X$ is nilpotent and of finite rational type if and only if it is rationally homotopy equivalent (or, equivalently, rationally homology equivalent) to a space $X'$ for which the terminal map from each Postnikov section $X'_{\leq n}\to \ast$ factors as the composite of finitely many principal fibrations with fibres of the form $K(\mathbb{Q},r)$ where $r\geq 1$.
A space $X'$ of this latter form is called a \emph{rational} nilpotent space of finite type, and we write $Ho(\mathcal{S})^\mathbb{Q}_{\mathrm{nil,f.t.}} \hookrightarrow Ho(\mathcal{S})$ for the full subcategory spanned by such spaces.

On the algebraic side, we say that a cofibrant cdga $A$ is \emph{(homologically) connected} if the unit $\mathbb{Q}\to A$ induces an isomorphism in $\mathbb{Q}\to H^0(A)$.
Such a cdga is moreover of \emph{finite homotopical type} if the minimal model $N\to A$ is finite dimensional in each dimension.
Equivalently, a cofibrant connected cdga $A$ is of finite homotopical type precisely if each of the vector spaces
\[
\pi^n (A) := H^n ( QA)
\]
is finite dimensional.
Here, $Q$ is the complex of indecomposables which is given in terms of the augmentation ideal $\overline{A} = \mathrm{aug}_\mathbb{Q}(A) = \mathrm{coker}(\mathbb{Q}\to A)$ as $QA = \overline{A}/(\overline{A}\cdot \overline{A})$.
The full subcategory of $Ho(\mathrm{cDGA})$ spanned by the (cofibrant) connected cdgas of finite homotopical type is written $Ho(\mathrm{cDGA})_{\mathrm{f.t.,}\geq 1}\hookrightarrow Ho(\mathrm{cDGA})$.

\begin{remark}
\label{rem:CohomotopytoHomotopy}
The functor $A\mapsto \pi^\ast (A)$ on augmented algebras sends weak equivalences between cofibrant cdgas to isomorphisms of graded rational vector spaces.
If $A$ is a cofibrant connected cdga of finite homotopical type, for each $n\geq 1$ there is an isomorphism of rational vector spaces
\[
\pi_n \mathfrak{S}(A)\cong \mathrm{Hom}(\pi^n A, \mathbb{Q})\,,
\]
which is moreover a group isomorphism for $n\geq 2$ \cite[Proposition 8.13]{bousfield_rational_1976}. 
\end{remark}

Using the previous remark and the minimal models of Example \ref{exam:MinimalEMandS}, one proves that the derived unit and counit of the Sullivan--de Rham adjunction $(\mathfrak{A}\dashv \mathfrak{S})$ are, respectively, rational homotopy equivalences and quasi-isomorphism for the Eilenberg--Mac Lane spaces $K(\mathbb{Q},n)$ and their associated minimal models.
Working stage-by-stage in the Postnikov tower by means of an Eilenberg--Moore theorem for polynomial differential forms, Bousfield and Gugenheim prove the following
\begin{theorem}
\label{thm:SdR}
The derived Sullivan--de Rham adjunction restricts to an equivalence of categories
\[
\begin{tikzcd}
Ho(\mathcal{S})^\mathbb{Q}_{\mathrm{nil,f.t.}}
\ar[rr, shift left = 2, "\mathbf{L}\mathfrak{A}"]
\ar[rr, shift left =-2, leftarrow, "\simeq", "\mathbf{R}\mathfrak{S}"']
&&
Ho(\mathrm{cDGA})^\mathrm{op}_{\mathrm{f.t.,}\geq 1}\,.
\end{tikzcd}
\]
\end{theorem}
We shall also need the following property of the Sullivan--de Rham functor $\mathfrak{A}$.
\begin{lemma}
\label{lem:SdRandProducts}
For any simplicial sets $K, L$ there is a natural quasi-isomorphism of commutative differential graded algebras $\mathfrak{A}(K)\otimes \mathfrak{A}(L)\to \mathfrak{A}(K\times L)$.
\end{lemma}
\begin{proof}
The projections $K\times L\to K$ and $K\times L\to L$ induce maps of cdgas $\mathfrak{A}(K)\to \mathfrak{A}(K\times L)$ and $\mathfrak{A}(L)\to \mathfrak{A}(K\times L)$ which, by Remark \ref{rem:StokesMap}, realise the projections in rational cohomology. 
The K\"{u}nneth theorem implies that  $\mathfrak{A}(K)\otimes \mathfrak{A}(L)\to \mathfrak{A}(K\times L)$ is a quasi-isomorphism.
Naturality of this quasi-isomorphism is evident from the functoriality both of $\mathfrak{A}$ and of the projection maps.
\end{proof}

\subsection{Relative Sullivan--de Rham adjunctions}
In this section we examine relative versions of the Sullivan--de Rham adjunction obtained by passing to retractive objects. 
Having already encountered the topological protagonists of these adjunctions---the categories of retractive spaces $\mathrm{sSet}_{\dslash X}$ discussed above---we now briefly introduce the participants on the algebraic side of the story.
For any cdga $A$ there is a model category of \emph{augmented $A$-algebras}
\[
\mathrm{cDGA}_{\dslash A} := (\mathrm{cDGA}^{A/})_{/\mathrm{id}_A}
\]
whose objects are commuting diagrams of cdgas
\[
\begin{tikzcd}
A
\ar[rr, bend left =-20, "\mathrm{id}_A"']
\ar[r]
&
B
\ar[r]
&
A
\end{tikzcd}
\]
with the evident notion of morphism.
The model structure of $\mathrm{cDGA}_{\dslash A}$ is created by the forgetful functor to $\mathrm{cDGA}$ in the sense that a map of augmented $A$-algebras $B\to B'$ is a weak equivalence, fibration, or cofibration precisely if the underlying map of cdgas is so.
\begin{proposition}
For each cdga $A$ there is a Quillen adjunction
\begin{equation}
\label{eqn:SlicedSdR}
\begin{tikzcd}
\mathrm{sSet}_{\dslash \mathfrak{S}(A)}
\ar[rr, shift left=2,"\mathfrak{A}_{\dslash A}"]
\ar[rr, shift left=-2, leftarrow, "\bot", "\mathfrak{S}_{\dslash A}"']
&&
\big(\mathrm{cDGA}_{\dslash A})^\mathrm{op}\,.
\end{tikzcd}
\end{equation}
\end{proposition}
\begin{proof}[Proof sketch.]
This a routine application of the fact that Quillen adjunctions induce Quillen adjunctions on slice and coslice categories.

In slightly more detail, the left adjoint $\mathfrak{A}_{\dslash A}$ sends a retractive space $\mathfrak{S}(A)\xrightarrow{i} Z\xrightarrow{p} \mathfrak{S}(A)$ to the augmented $A$-algebra given by the top horizontal row in the following iterated pullback diagram of cdgas:
\[
\begin{tikzcd}
A
\ar[r]
\ar[d]
&
\mathfrak{A}_{\dslash A} (Z)
\ar[r]
\ar[d]
&
A
\ar[d]
\\
\mathfrak{A}\mathfrak{S}(A)
\ar[r, "\mathfrak{A}(p)"]
&
\mathfrak{A}(Z)
\ar[r, "\mathfrak{A}(i)"]
&
\mathfrak{A}\mathfrak{S}(A)\,.
\end{tikzcd}
\]
The right adjoint $\mathfrak{S}_{\dslash A}$ sends an augmented $A$-algebra $A \to B\to A$ to $\mathfrak{S}(A)\leftarrow \mathfrak{S}(B)\leftarrow \mathfrak{S}(A)$.
The unit and counit natural transformations of the adjunction $(\mathfrak{A}_{\dslash A}\dashv \mathfrak{S}_{\dslash A})$ arise from the unit and counit transformations of $(\mathfrak{A}\dashv \mathfrak{S})$ in the obvious way.

That the adjunction $(\mathfrak{A}_{\dslash A}\dashv \mathfrak{S}_{\slash A})$ is Quillen is a consequence of the fact that $\mathfrak{S}_{\dslash A}$ is given on underlying objects by $\mathfrak{S}$, so preserves fibrations and acyclic fibrations.
\end{proof}

Given a morphism of cdgas $f\colon A\to B$, there is a functor $f^\ast \colon \mathrm{cDGA}_{\dslash B}\to \mathrm{cDGA}_{\dslash A}$ given by forming pullbacks along $f$.
Concretely, an augmented $B$-algebra $B\to C\to B$ is sent to the augmented $A$-algebra appearing as the top horizontal row of the iterated pullback diagram of cdgas:
\[
\begin{tikzcd}
A
\ar[r]
\ar[d, "f"']
&
f^\ast C
\ar[r]
\ar[d]
&
A
\ar[d, "f"]
\\
B
\ar[r]
&
C
\ar[r]
&
B\,.
\end{tikzcd}
\]
The functor $f^\ast$ has left adjoint $f_!$ acting on objects by $(A\to D\to A)\mapsto (B\to B\bigotimes_{A} D\to B)$.
The adjunction $(f_!\dashv f^\ast)\colon \mathrm{cDGA}_{\dslash A}\to \mathrm{cDGA}_{\dslash B}$ is Quillen; the right adjoint $f^\ast$ clearly preserves fibrations and acyclic fibrations.

The base change adjunctions between model categories of retractive spaces and augmented relative algebras are compatible in the following sense:
\begin{proposition}
\label{prop:SlicedSdRPNat}
For any morphism of cdgas $f\colon A\to B$ the diagram of left Quillen functors
\[
\begin{tikzcd}
\mathrm{sSet}_{\dslash \mathfrak{S}(B)}
\ar[r, "\mathfrak{A}_{\dslash B}"]
\ar[d, "\mathfrak{S}(f)_!"']
&
\mathrm{cDGA}^\mathrm{op}_{\dslash B}
\ar[d, "f^\ast"]
\\
\mathrm{sSet}_{\dslash \mathfrak{S}(A)}
\ar[r, "\mathfrak{A}_{\dslash A}"]
&
\mathrm{cDGA}^\mathrm{op}_{\dslash A}
\end{tikzcd}
\]
commutes up to natural isomorphism.
\end{proposition}
\begin{proof}
The functor  $\mathfrak{A}\colon \mathrm{sSet}\to\mathrm{cDGA}^\mathrm{op}$ is a left adjoint, so sends pushouts of simplicial sets to pullbacks of cdgas.
For a retractive space $Z$ over $\mathfrak{S}(B)$, applying $\mathfrak{A}$ to the pushout diagram
\[
\begin{tikzcd}
\mathfrak{S}(B)
\ar[r, "\mathfrak{S}(f)"]
\ar[d]
&
\mathfrak{S}(A)
\ar[d]
\\
Z
\ar[r]
&
\mathfrak{S}(f)_! Z\
\end{tikzcd}
\]
gives an isomorphism of cdgas $\mathfrak{A}(\mathfrak{S}(f)_! Z) \cong \mathfrak{A}\mathfrak{S}(A)\prod_{\mathfrak{A}\mathfrak{S}(B)}\mathfrak{A}(Z)$ from which it follows that $(\mathfrak{A}_{\dslash A}\circ \mathfrak{S}(f)_!)(Z)$ is the pullback of the diagram
\[
\begin{tikzcd}
A
\ar[r]
&
\mathfrak{A}\mathfrak{S}(A)
\ar[r, "\mathfrak{A}\mathfrak{S}(f)"]
&
\mathfrak{A}\mathfrak{S}(B)
\ar[r, leftarrow]
&
\mathfrak{A}(Z)
\end{tikzcd}
\]
Naturality of the $(\mathfrak{A}\dashv \mathfrak{S})$-counit implies that this is the same as the pullback of the diagram
\[
\begin{tikzcd}
A
\ar[r, "f"]
&
B
\ar[r, "\mathfrak{A}\mathfrak{S}(f)"]
&
\mathfrak{A}\mathfrak{S}(B)
\ar[r, leftarrow]
&
\mathfrak{A}(Z)
\end{tikzcd}
\]
computing $(f^\ast\circ  \mathfrak{A}_{\dslash B})(Z)$.
We thus obtain an isomorphism $(\mathfrak{A}_{\dslash A}\circ \mathfrak{S}(f)_!)(Z)\cong (f^\ast\circ  \mathfrak{A}_{\dslash B})(Z)$ which is manifestly natural in $Z$.
\end{proof}

\subsection{Homotopy theory of differential graded modules}
In this section we discuss model categories of differential graded modules over differential graded algebras, the algebraic arena of our work in Sections \ref{S:FibStabRat} and \ref{S:Dictionary} below.

\begin{definition}
Let $A$ be a cdga over $\mathbb{Q}$.
A \emph{differential graded $A$-module} over $A$ (or simply \emph{$A$-module}) is a rational cochain complex $(M^\ast , d_M)$ equipped with a linear map of degree zero $\rho_M \colon A\otimes M \to M$, $(a,m)\mapsto a\cdot m$, subject to the conditions that 
\begin{itemize}
  \item $a\cdot (b\cdot m) = (a\cdot b)\cdot m$ for all $a,b \in A$ and $m\in M$,
  \item $1 \cdot m = m$ for all $m\in M$, with $1\in A$ the unit, and
  \item $d_M (a\cdot m) = d_A a \cdot m + (-1)^{|a|} a\cdot d_M m$ for all $a\in A$ and $m\in M$.
\end{itemize}
The data $(M^\ast , d_M, \rho_M)$ are frequently abbreviated to $M$, with differential $d_M$ and $A$-module structure $\rho_M$ understood from context.
We do not assume $M^\ast$ to be connective, so that the grading is over $\mathbb{Z}$.

A morphism $f\colon M \to N$ of $A$-modules is a map of cochain complexes $f\colon (M^\ast, d_M)\to (N^\ast, d_N)$ such that $f\circ \rho_M = \rho_N\circ (\mathrm{id}_A \otimes f)$.
The $A$-modules are organised into a category, denoted by $A\mathrm{-Mod}$.
\end{definition}

For any cdga $A$, $A\mathrm{-Mod}$ is equipped with a natural projective model structure.
In this projective model structure, a morphism of $A$-modules $f\colon M\to N$ is a \emph{weak equivalence} or \emph{fibration} if the underlying map of cochain complexes is a quasi-isomorphism or degreewise epimorphism respectively.
Cofibrations in $A\mathrm{-Mod}$ are morphisms that have the left lifting property with respect to acyclic fibrations.

\begin{remark}
The category $Ch$ of unbounded rational cochain complexes admits a cofibrantly generated model structure for which the weak equivalences are quasi-isomorphisms and fibrations are degreewise epimorphisms.
A set of generating cofibrations is given in terms of the sphere and disk complexes (cf.~Remark \ref{rem:cDGACofibGen}) by $\mathcal{I} = \{i_n \colon S(n)\to D(n)\mid n\in \mathbb{Z}\}$,
with a set of generating acyclic cofibrations given by $\mathcal{J} = \{j_n \colon 0\to D(n)\mid n\in \mathbb{Z}\}$.

For any cdga $A$ the free-forgetful adjunction
\[
\begin{tikzcd}
Ch
\ar[rr, shift left =2 , "A\otimes (-)"]
\ar[rr, leftarrow, shift left =-2, "\bot"]
&&
A\mathrm{-Mod}
\end{tikzcd}
\]
induces a transferred model structure on $A\mathrm{-Mod}$.
This transferred model structure coincides with the projective model structure so that $A\mathrm{-Mod}$ is cofibrantly generated with generating cofibrations $\mathcal{I}_A =\{A\otimes i_n\mid n\in \mathbb{Z}\}$ and generating acyclic cofibrations $\mathcal{J}_A =\{A\otimes j_n \mid n\in \mathbb{Z}\}$.
\end{remark}

\begin{definition}
\label{defn:SemifreeAMod}
A morphism of $A$-modules $f\colon M\to N$ is \emph{semifree} if the following conditions are satisfied:
\begin{enumerate}[label=(\roman*)]
  \item The underlying graded $A^\ast$-module of $N^\ast$ is $M^\ast \oplus (A^\ast\otimes V)$ for some graded rational vector space $V$.
  
  \item $V$ has a basis $\{v_\alpha\mid \alpha\in \mathcal{I}\}$ for some well-ordered set $\mathcal{I}$ and the differential $d_N$ on $N$ is such that for all $\beta\in \mathcal{I}$, $d_B v_\beta$ lies in the $A$-submodule $M\oplus (A\otimes V_{<\beta})$, with $V_{<\beta}= \mathrm{span}\{v_\alpha\mid \alpha < \beta\}$.
  
  \item The morphism of $A$-modules $f\colon M\to N$ is given on underlying cochain complexes by the evident inclusion $M\to M\oplus (A\otimes V)$.
\end{enumerate}
A semifree morphism of $A$-modules $f\colon M\to N$ is moreover \emph{minimal} if the basis  $\{v_\alpha\mid \alpha\in \mathcal{I}\}$ of $V$ satisfies the additional condition
\[
\alpha\leq \beta \,\Longrightarrow\;|v_\alpha|\leq |v_\beta|\,.
\]
An $A$-module $M$ is \emph{semifree} or \emph{minimal} if the trivial map of $A$-modules $0\to M$ is so. 
\end{definition}

\begin{remark}
\label{rem:SemiFreeChar}
Any semifree morphism of $A$-modules is a cofibration.
Indeed, for a semifree map $f\colon M\to N$, parts (ii) and (iii) of Definition 
\ref{defn:SemifreeAMod} are equivalent to expressing $f$ as the colimit of a sequence of maps 
\[
M
\longrightarrow
\dotsb
\longrightarrow
N_{\alpha}
\longrightarrow
N_{\alpha +1}
\longrightarrow 
\dotsb
\]
where $N_\alpha = M\oplus (A\otimes V_{<(\alpha +1)})$ and there is a pushout diagram of $A$-modules
\begin{equation}
\label{eqn:DiskAttachAMod}
\begin{tikzcd}
A\otimes S(|v_{\alpha+1}|+1)
\ar[r, "dv_{\alpha+1}"]
\ar[d]
&
N_{\alpha}
\ar[d]
\\
A\otimes D(|v_{\alpha+1}| +1)
\ar[r]
&
N_{\alpha+1}
\end{tikzcd}
\end{equation}
for each $\alpha\in \mathcal{I}$.
The condition that $f$ is minimal is equivalent to requiring that the degrees of the disk attachments \eqref{eqn:DiskAttachAMod} is a monotonically increasing function of $\alpha$.
\end{remark}

\begin{example}
\label{exam:MinimalOverQ}
From the characterisation of minimality given above it follows that the minimal $\mathbb{Q}$-modules are precisely the bounded below cochain complexes with trivial differential.
\end{example}

We shall make use of some important properties of minimal maps of $A$-modules.
In the case of connective $A$-modules (cf.~Definiton \ref{defn:ConnectiveModule} below), proofs of these properties can be found in the appendix to \cite{roig_minimal_2000}.
Essentially the same proofs apply in the present setting.
\begin{lemma}
\label{lem:MinimalModuleExist}
Let $A$ be a cdga and let $f\colon M\to N$ be a morphism of $A$-modules.
If there is some integer $k$ such that $H^p (f)\colon H^p(M)\to H^p(N)$ is an isomorphism for $p< k$ and a monomorphism for $p =k$ then $f$ factorises as
\[
\begin{tikzcd}
M
\ar[rr, "f"]
\ar[dr, "i"']
&&
N
\\
&
M\oplus (A\otimes V)
\ar[ur, "f'"']
&
\end{tikzcd}
\]
where $i$ is minimal and $f'$ is a quasi-isomorphism.
In particular, any bounded below $A$-module has a minimal model.
\end{lemma}

\begin{lemma}
Let $A$ be a cdga and $f\colon M\to N$ a map of $A$-modules.
If a minimal model for $f$ exists, it is unique up to isomorphism fixing $M$.
\end{lemma}

\begin{definition}
\label{defn:ConnectiveModule}
An $A$-module $M$ is \emph{connective} if the underlying cochain complex $M^\ast$ is connective, that is $M^k  = 0$ for $k<0$.
The full subcategory of $A\mathrm{-Mod}$ spanned by the connective $A$-modules is written $A\mathrm{-Mod}_{\geq 0}$.
\end{definition}

\begin{remark}
\label{rem:AModConn}
The inclusion $A\mathrm{-Mod}_{\geq 0}\hookrightarrow A\mathrm{-Mod}$ preserves all small limits and hence has a left adjoint $\mathrm{cn}^0_A$ by the adjoint functor theorem.
The behaviour of $\mathrm{cn}^0_A$ on general $A$-modules is not straightforward though an explicit description can be given using monadicity.
For free $A$-modules $M = A\otimes V$ we have that $\mathrm{cn}^0_A (M) \cong A\otimes \mathrm{cn}^0 V$ is the free $A$-module on the cochain complex
\[
\mathrm{cn}^0 V 
=
\Big[
\begin{tikzcd}
\dotsb
\ar[r]
0
\ar[r]
&
V^0/d(V^{-1})
\ar[r, "d'"]
&
V^1 
\ar[r, "d"]
&
\dotsb
\end{tikzcd}
\Big]
\]
with $d'$ the unique map induced by the differential on the quotient.
\end{remark}

\begin{example}
For a minimal $A$-module $A\otimes V$, let $\{v_\alpha\}_{\alpha\in \mathcal{I}}$ be a basis for $V$ indexed by a well-ordered set $\mathcal{I}$ subject to the conditions that (i) $d v_\alpha \in  A\otimes V_{<\alpha}$ and (ii) $\alpha \leq \beta$ implies $|v_\alpha| \leq |v_\beta|$ for all $\alpha, \beta\in \mathcal{I}$.
The connective cover of $A\otimes V$ is determined by the cofibre sequence of $A$-modules
\[
A\otimes V^{\leq -1}
\longrightarrow
A\otimes V
\longrightarrow
\mathrm{cn}^0_A (A\otimes V)\,,
\]
where $A\otimes V^{\leq -1}$ is the minimal $A$-submodule generated by the $v_\alpha$ for which $|v_\alpha|\leq -1$.
\end{example}

For a map of cdgas $f\colon B\to A$, the \emph{augmentation ideal} of $B$ is the kernel $\mathrm{aug}_A (B) := \ker (f)$.
In the case that $B\in \mathrm{cDGA}_{\dslash A}$ is an augmented $A$-algebra, the underlying cochain complex of $B$ splits as  $B \cong A\oplus \mathrm{aug}_A (B)$ and the augmentation ideal inherits the structure of an $A$-module.
It is easy to check that the resulting functor
$
\mathrm{aug}_A\colon \mathrm{cDGA}_{\dslash A}\longrightarrow A\mathrm{-Mod}_{\geq 0}
$
preserves limits and so has a left adjoint $\Lambda_A$ by the adjoint functor theorem.
The left adjoint $\Lambda_A$ computes free $A$-algebras; we do not need an explicit description in the general case, but note that for a free $A$-module $M\cong A\otimes V$ we have $\Lambda_A (M) \cong A\otimes \Lambda V$.

In summary, for each cdga $A$ we have a composite adjunction
\begin{equation}
\label{eqn:SymAug}
\begin{tikzcd}
A\mathrm{-Mod}
\ar[rr, shift left =2 , "\mathrm{cn}^0_A"]
\ar[rr, hookleftarrow, shift left =-2, "\bot"]
&&
A\mathrm{-Mod}_{\geq 0}
\ar[rr, shift left =2 , "\Lambda_A"]
\ar[rr, leftarrow, shift left =-2, "\bot", "\mathrm{aug}_A"']
&&
\mathrm{cDGA}_{\dslash A}
\end{tikzcd}
\end{equation}
By an abuse of notation we write this adjunction as $(\Lambda_A\dashv \mathrm{aug}_A)$.

\begin{proposition}
\label{prop:SymAugPNat}
The adjunction
\[
\begin{tikzcd}
A\mathrm{-Mod}
\ar[rr, shift left =2 , "\Lambda_A"]
\ar[rr, leftarrow, shift left =-2, "\bot", "\mathrm{aug}_A"']
&&
\mathrm{cDGA}_{\dslash A}
\end{tikzcd}
\]
is Quillen.
\end{proposition}
\begin{proof}
Let $f\colon B\to C$ be a map of augmented $A$-algebras.
The underlying map of chain complexes splits as $\mathrm{id}\oplus \mathrm{aug}_A f\colon A\oplus \mathrm{aug}_A B\to A\oplus \mathrm{aug}_A C$ and it is clear that $f$ is a quasi-isomorphism or epimorphism precisely if $\mathrm{aug}_A f$ is.
Thus $\mathrm{aug}_A$ preserves fibrations and acyclic fibrations and so the adjunction is Quillen.
\end{proof}

Associated to a morphism of cdgas $f\colon A\to B$ there is a base change adjunction relating the module categories $(f_!\dashv f^\ast)\colon A\mathrm{-Mod}\to B\mathrm{-Mod}$.
The left adjoint sends an $A$-module $M$ to the $B$-module $f_! (M) := B\otimes_A M$ and the right adjoint regards a $B$-module $N$ as an $A$-module via restriction along $f$.
Since $f^\ast$ leaves the underlying cochain complexes fixed it is evidently a right Quillen functor.
Note that we use the same symbols to denote base change functors between categories of modules and augmented relative algebras.
This mild abuse of notation is justified by the following 
\begin{proposition}
\label{prop:Alg2ModPNat}
For any morphism of cdgas $f\colon A\to B$ the diagram of left Quillen functors
\[
\begin{tikzcd}
A\mathrm{-Mod}
\ar[r, "\Lambda_A"]
\ar[d, "f_!"']
&
\mathrm{cDGA}_{\dslash A}
\ar[d, "f_!"]
\\
B\mathrm{-Mod}
\ar[r, "\Lambda_B"]
&
\mathrm{cDGA}_{\dslash B}
\end{tikzcd}
\]
commutes up to natural isomorphism.
\end{proposition}
\begin{proof}
By essential uniqueness of adjoints, to prove the claim it is sufficient to check that there is a natural isomorphism of composite right adjoints $\mathrm{aug}_A \circ f^\ast\cong f^\ast \circ \mathrm{aug}_B$.
For any augmented $B$-algebra $C$ we have a natural splitting of cochain complexes $C\cong B\oplus \mathrm{aug}_B C$ which implies that the underlying cochain complex of the augmented $A$-algebra $f^\ast C$ naturally splits as $f^\ast C\cong A\oplus \mathrm{aug}_B C$.
By construction, the induced $A$-module structure on $\mathrm{aug}_A (f^\ast C) \cong \mathrm{aug}_B C$ coincides with the restriction of the $B$-module structure along $f$.
\end{proof}

\begin{remark}
\label{rem:BCMinimal}
For a morphism of cdgas $f\colon A\to B$, extension of scalars $f_!\colon A\mathrm{-Mod}\to B\mathrm{-Mod}$ preserves semifreeness and minimality.
This is an easy consequence of the characterisations of these properties given in Remark \ref{rem:SemiFreeChar}.
In the case that $f\colon A\to \mathbb{Q}$ is an augmentation, any minimal $A$-module $A\otimes V$ is thus sent by $f_!$ to $V$ regarded as a chain complex with trivial differential.
\end{remark}

\subsection{Modular models and fibrewise suspensions}
\label{ss:ModularModels}
Combining \eqref{eqn:SlicedSdR} and \eqref{eqn:SymAug}, for any cdga $A$ we obtain a composite Quillen adjunction
\begin{equation}
\label{eqn:UnstableAdj}
(\mathfrak{M}_A^u\dashv \mathfrak{P}_A^u)\colon 
\begin{tikzcd}
\mathrm{sSet}_{\dslash \mathfrak{S}(A)}
\ar[rr, shift left =2, "\mathfrak{A}_{\dslash A}"]
\ar[rr, shift left=-2, leftarrow, "\bot", "\mathfrak{S}_{\dslash A}"']
&&
\mathrm{cDGA}_{\dslash A}^\mathrm{op}
\ar[rr, shift left =2, "\mathrm{aug}_A"]
\ar[rr, shift left=-2, leftarrow, "\bot", "\Lambda_A"']
&&
A\mathrm{-Mod}^\mathrm{op}\,.
\end{tikzcd}
\end{equation}
We will later construct algebraic models for rational parametrised spectra by stabilising these adjunctions over a fixed cdga $A$ (the superscript \lq\lq$u$'' is for \lq\lq unstable'').
In order to do this, we must first examine modules in the image of $\mathfrak{M}^u_A$ and how this functor behaves under forming fibrewise suspensions. 

For a retractive space $Z$ over $\mathfrak{S}(A)$, the underlying cochain complex of $\mathfrak{A}(Z)$ splits naturally as $\mathfrak{A}(Z)\cong \mathfrak{A}\mathfrak{S}(A)\oplus \mathrm{aug}_{\mathfrak{A}\mathfrak{S}(A)} \mathfrak{A}(Z)$.
Since $\mathfrak{A}_{\dslash Z}$ is constructed by pulling back $\mathfrak{A}(Z)\to \mathfrak{A}\mathfrak{S}(A)$ along the Sullivan--de Rham counit $A\to \mathfrak{A}\mathfrak{S}(A)$, it follows that 
\[
\mathrm{aug}_{A}\big(\mathfrak{A}_{\dslash A} (Z)\big) \cong \mathrm{aug}_{\mathfrak{A}\mathfrak{S}(A)}\big( \mathfrak{A}(Z)\big)
\] 
as $A$-modules.
Further, the Sullivan--de Rham functor $\mathfrak{A}\colon \mathrm{sSet}\to \mathrm{cDGA}^\mathrm{op}$ sends the pushout diagram of simplicial sets 
\[
\begin{tikzcd}
\mathfrak{S}(A)
\ar[r]
\ar[d]
&
\ast
\ar[d]
\\
Z
\ar[r]
&
Z/\mathfrak{S}(A)
\end{tikzcd}
\]
to a pullback diagram of cdgas.
Since $Z/\mathfrak{S}(A)$ is pointed, the underlying cochain complex of $\mathfrak{A}(Z/\mathfrak{S}(A))$ splits as $\mathfrak{A}(Z/\mathfrak{S}(A))\cong \mathbb{Q}\oplus \mathrm{aug}_\mathbb{Q} (\mathfrak{A}(Z/\mathfrak{S}(A)))$ and there is a natural isomorphism
$
\mathrm{aug}_{\mathfrak{A}\mathfrak{S}(A)}( \mathfrak{A}(Z))
\cong 
\mathrm{aug}_\mathbb{Q}(\mathfrak{A}(Z/\mathfrak{S}(A))
$ of $\mathfrak{A}\mathfrak{S}(A)$-modules.
Combined with the above, we get a natural isomorphism of $A$-modules
\[
\mathrm{aug}_A \big(\mathfrak{A}_{\dslash A}(Z)\big)\cong \mathrm{aug}_\mathbb{Q} \big(\mathfrak{A}(Z/\mathfrak{S}(A)\big)
\]
so that $\mathfrak{M}^u_A$ sends $Z$ to the $A$-module $\mathrm{aug}_\mathbb{Q}(\mathfrak{A}(Z/\mathfrak{S}(A))$.
The next result provides a topological interpretation of this $A$-module under certain mild conditions.
\begin{proposition}
\label{prop:InterpretUModule}
Let $A$ be a cofibrant connected cdga of finite homotopical type. 
For any retractive space $Z$ over $\mathfrak{S}(A)$, the $H^\bullet (A)$-action on $H^\bullet(\mathfrak{M}^u_A(Z))$ is naturally isomorphic to the $H^\bullet(\mathfrak{S}(A))$-action on the reduced rational cohomology $\widetilde{H}^\bullet (Z/\mathfrak{S}(A))$.
\end{proposition}
\begin{proof}
For any simplicial set $X$ and retractive space $Z\in \mathrm{sSet}_{\dslash X}$, applying the rational cochains functor to the iterated pushout diagram
\begin{equation}
\label{eqn:CollapseRetract}
\begin{tikzcd}
X
\ar[r, "i"]
\ar[d]
&
Z
\ar[r, "p"]
\ar[d]
&
X
\ar[d]
\\
\ast
\ar[r]
&
Z/X
\ar[r]
&
\ast
\end{tikzcd}
\end{equation}
gives splittings of cochain complexes
$
C^\bullet(Z)\cong C^\bullet(X) \oplus M
$
and $C^\bullet(Z/X)\cong \mathbb{Q} \oplus M'$.
With respect to these splittings the quotient map $Z\to Z/X$ gives rise to the map $C^\bullet(Z/X)\to C^\bullet(Z)$ given as the sum of the unit $\mathbb{Q}\to C^\bullet(X)$ and a quasi-isomorphism $M' \to M$.

The $H^\bullet(X)$-module structure on $\widetilde{H}^\bullet (Z/X)$ is induced by the space-level coaction
\begin{align*}
Z/X &\longrightarrow X_+\wedge Z/X\\
[z]&\longmapsto p(z) \wedge [z]\,.
\end{align*}
Examining the commuting diagram
\[
\begin{tikzcd}
Z_+
\ar[r]
\ar[d]
&
Z_+ \wedge Z_+
\ar[r, "p_+\wedge\mathrm{id}_{Z_+}"]
&
X_+\wedge Z_+
\ar[d]
\\
Z/X
\ar[rr]
&&
X_+\wedge Z/X
\end{tikzcd}
\]
we find that the  $H^\bullet(X)$-module structure  on $\widetilde{H}^\bullet (Z/X)$ is compatible with the $H^\bullet(X)$-module structure on $H^\bullet(Z)$ induced by the coaction $Z\to X\times Z$.
Taking cohomology, the $M'\to M$ thus becomes an isomorphism of $H^\bullet(X)$-modules.

Applying the Sullivan--de Rham functor $\mathfrak{A}$ to the diagram \eqref{eqn:CollapseRetract} we get splittings of cochain complexes $\mathfrak{A}(Z)\cong \mathfrak{A}(X)\oplus \mathrm{aug}_{\mathfrak{A}(X)}\mathfrak{A}(Z)$ and $\mathfrak{A}(Z/X)\cong \mathbb{Q}\oplus \mathrm{aug}_\mathbb{Q}\mathfrak{A}(Z/X)$.
By \eqref{eqn:CollapseRetract}, the components of the natural $A_\infty$-quasi-isomorphism $\mathfrak{A}(Y)\rightsquigarrow C^\bullet(Y)$ (Remark \ref{rem:StokesMap}) at $Z$ and $Z/X$ respect the splittings and hence we obtain a commuting diagram of quasi-isomorphisms of cochain complexes
\[
\begin{tikzcd}
\mathrm{aug}_\mathbb{Q}\mathfrak{A}(Z/X)
\ar[r]
\ar[d]
&
\mathrm{aug}_{\mathfrak{A}(X)}\mathfrak{A}(Z)
\ar[d]
\\
M'
\ar[r]
&
M
\end{tikzcd}
\] 
Since the vertical arrows arise from $A_\infty$-quasi-isomorphisms, all arrows in the above diagram are isomorphisms of $H^\bullet(X)\cong H^\bullet(\mathfrak{A}(X))$-modules in cohomology.

In the case that $X=\mathfrak{S}(A)$ is the spatial realisation of a cofibrant connected cdga $A$ of finite rational type, the Sullivan--de Rham counit $A\to \mathfrak{A}\mathfrak{S}(A)$ is a weak equivalence by the Sullivan--de Rham equivalence theorem.
For a retractive space $Z\in \mathrm{sSet}_{\dslash \mathfrak{S}(A)}$, the $A$-module structure on $\mathfrak{M}^u_A (Z)$ is obtained by restricting the $\mathfrak{A}\mathfrak{S}(A)$-module structure on $\mathrm{aug}_\mathbb{Q} \mathfrak{A}(Z/\mathfrak{S}(A))$ considered above along the counit $A\to \mathfrak{A}\mathfrak{S}(A)$.
With this observation, the result follows.
\end{proof}

\begin{example}
\label{exam:FreeFibBasePointModule}
Consider the case that $Z = Y_{X+}$ for some $p\colon Y\to X$ (cf.~Example \ref{exam:FibrewiseSuspensionSpectra}).
In this case we have $Z/X \cong Y_+$ and $\mathfrak{M}^u_A(Z) \cong \mathfrak{A}(Y)$ regarded as an $A$-module by restriction along $A\to \mathfrak{A}\mathfrak{S}(A) = \mathfrak{A}(X)\to \mathfrak{A}(Y)$.
\end{example}

We now seek to understand how $\mathfrak{M}^u_A$ interacts with fibrewise stabilisation.
To ease the notation we adopt the following convention:
\begin{center}
{\bf For the remainder of this section we write $X:= \mathfrak{S}(A)$ for the spatial realisation of $A$}.
\end{center}
As $\Sigma_X Z$ coincides with the cofibre of $\partial\Delta^1\otimes_X Z\to \Delta^1 \otimes_X Z $ (cf.~Remark \ref{rem:SuspensionPresentation}), our first step is to analyse how $\mathfrak{A}_{\dslash A}$ interacts with finite $\mathrm{sSet}$-tensors.

\begin{lemma}
\label{lem:FinTensorModuleCompare}
For any simplicial set $K$ and $Z\in \mathrm{sSet}_{\dslash X}$ there is a natural quasi-isomorphism of $A$-modules
\[
\mathfrak{A}(K)\otimes \mathfrak{M}^u_A(Z)\longrightarrow
\mathfrak{M}^u_A (K\otimes_X Z)
\]
with $A$ acting on the second tensor factor of $\mathfrak{A}(K)\otimes \mathfrak{M}^u_A(Z)$.
\end{lemma}
\begin{proof}
There is a commuting diagram of simplicial sets in which each square is a pushout
\begin{equation}
\label{eqn:BigTensorDiagram}
\begin{tikzcd}[sep=small]
K\times X
\ar[r]
\ar[d]
&
X
\ar[r]
\ar[d]
&
\ast
\ar[r, leftarrow]
\ar[d]
&
K
\ar[d]
\\
K\times Z
\ar[r]
\ar[d]
&
K\otimes_X Z
\ar[r]
\ar[d]
&
(K\times Z)/(K\times X)
\ar[r, leftarrow]
\ar[d]
&
K\times Z/X
\ar[d]
\\
K\times X
\ar[r]
&
X
\ar[r]
&
\ast
\ar[r, leftarrow]
&
K
\end{tikzcd}
\end{equation}
which, upon applying the Sullivan--de Rham functor $\mathfrak{A}$, gives rise to isomorphisms of cochain complexes
\[
\mathrm{aug}_{\mathfrak{A}(K\times X)}\mathfrak{A}(K\times Z)
\cong
\mathrm{aug}_{\mathfrak{A}(X)}\mathfrak{A}(K\otimes_X Z)
\cong
\mathrm{aug}_{\mathbb{Q}}\mathfrak{A}((K\times Z)/(K\times X))
\cong
\mathrm{aug}_{\mathfrak{A}(K)}\mathfrak{A}(K\times Z/X)
\]
compatible with the various module structures.
By Lemma \ref{lem:SdRandProducts},  associated to the right hand column in the diagram \eqref{eqn:BigTensorDiagram} there is a commuting diagram of cdgas
\[
\begin{tikzcd}[sep=small]
\mathfrak{A}(K)
\ar[r]
\ar[d]
&
\mathfrak{A}(K\times Z/X)
\ar[d]
\\
\mathfrak{A}(K)\otimes \mathfrak{A}(Z/X)
\ar[r]
\ar[ur, "\varpi", "\sim"']
&
\mathfrak{A}(K)\,,
\end{tikzcd}
\]
in which $\varpi$ is a quasi-isomorphism and the left hand vertical and bottom horizontal arrows arise from $\mathbb{Q}\to \mathfrak{A}(Z/X)\to \mathbb{Q}$ by tensoring with $\mathfrak{A}(K)$.
Passing to augmentation ideals we get a quasi-isomorphism of cochain complexes
$\mathfrak{A}(K)\otimes \mathrm{aug}_\mathbb{Q}\mathfrak{A}(Z/X)\to \mathrm{aug}_{\mathfrak{A}(K)}\mathfrak{A}(K\times Z/X)$ and hence a quasi-isomorphism
\[
\mathfrak{A}(K)\otimes \mathrm{aug}_\mathbb{Q}\mathfrak{A}(Z/X)
\longrightarrow
\mathrm{aug}_{\mathfrak{A}(X)}\mathfrak{A}(K\otimes_X Z)\,.
\]
Similarly to the proof of Proposition \ref{prop:InterpretUModule}, it is easy to show that this quasi-isomorphism respects the $A$-module structures.
There are natural quasi-isomorphisms $\mathfrak{M}^u_A(Z) \cong \mathrm{aug}_\mathbb{Q}\mathfrak{A}(Z/X)$ and $\mathfrak{M}^u_A(K\otimes_X Z)\cong \mathrm{aug}_{\mathfrak{A}(X)}\mathfrak{A}(K\otimes_X Z)$ respecting the $A$-module structures.
This  observation completes the proof.
\end{proof}

The category $\mathrm{sSet}_{\dslash X}$ is also tensored over pointed simplicial sets, so we have the following
\begin{corollary}
\label{cor:FinSmashModuleCompare}
For any pointed simplicial set $K$ and $Z\in \mathrm{sSet}_{\dslash X}$ there is a natural quasi-isomorphism of $A$-modules
\[
\mathfrak{M}^u_\mathbb{Q}(K)\otimes \mathfrak{M}^u_A(Z)\longrightarrow
\mathfrak{M}^u_A (K\owedge_X Z)
\]
with $A$ acting on the second tensor factor of $\mathfrak{M}^u_\mathbb{Q}(K)\otimes \mathfrak{M}^u_A(Z)$.
\end{corollary}
\begin{proof}
For a pointed simplicial set $(K,k)$ and retractive space $Z$ over $X$, the tensoring $K\owedge_X Z$ is the cofibre of the map $\ast \otimes_X Z \cong Z\to K\otimes_X Z$ obtained by tensoring the retractive space $Z$ with $k\colon \ast \to K$.
Since $\mathfrak{A}(K)\cong \mathbb{Q} \oplus \mathfrak{M}^u_\mathbb{Q}(K)$ as chain complexes, the result follows by the Lemma, the fact that $\mathfrak{M}^u_A$ sends cofibre sequences to fibre sequences of $A$-modules, and a routine application of the five lemma. 
\end{proof}

\begin{remark}
\label{rem:Shifts}
For any integer $n$, write $\mathbb{Q}[n]$ for the rational cochain complex consisting of a single copy of $\mathbb{Q}$ concentrated in degree $-n$.
For a cochain complex $M$ we write $M[n] := M\otimes \mathbb{Q}[n]$, so that  $M[n]^k = M^{k+n}$ and $d_{M[n]} = (-1)^n d_M$.
If $M$ is an $A$-module then so is $M[n]$; writing $m\mapsto m[n]$ for the canonical isomorphism $M^{k+n}\to M[n]^k$ the action of a homogeneous element $a$ of $A$ is given by $a\cdot m[n] := (-1)^{n\cdot|a|} (a\cdot m)[n]$.
\end{remark}

\begin{lemma}
\label{lem:SuspComparison}
Let $A$ be a cdga.
For any retractive space $Z$ over $X =\mathfrak{S}(A)$ there is a natural quasi-isomorphism of $A$-modules
$\tau_Z\colon \mathfrak{M}^u_A (Z)[-1]\to \mathfrak{M}^u_A (\Sigma_X Z)$.
\end{lemma}
\begin{proof}
By Remark \ref{rem:SuspensionPresentation}, the fibrewise suspension $\Sigma_X Z$ is computed via the pushout diagram in $\mathrm{sSet}_{\dslash X}$
\[
\begin{tikzcd}
\partial \Delta^1 \otimes_X Z
\ar[r, rightarrowtail]
\ar[d]
&
\Delta^1 \otimes_X Z
\ar[d]
\\
X
\ar[r]
&
\Sigma_X Z\,.
\end{tikzcd}
\]
The tensor $\partial \Delta^1 \otimes_X Z = Z\coprod_X Z$ is the coproduct of  two copies of $Z$ in $\mathrm{sSet}_{\dslash X}$ and $X\in \mathrm{sSet}_{\dslash X}$ is the zero object, hence upon applying $\mathfrak{M}^u_A$ and invoking Lemma \ref{lem:FinTensorModuleCompare} we get a morphism of pullback squares of $A$-modules
\begin{equation}
\label{eqn:tau1}
\begin{tikzcd}[sep = small]
&
\mathfrak{M}^u_A(\Sigma_X Z)
\ar[rr]
\ar[dd]
&&
0
\ar[dd]
\\
P(Z)
\ar[rr, crossing over]
\ar[ur]
\ar[dd]
&&
0
\ar[ur, equals]
&
\\
&
\mathfrak{M}^u_A(\Delta^1\otimes_X Z)
\ar[rr, twoheadrightarrow]
&&
\mathfrak{M}^u_A(Z)\oplus \mathfrak{M}^u_A(Z)
\\
\mathfrak{A}_1\otimes \mathfrak{M}^u_A (Z)
\ar[ur, "\sim"]
\ar[rr, twoheadrightarrow]
&&
\mathfrak{M}^u_A(Z)\oplus \mathfrak{M}^u_A(Z)
\ar[ur, equals]
\ar[from = uu, crossing over]
&
\end{tikzcd}
\end{equation}
with fibrations and quasi-isomorphisms as indicated.
The composite $\mathfrak{A}_1 \to N^\bullet(\Delta^1)$ of the Stokes map with the projection to normalised cochains has a section
\[
\varsigma \colon
N^\bullet(\Delta^1)
=
\bigg[
\begin{tikzcd}
\underset{0}{\mathbb{Q}\oplus\mathbb{Q}}
\ar[r, "{(a,b)\mapsto b-a\vphantom{\big(}}"]
&
\underset{1}{\mathbb{Q}}
\ar[r]
&
0
\ar[r]
&
\dotsb
\end{tikzcd}
\bigg]
\longrightarrow
\mathfrak{A}_1
\]
given in degrees $0$ and $1$ by the assignments
$(a,b)\mapsto at_0 + bt_1$ and $c \mapsto cdt_1$ respectively.
The section $\varsigma$ is a quasi-isomorphism and gives rise to a morphism of pullback squares of $A$-modules 
\begin{equation}
\label{eqn:tau2}
\begin{tikzcd}[sep = small]
&
P(Z)
\ar[rr]
\ar[dd]
&&
0
\ar[dd]
\\
\mathbb{Q}[-1]\otimes \mathfrak{M}^u_A (Z) 
\ar[rr, crossing over]
\ar[ur]
\ar[dd]
&&
0
\ar[ur, equals]
&
\\
&
\mathfrak{A}_1\otimes \mathfrak{M}^u_A (Z)
\ar[rr, twoheadrightarrow]
&&
\mathfrak{M}^u_A(Z)\oplus \mathfrak{M}^u_A(Z)
\\
N^\bullet(\Delta^1)\otimes \mathfrak{M}^u(Z)
\ar[ur, "\sim"]
\ar[rr, twoheadrightarrow]
&&
\mathfrak{M}^u_A(Z)\oplus \mathfrak{M}^u_A(Z)
\ar[ur, equals]
\ar[from = uu, crossing over]
&
\end{tikzcd}
\end{equation}
with fibrations and quasi-isomorphisms as indicated.
All objects of $A\mathrm{-Mod}$ are fibrant, hence by Kan's cube lemma (see \cite[Lemma 5.2.6]{hovey_model_1999}) the top left morphisms of the diagrams \eqref{eqn:tau1} and \eqref{eqn:tau2} are quasi-isomorphisms.
The composite $\tau_Z\colon \mathfrak{M}^u_A(Z)[-1]\to \mathfrak{M}^u_A(\Sigma_X Z)$ is thus a quasi-isomorphism for any retractive space $Z$ over $X$.
The naturality of $\tau_Z$ in $Z$ follows from Lemma \ref{lem:FinTensorModuleCompare}.
\end{proof}

\begin{remark}
Suppose that $p\colon Z\to X$ is a nilpotent fibration of nilpotent spaces of finite rational type.
Supposing that $A$ is a Sullivan model for the rational homotopy type of $X$, then $p$ is modelled in $\mathrm{cDGA}$ by a minimal morphism $A\to \widetilde{A}(Z) := A\otimes \Lambda V$.
If $Z$ is a retractive space over $X$ then $p$ admits a distinguished section, modelled in $\mathrm{cDGA}$ by the morphism $A\otimes \Lambda V\to A$ sending $V$ to $0$.
In this situation, F\'{e}lix, Murillo and Tanr\'{e} have shown \cite{felix_fibrewise_2010} that the fibrewise suspension $\Sigma_X Z\to X$ is modelled by
\[
A\longrightarrow
\widetilde{\mathfrak{A}}(\Sigma_X Z) := A\otimes \big(\mathbb{Q}\oplus (\Lambda^{\geq 1} V)[-1]\big)
\] 
where the product of any pair of elements in $(\Lambda^{\geq 1} V)[-1]$ is zero, and if $v\in \Lambda^{\geq 1} V$ has differential $dv = \sum_i a_i \otimes w_i$ in $A\otimes \Lambda V$, then $d(v[-1]) = \sum_i (-1)^{|a_i|+1} a_i \otimes w_i[-1]$ in $\widetilde{\mathfrak{A}}(\Sigma_X Z)$.

Under the simplifying assumption $X=\mathfrak{S}(A)$, there are quasi-isomorphisms of $A$-modules $\mathrm{aug}_A \widetilde{\mathfrak{A}}(Z) = A\otimes \Lambda^{\geq 1}V \to \mathfrak{M}^u_A (Z)$ and $\mathrm{aug}_A \widetilde{\mathfrak{A}}(\Sigma_X Z) = A\otimes (\Lambda^{\geq 1}V)[-1] \to \mathfrak{M}^u_A (\Sigma_X Z)$.
Taking cohomology we get an isomorphism $H^\bullet(A\otimes (\Lambda^{\geq 1}V)[-1])\cong H^{\bullet+1}(A\otimes \Lambda^{\geq 1} V))$, in accordance with the Lemma.
\end{remark}

\section{Fibrewise stable rational homotopy}
\label{S:FibStabRat}
We have so far established adjunctions \eqref{eqn:UnstableAdj} relating retractive spaces and modules and examined how these adjunctions behave with respect to fibrewise stabilisations.
In this section, we build upon these results to give an algebraic characterisation of the rational homotopy theory of nilpotent parametrised spectra over a nilpotent base space.

In detail, for each cdga $A$ we formulate a Quillen adjunction
\begin{equation}
\label{eqn:StableAdj}
(\mathfrak{M}_A\dashv \mathfrak{P}_A)\colon
\begin{tikzcd}
\mathrm{Sp}^\mathbb{N}_{\mathfrak{S}(A)}
\ar[rr, shift left =2]
\ar[rr, leftarrow, shift left =-2, "\bot"]
&&
A\mathrm{-Mod}^\mathrm{op}
\end{tikzcd}
\end{equation}
which is pseudonatural in $A$ and so covers the Sullivan--de Rham adjunction in a certain sense.
These Quillen adjunctions induce derived adjunctions of homotopy categories and the main result of this article, Theorem \ref{thm:RatParamHomThry}, states that if $A$ is cofibrant, connected, and of finite homotopical type the derived adjunctions restrict to an equivalence of categories
\[
\begin{tikzcd}
Ho(\mathrm{Sp}_{\mathfrak{S}(A)})^\mathbb{Q}_{\mathrm{f.t., nil, bbl}}
\ar[rr, shift left =2]
\ar[rr, leftarrow, shift left =-2, "\simeq"]
&&
Ho(A\mathrm{-Mod})^\mathrm{op}_{\mathrm{f.h.t.}}
\end{tikzcd}
\]
Here we have that:
\begin{itemize}
  \item $Ho(A\mathrm{-Mod})^\mathrm{op}_{\mathrm{f.t., bbl}}$ is the full subcategory spanned by $A$-modules of \emph{finite homotopical type}, namely those $A$-modules admitting a bounded below minimal model which is finitely generated in each degree (see Definition \ref{defn:FHT} below).
  
  \item $Ho(\mathrm{Sp}_{ \mathfrak{S}(A)})^\mathbb{Q}_{\mathrm{f.t., nil, bbl}}$ is the full category of the rational homotopy theory of $\mathfrak{S}(A)$-spectra spanned by rational parametrised that are:
  \begin{enumerate}[label=(\roman*)]
    \item of finite type (e.g.~the fibrewise rational stable homotopy groups are finite dimensional vector spaces in each degree);
    \item nilpotent (e.g.~$\pi_1 \mathfrak{S}(A)$ acts nilpotently on fibrewise stable homotopy as in Section \ref{SS:NilParamSpec}); and
    \item bounded below (e.g.~the fibrewise stable homotopy groups vanish in sufficiently large negative degree).
  \end{enumerate}
\end{itemize}
Under the above conditions, the $A$-module $\mathfrak{M}_A(P)$ associated to any cofibrant $\mathfrak{S}(A)$-spectrum $P$ computes the $H^\bullet(\mathfrak{S}(A))$-action on the rational cohomology groups of the pushforward $\mathfrak{S}(A)_! P$ (Proposition \ref{prop:TopInterp}).
Our result states that under some mild conditions, this $A$-module $\mathfrak{M}_A(P)$ completely determines $P$ up to fibrewise rational homotopy equivalence.

Throughout this section we work with respect to a fixed cdga $A$ and we $X:= \mathfrak{S}(A)$ for the spatial realisation.

\subsection{From spectra to modules}
In this section we construct the stabilised adjunction $(\mathfrak{M}_A\dashv \mathfrak{P}_A)$ relating sequential $X$-spectra with unbounded $A$-modules and we give a topological interpretation the $A$-module $\mathfrak{M}_A$.
For this we require the following auxiliary notion:
\begin{definition}
An \emph{$A$-module cospectrum} is a sequence of connective $A$-modules $N_0, N_1, N_2, \dotsc$ equipped with morphisms of $A$-modules $\varsigma_n \colon N_{n+1} \to N_n[-1]$ for each $n\geq 0$.
A morphism of $A$-module cospectra $f\colon N\to M$ is a sequence of $A$-module morphisms $f_n\colon N_n\to M_n$ commuting with the structure maps $\varsigma_n$.
The category of $A$-module cospectra is written $\mathrm{CoSp}(A\mathrm{-Mod}_{\geq 0})$.
\end{definition}
\begin{remark}
To ease the notation, we write $s$ for the shift functor on $A$-modules $M\mapsto M[-1]$.
The shift functor $s$ is an auto-equivalence of $A\mathrm{-Mod}$ with inverse equivalence $\ell\colon M\mapsto M[1]$.
In the usual fashion, $s$ and $\ell$ determine adjoint equivalences of categories $(s\dashv \ell)$ and $(\ell\dashv s)$.

While the restriction of $s$ to $A\mathrm{-Mod}_{\geq 0}$ is no longer an auto-equivalence, it has a left adjoint given in terms of the reflection adjunction $(\mathrm{cn}^0_A\dashv i)\colon A\mathrm{-Mod}\to A\mathrm{-Mod}_{\geq 0}$ by $\ell' := \mathrm{cn}^0_A \circ \ell\circ i$.
It is clear that if $M$ is $1$-connective, i.e.~ $M^k = 0$ for $k\leq 0$, then $\ell' M = \ell M$.
\end{remark}

By Lemma \ref{lem:SuspComparison} there is a natural weak equivalence of functors $\mathrm{sSet}_{\dslash X}\to A\mathrm{-Mod}^{\mathrm{op}}_{\geq 0}$
\[
\tau\colon \mathfrak{M}^u_A \circ \Sigma_X
\xrightarrow{\;\;\cong\;\;} s\circ \mathfrak{M}^u_A\,.
\]
Consider the \lq\lq dual'' natural transformation $D\tau\colon \ell \circ \mathfrak{P}^u_A\to \Omega_X\circ \mathfrak{P}^u_A$ with components determined by the adjointness diagram
\[
\begin{tikzcd}[row sep=tiny]
A\mathrm{-Mod}^\mathrm{op}_{\geq 0}\big(s\mathfrak{M}^u_A (Z), M\big)
\ar[d, equals, "\wr"]
\ar[r, "\tau_Z"]
&
A\mathrm{-Mod}^\mathrm{op}_{\geq 0}\big(\mathfrak{M}^u_A (\Sigma_X Z), M\big)
\ar[d, equals, "\wr"]
\\
\mathrm{sSet}_{\dslash X}\big(Z, \mathfrak{P}^u_A (\ell' M)\big)
\ar[r, "D\tau_M"']
&
\mathrm{sSet}_{\dslash X} \big(Z, \Omega_X \mathfrak{P}^u_A(M)\big)
\end{tikzcd}
\]
Using $D\tau$ we construct a functor $\mathfrak{P}^\mathbb{N}_A\colon \mathrm{CoSp}(A\mathrm{-Mod}_{\geq 0})^\mathrm{op}\to \mathrm{Sp}^\mathbb{N}_X$ as follows.
For an $A$-module cospectrum $N$, the structure maps $N_{n+1} \to sN_n$ are adjoint to maps $\ell' N_{n+1}\to N_n$.
We set $\mathfrak{P}^\mathbb{N}_A(N)_n := \mathfrak{P}^u_A(N_n)$, which determines a sequential $X$-spectrum with  structure maps adjunct to 
\[
\begin{tikzcd}
\mathfrak{P}^u_A(N_n)
\ar[r]
&
\mathfrak{P}^u_A(\ell' N_n)
\ar[r, "D\tau_{N_n}"]
&
\Omega_X \mathfrak{P}^u_A (N_n)\,.
\end{tikzcd}
\]
This assignment on objects determines a functor $\mathfrak{P}^\mathbb{N}_A\colon \mathrm{CoSp}(A\mathrm{-Mod}_{\geq 0})^\mathrm{op}\to \mathrm{Sp}^\mathbb{N}_X$ which has a left adjoint $\mathfrak{M}^\mathbb{N}_A\colon \mathrm{Sp}_X^\mathbb{N}\to \mathrm{CoSp}(A\mathrm{-Mod}_{\geq 0})^\mathrm{op}$ that is somewhat less straightforward to describe.
Given a sequential $X$-spectrum $P$, iterating the natural weak equivalence $\tau$ gives rise to weak equivalences of $A$-modules
\[
\tau^p\colon  s^p \mathfrak{M}^u_A (P_q)\longrightarrow 
\mathfrak{M}^u_A (\Sigma_X^p P_q)
\]
for all $p,q\geq 0$.
Define $\mathfrak{M}^\mathbb{N}_A(P)_n$ to be the $A$-module given as the equaliser of the diagram
\begin{equation}
\label{eqn:EqualiserforCospec}
\begin{tikzcd}
\displaystyle\prod_{p+q=n}
s^p\mathfrak{M}^u_A(P_q)
\ar[r, shift left =1]
\ar[r, shift left=-1]
&
\displaystyle\prod_{p+q+r=n}
s^p \mathfrak{M}^u_A(\Sigma_X^q P_r)
\end{tikzcd}
\end{equation}
where the top arrow is induced by applying $s^p \mathfrak{M}^u_A$ to the iterated structure maps $\Sigma_X^p P_r\to P_{p+r}$ and the bottom arrow is induced by the maps $s^p\tau^q \colon s^{p+q} \mathfrak{M}^u_A (P_r)\to s^p \mathfrak{M}^u_A (\Sigma_X^q P_r)$.
The structure maps $\mathfrak{M}^\mathbb{N}_A(P)_{n+1} \to s\mathfrak{M}^\mathbb{N}_A(P)_{n}$ are obtained by viewing the equaliser diagram for $s\mathfrak{M}^\mathbb{N}_A(P)_{n}$ as the subdiagram of the equaliser diagram for $\mathfrak{M}^\mathbb{N}_A(P)_{n+1}$ consisting of all terms with at least one $s$.
A tedious but straightforward argument now proves the following
\begin{lemma}
\label{lem:SpecCoSpec}
For each cdga $A$ there is an adjunction
\[
(\mathfrak{M}^\mathbb{N}_A\dashv\mathfrak{P}^\mathbb{N}_A )\colon
\begin{tikzcd}
\mathrm{Sp}^\mathbb{N}_{\mathfrak{S}(A)}
\ar[rr, shift left =2]
\ar[rr, shift left=-2, leftarrow, "\bot"]
&&
\mathrm{CoSp}(A\mathrm{-Mod}_{\geq 0})^\mathrm{op}
\end{tikzcd}
\]
such that $\mathfrak{P}^\mathbb{N}_A(N)_n = \mathfrak{P}^u_A (N_n)$ for any $A$-module cospectrum $N$.
\end{lemma}

Any (not necessarily connective) $A$-module $M$ gives rise to an $A$-module cospectrum $\mathfrak{d}(M)$ as follows.
The $n$-th term of $\mathfrak{d}(M)$ is $\mathfrak{d}(M)_n := \mathrm{cn}^0_A (M[-n])$ with structure maps arising from the factorisations
\begin{equation}
\label{eqn:CoSpMod}
\begin{tikzcd}[row sep=small]
M[-(n+1)] 
\ar[rr]
\ar[dr]
&&
(\mathrm{cn}^0_A M[-n])[-1] = s (\mathrm{cn}^0_A M[-n])
\\
&
\mathrm{cn}^0_A M[-(n+1)]
\ar[ur, dashed, "\exists"']
&
\end{tikzcd}
\end{equation}
which exist since $(\mathrm{cn}^0_A M[-n])[-1]$ is connective.
The resulting functor $\mathfrak{d}$ participates in an adjunction
\[
(\mathfrak{d}\dashv \mathfrak{l})\colon
\begin{tikzcd}
A\mathrm{-Mod}
\ar[rr, shift left =2]
\ar[rr, shift left =-2, leftarrow, "\bot"]
&&
\mathrm{CoSp}(A\mathrm{-Mod}_{\geq 0})
\end{tikzcd}
\]
where $\mathfrak{l}$ sends an $A$-module cospectrum $N$ to the limit of the tower of $A$-modules 
\[
\mathfrak{l}(N) :=
\mathrm{lim}
\Big(
\begin{tikzcd}[sep=small]
\dotsb
\ar[r]
&
N_{n+1}[n+1]
\ar[r]
&
N_n[n]
\ar[r]
&
\dotsb
\ar[r]
&
N_0
\end{tikzcd}
\Big)
\]
obtained by tensoring each $N_{n+1}\to N_n[-1]$ with $\mathbb{Q}[n+1]$.
We leave it to the reader to perform the easy check that $\mathfrak{l}$ and $\mathfrak{d}$ are adjoints.

For any cdga $A$, by composing the adjunction of Lemma \ref{lem:SpecCoSpec} with (the opposite of) \eqref{eqn:CoSpMod} we get an adjunction relating sequential $X= \mathfrak{S}(A)$-spectra to unbounded $A$-modules:
\[
(\mathfrak{M}_A\dashv \mathfrak{P}_A)\colon
\begin{tikzcd}
\mathrm{Sp}^\mathbb{N}_{X}
\ar[rr, shift left =2]
\ar[rr, leftarrow, shift left=-2, "\bot"]
&&
A\mathrm{-Mod}^\mathrm{op}\,.
\end{tikzcd}
\]
This adjunction relates to fibrewise stabilisation in a natural way:
\begin{lemma}
\label{lem:SuspendInMod}
For each $k\geq 0$, the diagram of left adjoint functors
\[
\begin{tikzcd}
\mathrm{sSet}_{\dslash X}
\ar[r,"\mathfrak{M}^u_A"]
\ar[d, "\Sigma^{\infty-k}_X"']
&
A\mathrm{-Mod}_{\geq 0}^\mathrm{op}
\ar[d, "\ell^k \circ i"]
\\
\mathrm{Sp}^\mathbb{N}_X
\ar[r, "\mathfrak{M}_A"]
&
A\mathrm{-Mod}^\mathrm{op}
\end{tikzcd}
\]
commutes up to natural isomorphism.
\end{lemma}
\begin{proof}
We argue using the corresponding diagram of right adjoints, since adjoints are essentially unique.
The composite right adjoint $\widetilde{\Omega}_X^{\infty-k}\circ \mathfrak{P}_A$ sends an $A$-module $M$ to the $k$-th space of $\mathfrak{P}_A(M)$ of the sequential $X$-spectrum $\mathfrak{P}_A(M)$, which is $\mathfrak{P}^u_A(\mathrm{cn}^0_A(M[-k]))$.
As this is true for all $M$, we conclude that there is a natural isomorphism of functors $\widetilde{\Omega}_X^{\infty-k}\circ \mathfrak{P}_A\cong \mathfrak{P}^u_A \circ \mathrm{cn}^0_A \circ s^k$.
Since $\mathrm{cn}^0_A \circ s^k$ is \emph{right} adjoint to $\ell^k\circ i\colon A\mathrm{-Mod}_{\geq 0}^\mathrm{op}\to A\mathrm{-Mod}^\mathrm{op}$ the result is proven.
\end{proof}

We are now in a position to prove the first main result of this section.
\begin{theorem}
For any cdga $A$ the adjunction
\[
(\mathfrak{M}_A\dashv \mathfrak{P}_A)\colon
\begin{tikzcd}
\mathrm{Sp}^\mathbb{N}_{\mathfrak{S}(A)}
\ar[rr, shift left =2]
\ar[rr, leftarrow, shift left =-2, "\bot"]
&&
A\mathrm{-Mod}^\mathrm{op}
\end{tikzcd}
\]
is Quillen.
\end{theorem}
\begin{proof}
Writing $X = \mathfrak{S}(A)$ as usual,
 we first prove that the adjunction is Quillen for the projective model structure on sequential $X$-spectra and then show that $\mathfrak{M}_A$ sends the localising set $\mathcal{S}_{\mathbb{N},X}$ \eqref{eqn:LocClass} to weak equivalences and hence descends to the stable model structure.
 
For the first step, we must check that the generating sets  $\mathcal{I}^\mathbb{N}_X$ and $\mathcal{J}^\mathbb{N}_X$ of Remark \ref{rem:SeqSpecCofibGen} are sent to cofibrations and acyclic cofibrations in $A\mathrm{-Mod}^\mathrm{op}$, respectively.
For any simplex $\sigma\colon \Delta^n \to X$ and $k\geq 0$ we have
\[
\mathfrak{M}_A \big(\Sigma^{\infty-k}_X i_n (\sigma_{X+}) \big) 
\cong
\left[
\begin{tikzcd}
\mathfrak{A}_n[k]
\ar[r]
&
\mathfrak{A}(\partial\Delta^n)[k]
\end{tikzcd}
\right]
\] 
by Lemma \ref{lem:SuspendInMod} and Example \ref{exam:FreeFibBasePointModule}; the map $\mathfrak{A}_n [k]\to \mathfrak{A}(\partial\Delta^n)[k]$ is the epimorphism gotten by applying the Sullivan--de Rham functor $\mathfrak{A}$ to the boundary inclusion $\partial\Delta^n \to \Delta^n$, regarded as a map of $A$-modules by restricting along the maps of cdgas $A\to \mathfrak{A}(X)\to \mathfrak{A}_n \to \mathfrak{A}(\partial\Delta^n)$ and then shifting.
Similarly, for any $\sigma\colon \Delta^n \to X$, $k\geq 0$ and $0\leq l\leq k$, applying $\mathfrak{M}_A$ to $\Sigma^{\infty-k}_X h^n_l (\sigma_{X+})$ gives the map of $A$-modules $\mathfrak{A}_n[k]\to \mathfrak{A}(\Lambda^n_l)[k]$, which is an acyclic epimorphism.
Since $\mathfrak{M}_A$ sends morphisms in $\mathcal{I}^\mathbb{N}_X$ and $\mathcal{J}^\mathbb{N}_X$ to cofibrations and acyclic cofibrations in $A\mathrm{-Mod}^\mathrm{op}$, respectively, the adjunction $(\mathfrak{M}_A\dashv \mathfrak{P}_A)$ is Quillen for the projective model structure.

To complete the proof, we show that $\mathfrak{M}_A$ sends each morphism in the set $S_{\mathbb{N},X}$ \eqref{eqn:LocClass} to a weak equivalence.
For any retractive space $Z$ over $X$ and $k\geq 0$ we have 
\[
\mathfrak{M}_A \big(\Sigma^{\infty-(k+1)}_X \Sigma_X Z\big) \cong \mathfrak{M}^u_A \big(\Sigma_X Z\big)[k+1]
\quad
\mbox{ and }
\quad
\mathfrak{M}_A \big(\Sigma^{\infty-k}_X Z\big) \cong \mathfrak{M}^u_A \big(Z \big)[k]
\]
by Lemma \ref{lem:SuspendInMod}.
Carefully unwinding the definitions, we find that $\mathfrak{M}_A$ sends the map of retractive spaces $\zeta_{\mathbb{N},X}(Z)\colon\Sigma^{\infty-(k+1)}_X \Sigma_X Z\to \Sigma^{\infty-k}_X Z$ to the map of $A$-modules
\[
\begin{tikzcd}
{\mathfrak{M}^u_A(Z)[k]}
\ar[r, "{\tau_{Z}[k+1]}"]
&
{\mathfrak{M}^u_A(\Sigma_X Z)[k+1]}
\end{tikzcd}
\]
given by shifting the natural quasi-isomorphism $\tau$ of Lemma \ref{lem:SuspComparison}.
Thus every morphism in $S_{\mathbb{N},X}$ is sent by $\mathfrak{M}_A$ to a weak equivalence and hence $(\mathfrak{M}_A\dashv \mathfrak{P}_A)$ is Quillen for the stable model structure.
\end{proof}

The next result shows that the Quillen adjunctions $(\mathfrak{M}_A\dashv \mathfrak{P}_A)$ vary naturally in $A$.
\begin{proposition}
\label{prop:Spec2ModPNat}
For any morphism of cdgas $f\colon A\to B$ the diagram of left Quillen functors
\[
\begin{tikzcd}
\mathrm{Sp}^\mathbb{N}_{\mathfrak{S}(B)}
\ar[r,"\mathfrak{M}_B"]
\ar[d, "\mathfrak{S}(f)_!"']
&
B\mathrm{-Mod}^\mathrm{op}
\ar[d, "f^\ast"]
\\
\mathrm{Sp}^\mathbb{N}_{\mathfrak{S}(A)}
\ar[r, "\mathfrak{M}_A"]
&
A\mathrm{-Mod}^\mathrm{op}
\end{tikzcd}
\]
commutes up to natural isomorphism.
\end{proposition}
\begin{proof}
The proof is quite straightforward and consists mainly of stringing together various previous observations.
Write $X= \mathfrak{S}(A)$ and $Y = \mathfrak{S}(B)$ and observe that since the pushforward functor $\mathfrak{S}(f)_!\colon \mathrm{sSet}_{\dslash Y}\to \mathrm{sSet}_{\dslash X}$ preserves $\mathrm{sSet}_{\ast}$-tensors there is a natural isomorphism $\Sigma_X \circ \mathfrak{S}(f)_!\cong \mathfrak{S}(f)_! \circ \Sigma_Y$.
Propositions \ref{prop:SlicedSdRPNat} and (the opposite of) \ref{prop:Alg2ModPNat} together imply that for any $Z\in \mathrm{sSet}_{\dslash X}$, the comparison morphism of $A$-modules of Lemma \ref{lem:SuspComparison}
\[
\tau_{\mathfrak{S}(f)_!Z}\colon \mathfrak{M}^u_A(\mathfrak{S}(f)_! Z)[-1]\longrightarrow \mathfrak{M}^u_A (\Sigma_X \mathfrak{S}(f)_! Z)
\]
is naturally isomorphic to
\[
\tau_{Z}\colon \mathfrak{M}^u_B(Z)[-1]\longrightarrow \mathfrak{M}^u_B (\Sigma_Y Z)
\]
regarded as a morphism of $A$-modules by restriction along $f$.

Let $P$ be a sequential $Y$-spectrum.
Comparing with the equaliser diagram \eqref{eqn:EqualiserforCospec} defining the terms of the cospectrum $\mathfrak{M}^\mathbb{N}_B (P)$, we now have a natural isomorphism of $A$-module cospectra
\[
f^\ast \mathfrak{M}^\mathbb{N}_B (P) \cong \mathfrak{M}^\mathbb{N}_A (\mathfrak{S}(f)_! P)
\]
with $f^\ast$ the levelwise restriction of modules along $f\colon A\to B$ (which leaves underlying sequences of cochain complexes fixed).
Shifts and limits in $A\mathrm{-Mod}$ and $B\mathrm{-Mod}$ are reflected by the forgetful functors to $Ch$, so we get a natural isomorphism $f^\ast \mathfrak{M}_B (P)\cong \mathfrak{M}_A (\mathfrak{S}(f)_! P)$ as required.
\end{proof}

\begin{remark}
The assignments $A\mapsto \mathrm{Sp}^\mathbb{N}_{\mathfrak{S}(A)}$ and $A\mapsto A\mathrm{-Mod}^\mathrm{op}$ determine pseudofunctors $\mathrm{cDGA}^\mathrm{op}\to \mathbf{Model}$ valued in the (2,1)-category of model categories, left Quillen functors, and pseudonatural isomorphisms.
The commuting diagrams of  Proposition \ref{prop:Spec2ModPNat} comprise the components of a pseudonatural transformation $\mathrm{Sp}^\mathbb{N}_{\mathfrak{S}(-)} \Rightarrow (-)\mathrm{-Mod}^\mathrm{op}$.
\end{remark}

We now turn to the issue of giving a topological interpretation of the functor $\mathfrak{M}_A$.
Under certain conditions on $A$, Proposition \ref{prop:InterpretUModule} gives that the $A$-module $\mathfrak{M}^u_A(Z)$ associated to a retractive space $Z$ over $X$ computes the $H^\bullet(X)$-action on the reduced rational cohomology of the quotient $\widetilde{H}^\bullet(Z/X)$.
As we now explain, the stable version of this statement also holds true for $\mathfrak{M}_A$.

For an $X$-spectrum $P$, the pushforward $X_! P$  is an $X_+$-comodule spectrum (cf.~Remark \ref{rem:ComoduleStructure}).
Rationalising and dualising, this $X_+$-coaction induces an action of $H^\bullet(X)$ on the graded rational vector space $\mathrm{Hom}(\spi_\bullet X_!P, \mathbb{Q}) \cong H^\bullet(X_!P)$.
Under certain conditions of $A$, the cohomology of $\mathfrak{M}_A(P)$ with its $H^\bullet
(A)$-action computes this $H^\bullet(X)$-module.

Let $P$ be a sequential $X$-spectrum.
By Proposition \ref{prop:Spec2ModPNat}, the underlying rational cochain complex of $\mathfrak{M}_A(P)$ is naturally isomorphic to $\mathfrak{M}_\mathbb{Q}(X_! P)$.
If $P\in \mathrm{Sp}^\mathbb{N}_X$ is cofibrant, the cohomology of the latter cochain complex computes $\mathrm{Hom}(\spi_\bullet X_!P, \mathbb{Q})$ by the following
\begin{lemma}
\label{lem:RatEquivofCofibSpec}
For any cofibrant sequential spectrum $E$ there are natural isomorphisms 
\[
H^\bullet(\mathfrak{M}_\mathbb{Q}E)\cong \mathrm{Hom}(\spi_{\bullet}(E),\mathbb{Q})
\cong H^\bullet(E)\,.
\]
\end{lemma}
\begin{proof}
Let $E^i$ be the \lq\lq eventual suspension spectrum''
\[
E^i_n =
\begin{cases}
E_n &n\leq i
\\
\Sigma^{n-i} E_i & n > i
\end{cases}
\]
equipped with the evident structure maps.
As $E$ is cofibrant, so are each of the $E^i$ and there is a sequence of cofibrations
$
E_0
\to
\dotsb
\to
E_i
\to
E_{i+1}
\to \dotsb
$
with (homotopy) colimit $E$.
For each $i$ there is a stable weak equivalence of cofibrant sequential spectra $\Sigma^{\infty-i} E_i\to E^i$ and hence  by Lemma \ref{lem:SuspendInMod} a quasi-isomorphism $\mathfrak{M}_\mathbb{Q}E^i\to \mathfrak{M}_\mathbb{Q}(\Sigma^{\infty-i}E_i) \cong \mathfrak{M}^u_\mathbb{Q} E_i [i]$.
As $\mathfrak{M}_\mathbb{Q}$ sends cofibrations of sequential spectra to fibrations in $Ch$, we get
\[
\mathfrak{M}_\mathbb{Q}(E) \cong \underset{i}{\mathrm{lim}} \;\mathfrak{M}_\mathbb{Q}(E^i)
\cong 
\underset{i}{\mathrm{holim}} \;\mathfrak{M}_\mathbb{Q}(E^i)
\cong
\underset{i}{\mathrm{holim}} \;\mathfrak{M}^u_\mathbb{Q}(E_i)[i]\,.
\]
As in the proof of Proposition \ref{prop:InterpretUModule}, for each $i$ the Stokes map and projection to normalised cochains gives a quasi-isomorphism
\[
\mathfrak{M}_\mathbb{Q}( E_i)[i]\longrightarrow \widetilde{N}^\bullet(E_i) [i]
\]
to the (shifted) normalised cochain complex of rational chains away from the basepoint.

The rational cohomology of $E$ is calculated by the cochain complex $\widetilde{N}^\bullet(E)$ obtained as the limit of the sequence of graded duals to
\[
\begin{tikzcd}
\widetilde{N}_\bullet(P_i)[i]
\cong
\widetilde{N}_\bullet(S^1)\otimes \widetilde{N}_\bullet(P_i)[i+1]
\ar[r, "\nabla"]
&
\widetilde{N}_\bullet(S^1\wedge P_i)[i+1]
\ar[r]
&
\widetilde{N}_\bullet(P_{i+1})[i+1]
\end{tikzcd}
\]
where $\nabla$ is the shuffle map and the second arrow is the result of applying the reduced normalised chains functor to $\Sigma P_i \to P_{i+1}$ and then shifting.
By Lemma \ref{lem:SuspComparison} and naturality of the Stokes map, there is a commuting diagram of cochain complexes
\[
\begin{tikzcd}
\dotsb
\ar[r]
&
\mathfrak{M}^u_\mathbb{Q}(E_{i+1} )[i+1]
\ar[r]
\ar[d]
&
\mathfrak{M}^u_\mathbb{Q}(E_{i} )[i]
\ar[r]
\ar[d]
&
\dotsb
\ar[r]
&
\mathfrak{M}^u_\mathbb{Q}(E_{0})
\ar[d]
\\
\dotsb
\ar[r]
&
\widetilde{N}^\bullet (E_{i+1})[i+1]
\ar[r]
&
\widetilde{N}^\bullet (E_{i})[i]
\ar[r]
&
\dotsb
\ar[r]
&
\widetilde{N}^\bullet (E_0)
\end{tikzcd}
\]
in which each vertical arrow is a quasi-isomorphism.
Passing to homotopy limits in the horizontal direction and analysing the resulting morphism of Milnor-style short exact sequences, we obtain a quasi-isomorphism $\mathfrak{M}_\mathbb{Q}(E)\to \widetilde{N}^\bullet(E)$, which proves the claim.
\end{proof}
\begin{corollary}
\label{cor:RatEquiv}
For any cdga $A$, the left Quillen functor $\mathfrak{M}_A$ sends rational equivalences of cofibrant $\mathfrak{S}(A)$-spectra to quasi-isomorphisms.
\end{corollary}
\begin{proof}
Writing $X=\mathfrak{S}(A)$ as usual, we have that the functor $X_!\colon \mathrm{Sp}^\mathbb{N}_X\to \mathrm{Sp}^\mathbb{N}$ preserves rational equivalences (this can be viewed as a consequence of the projection formula and the fact that rationalisation is implemented by smashing with $H\mathbb{Q}$).
By Lemma \ref{lem:RatEquivofCofibSpec}, we now have that the composite left Quillen functor $\mathfrak{M}_\mathbb{Q}\circ X_!\colon \mathrm{Sp}^\mathbb{N}_X\to Ch^\mathrm{op}$ sends rational equivalences of cofibrant sequential $X$-spectra to cofibrations.
By Proposition \ref{prop:Spec2ModPNat}, $\mathfrak{M}_\mathbb{Q}\circ X_!$ is naturally isomorphic to the composite of $\mathfrak{M}_A$ with the forgetful functor $A\mathrm{-Mod}^\mathrm{op}\to Ch^\mathrm{op}$, which reflects quasi-isomorphisms of $A$-modules.
\end{proof}

We are now able to provide a topological interpretation of the $A$-module $\mathfrak{M}_A$.
\begin{proposition}
\label{prop:TopInterp}
Let $A$ be a cofibrant connected cdga of finite homotopical type and let $X= \mathfrak{S}(A)$.
For any cofibrant sequential $X$-spectrum $P$, the $H^\bullet(A)$-action on $H^\bullet(\mathfrak{M}_A(P))$ is naturally isomorphic to the $H^\bullet(X)$-action on $H^\bullet(X_!P) \cong \mathrm{Hom}(\spi_{\bullet} (X_!P), \mathbb{Q})$.
\end{proposition}
\begin{proof}
Similarly to the proof of Lemma \ref{lem:RatEquivofCofibSpec}, any sequential $X$-spectrum is the sequential homotopy colimit of a diagram of shifted fibrewise suspension spectra \cite[Lemma 2.19]{braunack-mayer_combinatorial_2020}.
That is, for each sequential $X$-spectrum $P$ there is a sequence $Z_0, Z_1, \dotsc$ of retractive spaces over $X$ such that $P\cong \mathrm{hocolim}_i\, \Sigma^{\infty-i}_X Z_i$.
If $P$ is cofibrant, we may take $Z_i = P_i$, the spaces in the underlying sequence of $P$, noting that this decomposition is compatible with pushforward functors.
The left Quillen functor $\mathfrak{M}_A$ converts homotopy colimits of sequential $X$-spectra into homotopy limits of $A$-modules, hence $\mathfrak{M}_A (P)\cong \mathrm{holim}_i \,\mathfrak{M}_A (\Sigma^{\infty-i}_X P_i)\cong
\mathrm{holim}_i \,\mathfrak{M}^u_A (P_i)[i]$ as $A$-modules (the latter equivalence is by Lemma \ref{lem:SuspendInMod}). 

On the other hand, applying the pushforward functor $X_!$ we get equivalences of $X_+$-comodule spectra $X_! P \cong \mathrm{hocolim}_i \, X_! \Sigma^{\infty-i}_X P_i\cong \mathrm{hocolim}_i \Sigma^{\infty-i} (P_i/X)$.
Taking normalised rational cochains gives an equivalence of cochain complexes $N^\bullet(X_! P)\cong \mathrm{holim}_i \widetilde{N}^\bullet(P_i/X)[i]$
respecting $N^\bullet(X)$-actions.
As in the proof of Proposition \ref{prop:InterpretUModule}, by combining the natural $A_\infty$-quasi-isomorphism $\mathfrak{A}(Y)\rightsquigarrow C^\bullet(Y)\to N^\bullet(Y)$ (cf.~Remark \ref{rem:StokesMap}) with Lemma \ref{lem:SuspComparison} we get a commuting diagram of cochain complexes
\[
\begin{tikzcd}
\dotsb
\ar[r]
&
\mathfrak{M}^u_A(P_{i+1})[i+1]
\ar[r]
\ar[d]
&
\mathfrak{M}^u_A(P_{i})[i]
\ar[r]
\ar[d]
&
\dotsb
\ar[r]
&
\mathfrak{M}^u_A(P_{0})
\ar[d]
\\
\dotsb
\ar[r]
&
\widetilde{N}^\bullet (P_{i+1}/X)[i+1]
\ar[r]
&
\widetilde{N}^\bullet (P_{i}/X)[i]
\ar[r]
&
\dotsb
\ar[r]
&
\widetilde{N}^\bullet (P_0/X)\,.
\end{tikzcd}
\]
In this diagram, all vertical arrows are quasi-isomorphisms, horizontal arrows in the top row are $A$-module homomorphisms, horizontal arrows in the bottom row commute with $N^\bullet(X)$-actions, and all arrows are morphisms of $A_\infty$-$\mathfrak{A}(X)$-modules.
Passing to homotopy limits in the horizontal direction gives a quasi-isomorphism of cochain complexes $\mathrm{holim}_i \mathfrak{M}^u_A(P_i)[i]\to \mathrm{holim}_i \widetilde{N}^\bullet(P_i/X)[i]$ compatible with $A_\infty$-$\mathfrak{A}(X)$-module structures.

The hypotheses on $A$ guarantee that $A\to \mathfrak{A}(X)$ is a quasi-isomorphism of cdgas.
Since $\mathfrak{M}_A (P)\cong
\mathrm{holim}_i \,\mathfrak{M}^u_A (P_i)[i]$ as $A$-modules, we conclude that the cohomology $H^\bullet(A)$-module $H^\bullet(\mathfrak{M}_A(P))$ computes the $H^\bullet(X)$-action on $H^\bullet(\mathrm{holim}_i \widetilde{N}^\bullet(P_i/X))$.
As $P$ is cofibrant, so too is $X_! P$ and hence  $H^\bullet(\mathrm{holim}_i \widetilde{N}^\bullet(P_i/X))\cong \mathrm{Hom}(\spi_{\bullet} (X_!P),\mathbb{Q})$ by Lemma \ref{lem:RatEquivofCofibSpec}.
\end{proof}

\begin{corollary}
\label{cor:ConnectiveUnderM}
For a cdga $A$, if $P$ is a $k$-connective cofibrant sequential $\mathfrak{S}(A)$-spectrum then the cohomology of the $A$-module $\mathfrak{M}_A(P)$ is bounded in degrees $\geq k$. 
\end{corollary}
\begin{proof}
The conditions on $A$ in Proposition \ref{prop:TopInterp} were required to guarantee that $A\to \mathfrak{A}\mathfrak{S}(A)$ is a quasi-isomorphism.
If $A$ is \emph{any} cdga, and $P$ is a cofibrant sequential $\mathfrak{S}(A)$-spectrum,  the previous proof shows that $H^\bullet(\mathfrak{M}_A(P))$ is naturally isomorphic to $\mathrm{Hom}(\spi_{\bullet}(\mathfrak{S}(A)_! P), \mathbb{Q})$.
The result now follows from Corollary \ref{cor:PfwdConnectivity}.
\end{proof}

There is also a sort of \lq\lq adjoint'' to the last result:
\begin{lemma}
\label{lem:ConnectivityIsPreserved}
Fix an augmented cdga $A$ and an integer $k$.
If $M$ is a $k$-connective $A$-module then $\mathbf{R}\mathfrak{P}_A(M)$ is a $k$-connective $\mathfrak{S}(A)$-spectrum.
\end{lemma}
\begin{proof}
Since $A$ is connected there is a canonical augmentation $a\colon A\to \mathbb{Q}$ which gives rise to a distinguished $0$-simplex $x$ of $X= \mathfrak{S}(A)$ which is unique up to homotopy.
Since $M$ is $k$-connective, there is a minimal model $A\otimes V\to M$ with $V$ a graded rational vector space concentrated in degrees $\geq k$.
By (the adjoint of) Proposition \ref{prop:Spec2ModPNat}, the fibre spectrum of $\mathbf{R}\mathfrak{P}_A(M)$ at $x\in X$ is
\[
x^\ast \mathbf{R}\mathfrak{P}_A(M)
\cong 
x^\ast \mathfrak{P}_A(A\otimes V)
\cong 
\mathfrak{P}_\mathbb{Q}(a_! (A\otimes V))
\cong
\mathfrak{P}_\mathbb{Q}(V)\,,
\]
where $a_! (A\otimes V) = \mathbb{Q}\bigotimes_{A} (A\otimes V) \cong V$ has trivial differential by minimality.
The adjoint of Lemma \ref{lem:SuspendInMod} over $\mathbb{Q}$ shows that for each $l\geq 0$
\[
\Omega^{\infty-l}\mathfrak{P}_\mathbb{Q}(V)
\cong \mathfrak{P}^u_\mathbb{Q}(\mathrm{cn}^0 V[-l])
\cong \mathfrak{S}(\Lambda \,\mathrm{cn}^0 V[-l])\,,
\]
where $\Lambda\, \mathrm{cn}^0 V[-l]$ is the free cdga generated by the graded rational vector space $\mathrm{cn}^0 V[-l]$ (so that all generators are closed).
By Remark \ref{rem:CohomotopytoHomotopy}, we have group isomorphisms
\[
\spi_{n} \mathfrak{P}_\mathbb{Q}(V) \cong \pi_{n+l} \Omega^{\infty-l}\mathfrak{P}_\mathbb{Q}(V) \cong \mathrm{Hom}((\mathrm{cn}^0 V[-l])^{n+l}, \mathbb{Q}) \cong \mathrm{Hom}(V^{n}, \mathbb{Q})
\]
whenever $n+l\geq 2$,
from which it follows that $\mathfrak{P}_\mathbb{Q}(V) \cong HV^\vee$ is $k$-connective.
\end{proof}

\begin{corollary}
\label{cor:Spec2ModkConnective}
Let $A$ a cofibrant connected cdga of finite homotopical type.
For any integer $k$ the derived adjunction $(\mathbf{L}\mathfrak{M}_A\dashv \mathbf{R}\mathfrak{P}_A)$ restricts to an adjunction
\[
\begin{tikzcd}
Ho(\mathrm{Sp}^{\geq k}_{\mathfrak{S}(A)})
\ar[rr, shift left = 2, "\mathbf{L}\mathfrak{M}_A"]
\ar[rr, shift left =-2, leftarrow, "\bot", "\mathbf{R}\mathfrak{P}_A"']
&&
Ho(A\mathrm{-Mod}_{\geq k})^\mathrm{op}\,.
\end{tikzcd}
\]
\end{corollary}

The proof of Lemma \ref{lem:ConnectivityIsPreserved} also establishes the following
\begin{lemma}
\label{lem:FibofMinModSpec}
Let $A$ be a cdga with augmentation $a\colon A\to \mathbb{Q}$ corresponding to $x\colon \ast \to X=\mathfrak{S}(A)$.
For a minimal $A$-module $A\otimes V$, the fibre spectrum $x^\ast \mathfrak{P}_A (A\otimes V)$ is stably equivalent to the Eilenberg--Mac Lane spectrum $HV^\vee$ on the graded dual $(V^\vee)^k := (V^k)^\vee$ of $V$.
\end{lemma}

\subsection{Strict algebraic models for rational parametrised spectra}
In the previous section we obtained, for any cdga $A$, a Quillen adjunction
\[
(\mathfrak{M}_A\dashv \mathfrak{P}_A)\colon
\begin{tikzcd}
\mathrm{Sp}^\mathbb{N}_{\mathfrak{S}(A)}
\ar[rr, shift left =2]
\ar[rr, leftarrow, shift left =-2, "\bot"]
&&
A\mathrm{-Mod}^\mathrm{op}
\end{tikzcd}
\]
relating $\mathfrak{S}(A)$-spectra and unbounded $A$-modules.
The left Quillen functor $\mathfrak{M}_A$ sends rational equivalences of cofibrant $\mathfrak{S}(A)$-spectra to quasi-isomorphisms (by Corollary \ref{cor:RatEquiv}) and hence the derived adjunction factors through the rational homotopy category of $\mathfrak{S}(A)$-spectra:
\begin{equation}
\label{eqn:DerivedAdj}
\begin{tikzcd}
Ho(\mathrm{Sp}_{\mathfrak{S}(A)})
\ar[rr, shift left =2]
\ar[rr, leftarrow, shift left =-2, "\bot"]
&&
Ho(\mathrm{Sp}_{\mathfrak{S}(A)})^\mathbb{Q}
\ar[rr, shift left =2]
\ar[rr, leftarrow, shift left =-2, "\bot"]
&&
Ho(A\mathrm{-Mod})^\mathrm{op}\,.
\end{tikzcd}
\end{equation}
By Sullivan's rational homotopy theory, if $A$ is cofibrant, connected, and of finite homotopical type it is an algebraic model for the rational homotopy type of the space $\mathfrak{S}(A)$.
Under these conditions on $A$, we prove that the derived adjunction \eqref{eqn:DerivedAdj} restricts to an equivalence between certain subcategories.
\begin{definition}
For a connected space $X$, an $X$-spectrum $P$ is said to be 
\begin{itemize}
  \item \emph{finite rational type} if for each $x\in X$, the graded vector space $\spi_\ast x^\ast P\otimes_\mathbb{Z} \mathbb{Q}$ is finite dimensional in each degree, and
  \item \emph{bounded below} if there is some integer $k$ for which $P_{\leq k} \cong 0_X$ is the zero $X$-spectrum.
\end{itemize}
We write $Ho(\mathrm{Sp}_X)^\mathbb{Q}_{\mathrm{f.t.,nil,bbl}}\hookrightarrow Ho(\mathrm{Sp}_X)^\mathbb{Q}$ for the full  subcategory spanned by the rational $X$-spectra that are nilpotent, bounded below, and of finite rational type.
\end{definition}
\begin{lemma}
\label{lem:NilFinRatType}
For a connected space $X$, an $X$-spectrum $P$ is nilpotent and of finite rational type if and only if it is rationally equivalent to an $X$-spectrum $Q$ for which for all $k$, the zero morphism $Q_{=k}\to 0_X$ factors as a finite sequence of extensions by $X$-spectra of the form $X^\ast (\Sigma^k (\bigoplus_{i=1}^n H\mathbb{Q}))$ for finite $n$.
\end{lemma}
\begin{proof}
Using the characterisation of nilpotent $X$-spectra of Theorem \ref{thm:Nil}, $P$ is nilpotent if and only if for each $k$ the zero morphism $P_{=k}\to 0_X$ factors as a finite sequence of extensions
\[
\begin{tikzcd}[row sep=small, column sep =tiny]
P_{=k} = (P_{=k})_{n_k}
\ar[r]
&
(P_{=k})_{n_k-1}
\ar[r]
&
\dotsb
\ar[r]
&
(P_{=k})_{j}
\ar[r]
&
(P_{=k})_{j-1}
\ar[r]
&
\dotsb 
\ar[r]
&
(P_{=k})_{1}
\ar[r]
&
0_X
\\
X^\ast \Sigma^kHA_{n_k}
\ar[u]
&
X^\ast \Sigma^kHA_{n_k-1}
\ar[u]
&
\dotsb
&
X^\ast \Sigma^kHA_{j}
\ar[u]
&
X^\ast \Sigma^kHA_{j-1}
\ar[u]
&
\dotsb 
&
X^\ast \Sigma^kHA_{1}
\ar[u]
\end{tikzcd}
\]
Rationalisation is implemented in $\mathrm{Sp}_X$ by smashing fibrewise with $H\mathbb{Q}$ and it is easy to see that $(P\wedge_X H\mathbb{Q})_{=k} \cong P_{=k}\wedge_X H\mathbb{Q}$.
Recall that the pullback functor $X\colon \mathrm{Sp}\to \mathrm{Sp}_X$ is strong monoidal for (fibrewise) smash products and we have $HA\wedge H\mathbb{Q}\cong H(A\otimes_\mathbb{Z}\mathbb{Q})\cong \bigoplus_{i=1}^{l} H\mathbb{Q}$ with $l= \mathrm{rank}(A)$ not necessarily finite.
Forming the fibrewise smash product of the above diagram with $H\mathbb{Q}$ expresses $(P\wedge_X H\mathbb{Q})_{=k} \to 0_X$ as a finite sequence of extensions by $X$-spectra of the form $X^\ast (\Sigma^k (\bigoplus_{i=1}^l H\mathbb{Q}))$ (where $l$ may be infinite).
For any $x\in X$ we have
\[
\spi_k x^\ast P_{=k} \otimes_\mathbb{Z} \mathbb{Q}
\cong
\spi_k x^\ast (P\wedge_X H\mathbb{Q})_{=k}
\cong
\spi_k x^\ast P \otimes_\mathbb{Z}\mathbb{Q}
\]
from which it follows that $\spi_k x^\ast P\otimes_\mathbb{Z} \mathbb{Q}$ is finite dimensional if and only if each abelian group $A_{j}$ in appearing in the above decomposition has finite rank.

Thus an $X$-spectrum $P$ is nilpotent of finite rational type if and only if $P\wedge_X H\mathbb{Q}$ is nilpotent of finite rational type, which in turn is the case if and only if for all $k$, $(P\wedge_X H\mathbb{Q})_{=k}\to 0_X$ factors as a finite sequence of extensions by $X$-spectra of the form $X^\ast (\Sigma^k (\bigoplus_{i=1}^lH\mathbb{Q}))$ for finite $l$.
\end{proof}

\begin{corollary}
\label{cor:NilBblFinRatType}
For a connected space $X$, a nilpotent, bounded below $X$-spectrum $P$ of finite rational type is rationally equivalent to an $X$-spectrum $Q$ for which for all $k$, the zero map $Q_{\leq k}\to 0_X$ factors as a finite sequence of extensions by $X$-spectra of the form $X^\ast (\Sigma^l(\bigoplus_{i=1}^n H\mathbb{Q}))$ for finite $n$ and $l\leq k$.
\end{corollary}
\begin{proof}
If $P$ is a (non-zero) bounded below $X$-spectrum there is some maximal integer $k$ such that $P_{\leq (k-1)} = 0_X$ is the zero $X$-spectrum.
For any $l\geq k$ there is a commuting diagram of $X$-spectra
\[
\begin{tikzcd}[row sep=small]
P_{=l}
\ar[r]
\ar[ddddd]
&
0_X
\ar[d]
\\
&
P_{=(l-1)}
\ar[r]
\ar[dddd]
&
0_X
\ar[d]
\\
&&
\dotsb
\ar[r]
&
0_X
\ar[d]
&
&
\\
&&
&
P_{=(k+1)}
\ar[dd]
\ar[r]
&
0_X
\ar[d]
&
\\
&&&&
P_{=k}
\ar[r]
\ar[d, "\cong"]
&
0_X
\ar[d, "\cong"]
\\
P_{\leq l}
\ar[r]
&
P_{\leq (l-1)}
\ar[r]
&
\dotsb
\ar[r]
&
P_{\leq (k+1)}
\ar[r]
&
P_{\leq k}
\ar[r]
&
P_{\leq (k-1)}
\end{tikzcd}
\]
in which each rectangle is a pullback-pushout square.
By stability, if $P_{= i}\to 0_X$, say, factors as a sequence of extensions
\[
\begin{tikzcd}[row sep=small, column sep =tiny]
X^\ast \Sigma^iHA_{n_i}
\ar[d]
&
X^\ast \Sigma^iHA_{n_i-1}
\ar[d]
&
\dotsb
&
X^\ast \Sigma^iHA_{j}
\ar[d]
&
X^\ast \Sigma^iHA_{j-1}
\ar[d]
&
\dotsb 
&
X^\ast \Sigma^iHA_{1}
\ar[d]
\\
P_{=i} = (P_{=i})_{n_i}
\ar[r]
&
(P_{=i})_{n_i-1}
\ar[r]
&
\dotsb
\ar[r]
&
(P_{=i})_{j}
\ar[r]
&
(P_{=i})_{j-1}
\ar[r]
&
\dotsb 
\ar[r]
&
(P_{=i})_{1}
\ar[r]
&
0_X
\end{tikzcd}
\]
then by the above we have a diagram of $X$-spectra
\[
\begin{tikzcd}[row sep=small, column sep =tiny]
X^\ast \Sigma^iHA_{n_i}
\ar[d]
&
X^\ast \Sigma^iHA_{n_i-1}
\ar[d]
&
\dotsb
&
X^\ast \Sigma^iHA_{j}
\ar[d]
&
X^\ast \Sigma^iHA_{j-1}
\ar[d]
&
\dotsb 
&
X^\ast \Sigma^iHA_{1}
\ar[d]
\\
P_{=i} = (P_{=i})_{n_i}
\ar[r]
\ar[d]
&
(P_{=i})_{n_i-1}
\ar[r]
\ar[d]
&
\dotsb
\ar[r]
&
(P_{=i})_{j}
\ar[r]
\ar[d]
&
(P_{=i})_{j-1}
\ar[r]
\ar[d]
&
\dotsb 
\ar[r]
&
(P_{=i})_{1}
\ar[r]
\ar[d]
&
0_X
\ar[d]
\\
P_{\leq i} =: (P_{\leq i})_{n_i}
\ar[r]
&
(P_{\leq i})_{n_i-1}
\ar[r]
&
\dotsb
\ar[r]
&
(P_{\leq i})_{j}
\ar[r]
&
(P_{\leq i})_{j-1}
\ar[r]
&
\dotsb 
\ar[r]
&
(P_{\leq i})_{1}
\ar[r]
&
P_{\leq (i-1)}
\end{tikzcd}
\]
in which each square is a pullback-pushout and hence each \lq\lq long hook'' (subdiagram consisting of the composite of two downward arrows followed by a rightward arrow) is a fibre sequence.
Replacing $P$ by the rationally equivalent $P\wedge_X H\mathbb{Q}$, the result now follows from Lemma \ref{lem:NilFinRatType}.
\end{proof}

On the algebraic side, we are interested in bounded below modules obeying a certain finiteness condition.
Any bounded below module has a minimal model by Lemma \ref{lem:MinimalModuleExist}.
\begin{definition}
\label{defn:FHT}
For a cdga $A$, a bounded below $A$-module $M$ is \emph{of finite homotopical type} if some (hence any) minimal model $A\otimes V\to M$ is such that the graded vector space $V$ is degreewise finite dimensional.
We write $
Ho(A\mathrm{-Mod})_{\mathrm{f.h.t.}}\hookrightarrow Ho(A\mathrm{-Mod})$ for the full subcategory spanned by the bounded below $A$-modules that are of finite homotopical type.
\end{definition}

\begin{remark}
\label{rem:FHTCounterexample}
If $A$ is connected of finite homotopical type, then any bounded below $A$-module of finite homotopical type has finite dimensional cohomology in each degree.
The converse does not hold.
Consider the cdga $\Lambda \langle x\rangle$ spanned by a closed generator $x$ of degree $1$ (the minimal model for $S^1$).
Setting $x$ to zero gives rise to an augmentation $\Lambda \langle x \rangle\to \mathbb{Q}$, furnishing $\mathbb{Q}$ with the structure of a $\Lambda \langle x \rangle$-module. 
The minimal model for this $\Lambda \langle x \rangle$-module is 
\[
\rho\colon \Lambda\langle x\rangle \otimes \mathrm{span}\{v_n\mid n\geq 0\}\longrightarrow
\mathbb{Q}\,,
\]
where each of the $v_n$ has degree $0$, $dv_0 = 0$ and $dv_{n+1} = x\cdot v_n$ for $n\geq 0$, and $\rho$ sends $v_0$ to $1$ and sets all remaining generators $v_n$ to zero.
Thus the $\Lambda\langle x\rangle$-module $\mathbb{Q}$ is not of finite homotopical type.
This is related to the fact that the path fibration $PS^1 \to S^1$ is not nilpotent due to having disconnected fibres.
\end{remark}

We are now in a position to prove our main result.
\begin{theorem}
\label{thm:RatParamHomThry}
Let $A$ be a cofibrant connected cdga of finite homotopical type. 
There is an adjoint equivalence of categories
\[
\begin{tikzcd}
Ho\big(\mathrm{Sp}_{\mathfrak{S}(A)}\big)^\mathbb{Q}_{\mathrm{f.t.,nil,bbl}}
\ar[rr, shift left =2, "\mathbf{L}\mathfrak{M}_A"]
\ar[rr, leftarrow, shift left =-2, "\simeq", "\mathbf{R}\mathfrak{P}_A"']
&&
Ho\big(A\mathrm{-Mod})^\mathrm{op}_\mathrm{f.h.t.}
\end{tikzcd}
\]
\end{theorem}
\begin{proof}
Write $X = \mathfrak{S}(A)$.
By our assumptions on $A$, there is a canonical augmentation $a\colon A\to \mathbb{Q}$ corresponding to a $0$-simplex $x$ of $X$ which is unique up to homotopy.
Our strategy is to first verify that the components at $X^\ast H\mathbb{Q}$ and $A$ of the derived unit and counit of the $(\mathfrak{M}_A\dashv \mathfrak{P}_A)$-adjunction are, respectively, a rational equivalence of $X$-spectra and quasi-isomorphism of $A$-modules.
Using that both of the derived functors $\mathbf{L}\mathfrak{M}_A$ and $\mathbf{R}\mathfrak{P}_A$ commute with shifts and extensions, we prove the result by an argument over \lq\lq Postnikov towers'' (cf.~Section \ref{ss:Postnikov}).

\paragraph{Derived unit at $X^\ast H\mathbb{Q}$:} There is an equivalence of $X$-spectra between the pullback $X$-spectrum $X^\ast S$ and the fibrewise stabilisation $\Sigma^\infty_{X+} X$ of the identity $X\to X$ (Example \ref{exam:FibrewiseSuspensionSpectra}).
The sequential $X$-spectrum $\Sigma^\infty_{X+} X$ is cofibrant and is sent by $\mathfrak{M}_A$ to $A$, regarded as a module over itself in the obvious way (Lemma \ref{lem:SuspendInMod}).
Regarded as an $A$-module, $A$ is cofibrant and so the component of the derived unit at $X^\ast S$ is presented in $\mathrm{Sp}^\mathbb{N}_{X}$ by the ordinary unit
\begin{equation}
\label{eqn:xastsUnit}
\eta\colon 
\Sigma^\infty_{X+} X
\longrightarrow
\mathfrak{P}_A \mathfrak{M}_A (\Sigma^\infty_{X+} X)\cong
\mathfrak{P}_A (A)\,.
\end{equation}
As $A\in A\mathrm{-Mod}^\mathrm{op}$ is fibrant, $\mathfrak{P}_A(A)$ is a fibrant $\Omega_X$-spectrum with fibre spectrum at $x$ given by
\[
x^\ast \mathfrak{P}_A (A) \cong \mathfrak{P}_\mathbb{Q} (a_! A) \cong \mathfrak{P}_\mathbb{Q} (\mathbb{Q})\,,
\]
using the adjoint of Proposition \ref{prop:Spec2ModPNat}.
The adjoint of Lemma \ref{lem:SuspendInMod} over $\mathbb{Q}$ shows that $\Omega^{\infty-k}\mathfrak{P}_\mathbb{Q}(\mathbb{Q}) \cong \mathfrak{P}^u_\mathbb{Q}(\mathbb{Q}[-k])$.
Unwinding the definitions, we get that $\mathfrak{P}^u_\mathbb{Q}(\mathbb{Q}[-k]) \cong \mathfrak{S}(\Lambda \langle x_k \rangle)$ is the result of applying the spatial realisation functor $\mathfrak{S}$ to the minimal cdga $\Lambda \langle x_k \rangle$ with a single closed generator in degree $k$.
For $n\geq 2$ we get $\pi_n\Omega^{\infty-k}\mathfrak{P}_\mathbb{Q}(\mathbb{Q})\cong \pi_n K(\mathbb{Q}, k)$ (cf.~Remark \ref{rem:CohomotopytoHomotopy}) from which it follows that $x^\ast \mathfrak{P}_A(A)\cong \mathfrak{P}_\mathbb{Q}(\mathbb{Q})\cong H\mathbb{Q}$.
We further observe that extension of scalars along the unit $u\colon \mathbb{Q}\to A$ sends $\mathbb{Q}$ to $A$ and hence
\[
\mathfrak{P}_A (A)= \mathfrak{P}_A (u_! \mathbb{Q})
\cong
X^\ast \mathfrak{P}_\mathbb{Q} (\mathbb{Q})\cong X^\ast H\mathbb{Q}
\]
by (the adjoint of) Proposition \ref{prop:Spec2ModPNat}.
Since the derived unit \eqref{eqn:xastsUnit} is non-zero, it is necessarily rationally equivalent to the map of $X$-spectra
$
X^\ast S
\to
X^\ast H\mathbb{Q}
$
obtained by fibrewise rationalisation.
Since $\mathfrak{M}_A$ sends rational equivalences of cofibrant $X$-spectra to quasi-isomorphisms (Corollary \ref{cor:RatEquiv}), this argument shows that the derived unit at $X^\ast H\mathbb{Q}$ is a rational equivalence of $X$-spectra.

\paragraph{Derived counit at $A$:} Regarded as an object of $A\mathrm{-Mod}^\mathrm{op}$, $A$ is fibrant.
By the above, the derived unit at $\Sigma^\infty_{X+} X$ is a rational equivalence of $X$-spectra
\[
\Sigma^\infty_{X+}X\longrightarrow \mathfrak{P}_A(A)
\]
with cofibrant domain.
Since $\mathfrak{M}_A$ sends rational equivalences of cofibrant $X$-spectra to quasi-isomorphisms of $A$-modules, the derived counit at $A$ is equivalent to the map of $A$-modules
\[
A
\longrightarrow
\mathfrak{M}_A\mathfrak{P}_A (A)
\longrightarrow
\mathfrak{M}_A (\Sigma^\infty_{X+} X)
\cong A\,,
\]
where we have used Lemma \ref{lem:SuspendInMod}.
Since the derived counit at $A$ is non-zero, by degree reasons the above map of $A$-modules must be isomorphic to the identity.
Thus the derived counit at $A$ is a quasi-isomorphism of $A$-modules.

\paragraph{Nilpotent bounded below $X$-spectra of finite rational type:}
Let $P$ be a nilpotent, bounded below $X$-spectrum of finite rational type.
By Corollary \ref{cor:NilBblFinRatType}, $P$ is rationally equivalent to an $X$-spectrum $Q$ satisfying the conditions:
\begin{itemize}
  \item $Q_{\leq k}\cong 0_X$ for some integer $k$
  \item Each $Q_{\leq l}\to Q_{\leq (l-1)}$ factors as a finite sequence of extensions by $X$-spectra of the form $X^\ast(\Sigma^l (\bigoplus_{i=1}^n H\mathbb{Q}))$ for finite $n$.
\end{itemize}
Taking cofibrant replacements as necessary and using that $\mathfrak{M}_A$ sends rational equivalences of cofibrant $X$-spectra to quasi-isomorphisms, we may suppose without loss of generality that $P$ itself satisfies the above conditions.

Let $k$ be the maximal integer for which $P_{\leq (k-1)} =0_X$ is the zero $X$-spectrum.
Then for any integer $l\geq k$, $P_{\leq l}$ is obtained from the zero $X$-spectrum by a finite sequence of extensions by $X$-spectra of the form  $X^\ast(\Sigma^j (\bigoplus_{i=1}^n H\mathbb{Q}))$ for finite $n$ and $k\leq j\leq l$.
We claim that the derived unit $P_{\leq l} \to \mathbf{R}\mathfrak{P}_A \mathbf{L}\mathfrak{M}_A (P_{\leq l})$ is a rational equivalence for all $l\geq k$.

To prove the claim, let $E_0$ be an $X$-spectrum for which the derived unit $E_0 \to \mathbf{R}\mathfrak{P}_A \mathbf{L}\mathfrak{M}_A (E_0)$ is a rational equivalence.
Suppose that $E_1$ is an extension of $E_0$ by $X^\ast (\Sigma^j (\bigoplus_{i=1}^n H\mathbb{Q}))$ for some finite $n$ and integer $j$, so that we have a commuting diagram of homotopy pullback-pushout squares
\[
\begin{tikzcd}
X^\ast( \Sigma^j (\bigoplus_{i=1}^n H\mathbb{Q}))
\ar[r]
\ar[d]
&
E_1
\ar[d]
\ar[r]
&
0_X
\ar[d]
\\
0_X
\ar[r]
&
E_0
\ar[r]
&
X^\ast( \Sigma^{j-1} (\bigoplus_{i=1}^n H\mathbb{Q}))
\end{tikzcd}
\]
The component of the derived unit at $X^\ast( \Sigma^j (\bigoplus_{i=1}^n H\mathbb{Q}))$  is a rational equivalence by the above and the fact that both $\mathbf{L}\mathfrak{M}_A$ and $\mathbf{R}\mathfrak{P}_A$ commute with shifts.
We thus have a commuting diagram of $X$-spectra
\[
\begin{tikzcd}
X^\ast( \Sigma^j (\bigoplus_{i=1}^n H\mathbb{Q}))
\ar[r]
\ar[d, "\sim_{\mathrm{rat}}"']
&
E_1
\ar[r]
\ar[d]
&
E_0
\ar[d, "\sim_{\mathrm{rat}}"]
\\
\mathbf{R}\mathfrak{P}_A \mathbf{L}\mathfrak{M}_A (X^\ast( \Sigma^j (\bigoplus_{i=1}^n H\mathbb{Q})))
\ar[r]
&
\mathbf{R}\mathfrak{P}_A \mathbf{L}\mathfrak{M}_A (E_1)
\ar[r]
&
\mathbf{R}\mathfrak{P}_A \mathbf{L}\mathfrak{M}_A (E_0)
\end{tikzcd}
\]
in which the outer vertical arrows are rational equivalences and both top and bottom horizontal rows are fibre-cofibre sequences ($\mathbf{L}\mathfrak{M}_A$ and $\mathbf{R}\mathfrak{P}_A$ are exact).
Passing to fibres at the unique $0$-simplex $x$ of $X$ and computing rational stable homotopy groups, a routine application of the five lemma  shows that $E_1\to \mathbf{R}\mathfrak{P}_A \mathbf{L}\mathfrak{M}_A (E_1)$ is a rational equivalence.
The component of derived unit at the zero $X$-spectrum is clearly a rational equivalence.
For any $l\geq k$, $P_{\leq l}\to 0$ factors as a finite sequence of extensions by $X^\ast (\Sigma^j (\bigoplus_{i=1}^n H\mathbb{Q}))$'s and hence the component of the derived unit at $P_{\leq l}$ is a rational equivalence by induction.

To conclude, fix $l\geq k$ and consider the fibre-cofibre sequence of $X$-spectra $P_{\geq (l+1)}\to P\to P_{\leq l}$.
The derived unit gives rise to a morphism of fibre-cofibre sequences
\[
\begin{tikzcd}
P_{\geq (l+1)}
\ar[r]
\ar[d]
&
P
\ar[r]
\ar[d]
&
P_{\leq l}
\ar[d, "\sim_\mathrm{rat}"]
\\
\mathbf{R}\mathfrak{P}_A \mathbf{L}\mathfrak{M}_A (P_{\geq (l+1)})
\ar[r]
&
\mathbf{R}\mathfrak{P}_A \mathbf{L}\mathfrak{M}_A (P)
\ar[r]
&
\mathbf{R}\mathfrak{P}_A \mathbf{L}\mathfrak{M}_A (P_{\leq l})
\end{tikzcd}
\]
in which the right-hand vertical arrow is a rational equivalence by the above. 
By Corollary \ref{cor:Spec2ModkConnective} the $X$-spectrum $\mathbf{R}\mathfrak{P}_A \mathbf{L}\mathfrak{M}_A (P_{\geq l})$ is $(l+1)$-connective.
By passing to fibre spectra at $x$ and computing rational stable homotopy groups, the five lemma shows that $x^\ast P\to x^\ast \mathbf{R}\mathfrak{P}_A\mathbf{L}\mathfrak{M}_A(P)$ induces $\spi_n\otimes_\mathbb{Z}\mathbb{Q}$-isomorphisms for all $n\leq l$.
Since this is true for all $l\geq k$, it follows that $P\to \mathbf{R}\mathfrak{P}_A\mathbf{L}\mathfrak{M}_A(P)$ is a rational equivalence of $X$-spectra.

\paragraph{$A$-modules of finite homotopical type:}
Let $M$ be an $A$-module of finite homotopical type with minimal model $A\otimes V\to M$.
By minimality, there is a well-ordered set $\mathcal{I}$ which indexes a basis $\{v_\alpha\mid \alpha\in \mathcal{I}\}$ for $V$ such that (i) $dv_\beta \in A\otimes \mathrm{span}\{v_\alpha\mid \alpha < \beta\}$ and (ii) $\alpha \leq \beta\Rightarrow |v_\alpha|\leq |v_\beta|$.

For each $\alpha \in\mathcal{I}$ write $V_{< \alpha} = \mathrm{span}\{v_\beta \mid \beta < \alpha\}$ and set $\alpha + 1 := \inf\{\beta\in \mathcal{I}\mid \beta > \alpha\}$.
The minimal model $A\otimes V$ is the directed colimit of the $A\otimes V_{< \alpha}$, where for each $\alpha\in \mathcal{I}$ there is a pushout diagram of $A$-modules
\[
\begin{tikzcd}
A \otimes S(|v_{\alpha +1}| +1)
\ar[d]
\ar[r, "dv_{\alpha +1}"]
&
A\otimes V_{< (\alpha+1)}
\ar[d]
\\
A\otimes D(|v_{\alpha +1}| +1)
\ar[r]
&
A\otimes V_{< (\alpha +2)}
\end{tikzcd}
\]
The left vertical arrow is a cofibration, all objects are cofibrant, and $A\otimes D(|v_{\alpha+1}|+1)$ is acyclic, so that this diagram expresses $A\otimes V_{<(\alpha+1)}$ as the homotopy cofibre of a map of $A$-modules $A[-(|v_{\alpha+1}|+1)]\to A\otimes V_{<\alpha}$.

For each integer $n$, let $V^{\leq n} = \mathrm{span}\{v_\alpha\mid |v_\alpha|\leq n\}$.
Since $V$ is finite dimensional in each degree, for each $n$ the evident inclusion $A\otimes V^{\leq n}\to A\otimes V^{\leq (n+1)}$ factors as a finite sequence of pushouts along $A\otimes S(n+1)\to A\otimes D(n+1)$.
Letting $k$ be the largest integer for which $V^{\leq (k-1)} = 0$, for each $l\geq k$ the zero map $0\to A\otimes V^{\leq l}$ factors as a finite sequence  (say, $n_l$ in total) of pushouts along the maps $A\otimes S(j+1)\to A\otimes D(j+1)$ for $k\leq j\leq l$, and so we get a diagram in $Ho(A\mathrm{-Mod})$ 
\begin{equation}
\label{eqn:AModExtSeq}
\begin{tikzcd}[row sep = small, column sep = small]
A[-j_{1}-1]
\ar[d]
&
A[-j_{2}-1]
\ar[d]
&
\dotsb
%&
%A[-j_{i+1}-1]
%\ar[d]
%&
%\dotsb
&
A[-j_{n_l}-1]
\ar[d]
\\
0
\ar[r]
&
(A\otimes V^{\leq 1})_1
\ar[r]
&
\dotsb
%\ar[r]
%&
%(A\otimes V^{\leq l})_i
%\ar[r]
%&
%\dotsb
\ar[r]
&
(A\otimes V^{\leq k})_{n_l-1}
\ar[r]
&
A\otimes V^{\leq k}
\end{tikzcd}
\end{equation}
in which each \lq\lq hook'' subdiagram (consisting of a vertical arrow followed by a horizontal arrow) is a cofibre sequence and $k\leq j_1\leq j_2\leq \dotsb \leq j_{n_l}\leq l$. 

Suppose $N$ is an $A$-module for which the derived counit $N\to \mathbf{L}\mathfrak{M}_A\mathbf{R}\mathfrak{P}_A (N)$ is a quasi-isomorphism, and suppose that there is a cofibre sequence of $A$-modules
$A[n]\to N \to N'$ for some integer $n$.
Naturality of the derived counit implies a commuting diagram of $A$-modules
\[
\begin{tikzcd}
A[n]
\ar[r]
\ar[d, "\sim_{\mathrm{qis}}"']
&
N
\ar[r]
\ar[d,"\sim_{\mathrm{qis}}"']
&
N'
\ar[d]
\\
\mathbf{L}\mathfrak{M}_A\mathbf{R}\mathfrak{P}_A (A)[n]
\ar[r]
&
\mathbf{L}\mathfrak{M}_A\mathbf{R}\mathfrak{P}_A (N)
\ar[r]
&
\mathbf{L}\mathfrak{M}_A\mathbf{R}\mathfrak{P}_A (N')
\end{tikzcd}
\]
in which both top and bottom rows are cofibre sequences.
The left hand vertical arrow is a quasi-isomorphism since the derived counit $A\to \mathbf{L}\mathfrak{M}_A\mathbf{R}\mathfrak{P}_A (A)$ is a quasi-isomorphism and $\mathbf{L}\mathfrak{M}_A$, $\mathbf{R}\mathfrak{P}_A$ commute with shifts; the middle vertical arrow is a quasi-isomorphism by hypothesis.
Passing to cohomology, a routine application of the five lemma demonstrates that $N'\to \mathbf{L}\mathfrak{M}_A\mathbf{R}\mathfrak{P}_A (N')$ is a quasi-isomorphism.
A straightforward inductive argument over the stages of the diagram \eqref{eqn:AModExtSeq} now shows that the derived counit $A\otimes V^{\leq l}\to \mathbf{L}\mathfrak{M}_A\mathbf{R}\mathfrak{P}_A (A\otimes V^{\leq l})$ is a quasi-isomorphism for all $l$.

To complete the proof, fix $l\geq k$ and consider the map of $A$-modules $A\otimes V^{\leq l}\to A\otimes V$.
Since this map is a cofibration, the homotopy cofibre is $A\otimes (V/V^{\leq l})$,  which is $(l+1)$-connective.
Naturality of the derived counit implies a commuting diagram of $A$-modules
\[
\begin{tikzcd}
A\otimes V^{\leq l}
\ar[r]
\ar[d, "\sim_{\mathrm{qis}}"']
&
A\otimes V
\ar[r]
\ar[d]
&
A\otimes (V/V^{\leq l})
\ar[d]
\\
\mathbf{L}\mathfrak{M}_A\mathbf{R}\mathfrak{P}_A (A\otimes V^{\leq l})
\ar[r]
&
\mathbf{L}\mathfrak{M}_A\mathbf{R}\mathfrak{P}_A (A\otimes V)
\ar[r]
&
\mathbf{L}\mathfrak{M}_A\mathbf{R}\mathfrak{P}_A (A\otimes (V/V^{\leq l}))
\end{tikzcd}
\]
By Corollary \ref{cor:Spec2ModkConnective}, both $A\otimes (V/V^{\leq l})$ and $\mathbf{L}\mathfrak{M}_A\mathbf{R}\mathfrak{P}_A (A\otimes (V/V^{\leq l}))$ have trivial cohomology in degrees $\leq l$.
Passing to cohomology, the five lemma implies that $A\otimes V\to \mathbf{L}\mathfrak{M}_A\mathbf{R}\mathfrak{P}_A (A\otimes V)$ induces an isomorphism in cohomology in degrees $\leq l$.
Since this is true for all $l\geq k$,  the derived counit $A\otimes V\to \mathbf{L}\mathfrak{M}_A\mathbf{R}\mathfrak{P}_A (A\otimes V)$ is a quasi-isomorphism.
\end{proof}

\section{A rational homotopy theory dictionary}
\label{S:Dictionary}
Our main result, Theorem \ref{thm:RatParamHomThry}, asserts an equivalence of rational homotopy theories between parametrised spectra and differential graded modules.
This correspondence can be interpreted in two ways: on the one hand it allows for identifications of topological constructions in algebra, where it is often much easier to carry out explicit calculations; on the other hand, it produces topological realisations of differential graded modules. 
In this section, we explain various aspects of this algebro-topological dictionary, summarised in Table \ref{table}.

\begin{table}[!p]
\caption{The rational homotopy theory dictionary}
\label{table}
\begin{center}
\begin{tabular}{C{0.16\textwidth}|| C{0.37\textwidth} | C{0.37\textwidth}}
&Topology
&
Algebra
\\
\hline
\hline
stable
&
spectrum $E$
&
cochain complex $C$
\\
\hline
unstable
&
spaces $X, Y$&
cdgas $A, B$
\\
\cline{2-3}
&
map of spaces 
&
map of cdgas 
\\[-1mm]
&
$X \xrightarrow{f} Y$
&
$A\xleftarrow{\phi} B $
\\
\hline
parametrised stable
&
$X$-spectra $P, Q$
&
$A$-modules $M, N$
\\
\cline{2-3} 
&
map of $X$-spectra
&
map of $A$-modules
\\[-1mm]
&
$P\to Q$
&
$M\leftarrow N$
\\
\cline{2-3}
&
fibrewise stabilisation (unpointed)
&
forget algebraic structure
\\[-1mm]
&
$(Y\to X)\mapsto \Sigma^\infty_{X+}Y$
&
$(B\leftarrow A)\mapsto B$
as $A$-module
\\
\cline{2-3}
&
fibrewise stabilisation (pointed)
&
augmentation ideal
\\[-1mm]
&
$(X \to Y\to X)\mapsto \Sigma^\infty_{X}Y$
&
$(A\leftarrow B\leftarrow A)\mapsto \mathrm{aug}_A(B)$
\\
\cline{2-3}
&
fibrewise infinite loop space
&
free $A$-algebra on connective cover
\\[-1mm]
&
$P\mapsto (\Omega^\infty_X \to X)$
&
$M\mapsto (\Lambda_A (\mathrm{cn}^0_A M)\leftarrow A)$
\\
\cline{2-3}
&
fibrewise (de)suspension
&
shift
\\[-1mm]
&
$P\mapsto \Sigma_XP$
&
$M\mapsto M[-1]$
\\[-1mm]
&
$Q\mapsto \Omega_X Q$
&
$N\mapsto N[1]$
\\
\cline{2-3}
&
pushforward along $f\colon X\to Y$
&
restriction along $\phi \colon B \to A$
\\[-1mm]
&
$P\mapsto f_! P$
&
$M \mapsto M$ (as $B$-module)
\\
\cline{2-3}
&
pullback along $f\colon X\to Y$
&
extension along $\phi\colon B\to A$
\\[-1mm]
&
$Q\mapsto f^\ast Q$
&
$N \mapsto A \bigotimes^\mathbf{L}_{B} N$
\\
\cline{2-3}
&
$k$-connective cover
&
quotient of minimal model
\\[-1mm]
&
$P_{\geq k}\to  P$
&
$A\otimes V^{\geq k}\leftarrow A\otimes V$
\\
\cline{2-3}
&
$k$-th Postnikov section
&
submodule of minimal model
\\[-1mm]
&
$P\to P_{\leq k}$
&
$A\otimes V\leftarrow A\otimes V^{\leq k}$
\\
\cline{2-3}
&
fibrewise smash product
&
derived tensor product
\\[-1mm]
&
$P\wedge_X Q$
&
$M\bigotimes^\mathbf{L}_A N$
\\
\cline{2-3}
&
fibrewise stable maps
&
Ext groups
\\[-1mm]
&
$\{P, Q\}^\ast_X\otimes_\mathbb{Z}\mathbb{Q}$
&
$\mathrm{Ext}_A^\ast(N, M)$
\\
\end{tabular}
\end{center}
\end{table}

\begin{remark}[Standard assumptions]
\label{rem:StandingAssumptions}
Throughout this section, we take $A$ to be a cdga satisfying the \emph{ standard assumptions}: $A$ is a cofibrant, connected cdga of finite homotopical type, with spatial realisation $X = \mathfrak{S}(A)$.
As $A$ is connected, $A^0= \mathbb{Q}$ and so there is a canonical augmentation $a\colon A\to\mathbb{Q}$ corresponding to a distinguished $0$-simplex $x$ of $X$. 
\end{remark}

\subsection{Fibrewise (de)stabilisation}
Let $p\colon Y\to X$ be a nilpotent fibration, with $Y$ nilpotent of finite rational type.
Let $F$ be the fibre of $p$ at $x$, noting that $F$ is also nilpotent of finite rational type.
By Proposition \ref{prop:StabNilFib}, the fibrewise stabilisation $\Sigma^\infty_{X+} Y = \Sigma^\infty_X (X\coprod Y)$ is a connective nilpotent $X$-spectrum with stable fibre $\Sigma^\infty_+ F$ at $x$.
Our assumptions imply that the graded vector space $H_\bullet (F)\cong \spi_\ast F\otimes_\mathbb{Z}\mathbb{Q}$ is finite dimensional in each degree and hence that $\Sigma^\infty_{X+} Y$ is of finite rational type. 
By Theorem \ref{thm:RatParamHomThry} there is an $A$-module corresponding to $\Sigma^\infty_{X+} Y$ in rational homotopy theory.

By Lemma \ref{lem:SuspendInMod} we have
\[
\mathfrak{M}_A \big(\Sigma^\infty_{X+} Y\big)
\cong
\mathfrak{M}^u_A \Big(X\coprod Y\Big)
\cong 
\mathfrak{A}(Y)\,,
\]
with $A$-action induced by the composite map of cdgas $A\xrightarrow{\varepsilon_A} \mathfrak{A}(\mathfrak{S}(A)) =\mathfrak{A}(X)\xrightarrow{\mathfrak{A}(p)} \mathfrak{A}(Y)$, with $\varepsilon_A$ the Sullivan--de Rham counit.
For any cofibrant replacement  $B\to \mathfrak{A}(Y)$ in $\mathrm{cDGA}$ then we can find an up-to-homotopy factorisation of $A\to \mathfrak{A}(Y)$ through $B$.
We can thus regard $B$ as an $A$-algebra, and the equivalence of Theorem \ref{thm:RatParamHomThry} identifies the rational homotopy type of the $X$-spectrum $\Sigma^\infty_{X+} Y$ with $B$.

\begin{remark}
More generally, consider the case of a retractive space $X\to Z\to X$ where $p\colon Z\to X$ is a nilpotent fibration and $Z$ is nilpotent of finite rational type. 
As above, the fibre $F$ of $p$ is also nilpotent of finite rational type.
By Proposition \ref{prop:StabNilFib2} the fibrewise stabilisation $\Sigma^\infty_X Z$ is a nilpotent connective $X$-spectrum with fibre $\Sigma^\infty F$.
The rationalised stable homotopy groups $\spi_\ast F\otimes_\mathbb{Z}\mathbb{Q}$ compute the reduced rational homology of $F$, which is finite dimensional in each degree.
By Lemma \ref{lem:SuspendInMod}, the $A$-module corresponding to $\Sigma_X^\infty Z$ under the correspondence of Theorem \ref{thm:RatParamHomThry} is 
\[
\mathfrak{M}_A \big(\Sigma^\infty_X Z\big) \cong \mathfrak{M}^u_A \big(Z\big)
\cong 
\mathrm{aug}_{\mathfrak{A}(X)} \mathfrak{A}(Z)
\]
(see also the discussion in Section \ref{ss:ModularModels}).
\end{remark}

We write:
\begin{itemize}
  \item $Ho(\mathcal{S}_{/X})^\mathbb{Q}_{\mathrm{nil,f.t.}}$ for the full subcategory of $Ho(\mathcal{S}_{/X})$ spanned by the nilpotent fibrations $Y\to X$ for which $Y$ is nilpotent of finite rational type; and
  
  \item $Ho(\mathrm{cDGA}^{A/})^\mathrm{op}_{\mathrm{f.t.,}\geq 1}$ for the full subcategory of $Ho(\mathrm{cDGA}^{A/})^\mathrm{op}$ spanned by maps of cdgas $\phi\colon A\to B$ for which (i) $B$ is of finite homotopical type and (ii) $H^1(\phi)\colon H^1(A)\to H^1 (B)$ is an isomorphism\footnote{The latter condition guarantees that $\phi$ has a minimal model $A\to A\otimes \Lambda V$ for $V$ concentrated in degrees strictly greater than $1$.}.
\end{itemize}
A straightforward adaptation of Bousfield and Gugenheim's proof of the Sullivan--de Rham equivalence theorem \cite[Section 10]{bousfield_rational_1976} shows that there is an adjoint equivalence of homotopy categories
\[ 
\begin{tikzcd}
Ho(\mathcal{S}_{/X})^\mathbb{Q}_{\mathrm{nil,f.t.}}
\ar[rr, shift left = 2, "\mathbf{L}\mathfrak{A}_{/A}"]
\ar[rr, leftarrow, shift left = -2, "\mathbf{R}\mathfrak{S}_{/A}"', "\simeq"]
&&
Ho(\mathrm{cDGA}^{A/})^\mathrm{op}_{\mathrm{f.t.,}\geq 1}\,.
\end{tikzcd}
\]
The left adjoint is the derived functor of the assignment
$
\mathfrak{A}_{/A}\colon (Y\to X)\mapsto (A\to \mathfrak{A}(X)\to \mathfrak{A}(Y))
$
which has right adjoint
$
\mathfrak{S}_{/A}\colon (A\to B)\mapsto (\mathfrak{S}(B)\to \mathfrak{S}(A) = X)
$.
Combined with the discussion of the preceding paragraphs, we have all but proven the following
\begin{proposition}
\label{prop:FibSuspendinMod}
Let $A$ be a cofibrant connected cdga of finite homotopical type with spatial realisation $X = \mathfrak{S}(A)$.
There is a diagram of derived left adjoint functors commuting up to natural isomorphism
\[
\begin{tikzcd}
Ho(\mathcal{S}_{/X})^\mathbb{Q}_{\mathrm{nil,f.t.}}
\ar[rr, "\mathbf{L}\mathfrak{A}_{/A}", "\simeq"']
\ar[d, "\Sigma^\infty_{X+}"']
&&
Ho(\mathrm{cDGA}^{A/})^\mathrm{op}_{\mathrm{f.t.,}\geq 1}
\ar[d, "U"]
\\
Ho\big(\mathrm{Sp}_{X}\big)^\mathbb{Q}_{\mathrm{f.t.,nil,bbl}}
\ar[rr, "\mathbf{L}\mathfrak{M}_A", "\simeq"']
&&
Ho\big(A\mathrm{-Mod})^\mathrm{op}_\mathrm{f.h.t.}
\end{tikzcd}
\]
where $U\colon (A\to B)\mapsto B$ sends an $A$-algebra to its underlying $A$-module.
\end{proposition}

The forgetful functor $U\colon \mathrm{cDGA}^{/A}\to A\mathrm{-Mod}$ is right Quillen with left adjoint given by the composite
\[
\Lambda_A^{\geq 0}
\colon
\begin{tikzcd}
A\mathrm{-Mod}
\ar[r, "\mathrm{cn}_A^0"]
&
A\mathrm{-Mod}_{\geq 0}
\ar[r, "\Lambda_A"]
&
\mathrm{cDGA}^{A/}
\end{tikzcd}
\]
of the $A$-module connective cover and free $A$-algebra functors.
Note that $U$ already presents the derived functor since all objects of $\mathrm{cDGA}^{A/}$ are fibrant. 
For any minimal $A$-module $A \otimes V$ with $V$ concentrated in degrees $>1$ we have
\[
\mathbf{L}\Lambda^{\geq 0}_A(A\otimes V) \cong  
\Lambda^{\geq 0}_A(A\otimes V)
\cong A\otimes \Lambda V
\] 
in $Ho(\mathrm{cDGA}^{A/})^\mathrm{op}_{\mathrm{f.t.,}\geq 1}$.
Taking adjoints in Proposition \ref{prop:FibSuspendinMod}, we get rational homotopy equivalences
\begin{align*}
\big(
\Omega^\infty_X \mathbf{R}\mathfrak{P}_A (A\otimes V)
\to X\big)
&\sim_{\mathrm{rat}}
\big((\mathbf{R}\mathfrak{S}_{/A} \mathbf{L} \Lambda^{\geq 0}_A )(A\otimes V) \to X\big)
\\
&\sim_{\mathrm{rat}}
\big(\mathfrak{S}(A\otimes \Lambda V) \to X\big)
\end{align*}
of nilpotent fibrations over $X$.

\begin{example}
The Hopf fibration $S^3 \to S^2$ is modelled in rational homotopy theory by the minimal morphism of cdgas
\[
A_{S^2}:= \Lambda\langle x_2, x_4\rangle
\Big/
\left(
{\begin{aligned}
dx_2 &= 0\\
dx_4 &= x_2\wedge x_2
\end{aligned}}
\right)
\longrightarrow
A_{S^3} :=\Lambda\langle x_1, x_2, x_4\rangle
\Big/
\left(
{\begin{aligned}
dx_1 &= x_2\\
dx_2 &= 0\\
dx_4 &= x_2\wedge x_2
\end{aligned}}
\right)
\]
with subscripts on generators indicating degree.
The rational homotopy type of the fibrewise stabilisation $\Sigma^\infty_{S^2+} S^3$ is thus given by $A_{S^3}$ regarded as an $A_{S^2}$-module.
Note that there is an isomorphism of $A_{S^2}$-modules
\[
A_{S^3} \cong A_{S^2} \otimes \Lambda \langle x_1\rangle
\cong A_{S^2}\otimes \big( \mathrm{span}\{x_1\}\oplus \mathbb{Q}\big)\,,
\qquad dx_1 = x_2
\]
so that $A_{S^3}$ is already minimal.
It follows that the component of the derived $(\Lambda^{\geq 0}_{A_{S^2}}\dashv U)$-counit at $A_{S^2}\to A_{S^3}$ is an isomorphism and hence that 
\[
\begin{tikzcd}[row sep = small]
S^3
\ar[rr, "\sim_{\mathrm{rat}}"]
\ar[dr]
&&
\Omega^\infty_{S^2}\Sigma^\infty_{S^2+} S^3
\ar[dl]
\\
&
S^2
\end{tikzcd}
\]
is a fibrewise rational homotopy equivalence.
\end{example}

\begin{example}[Path fibrations]
Let $A$ be a connected, cofibrant cdga of finite homotopical type such that $X = \mathfrak{S}(A)$ is simply connected.
In this case the path fibration $PX \to X$ is nilpotent and hence the fibrewise stabilisation of $PX\to X$ is modelled in rational homotopy theory by the homotopy type of $\mathbb{Q}$, regarded as an $A$-module by the canonical augmentation $A\to\mathbb{Q}$.
The assumption of simple connectedness is necessary; see Remark \ref{rem:FHTCounterexample} for a counterexample when this assumption does not hold.
As an $A$-module, $\mathbb{Q}$ is rarely cofibrant but the bar resolution $B(A;A; \mathbb{Q}) \to \mathbb{Q}$ provides a canonical cofibrant replacement. 
\end{example}

\subsection{Base change}
Let $\phi\colon B\to A$ be a morphism of cofibrant connected cdgas of finite homotopical type.
Under the Sullivan--de Rham equivalence, $\phi$ corresponds to a morphism of rational homotopy types
\[
(f\colon X\to Y) = \big(\mathfrak{S}(\phi)\colon \mathfrak{S}(A)\to \mathfrak{S}(B)\big)\,.
\]
It is easy to see that the derived pullback functor $\mathbf{R}f^\ast\colon Ho(\mathrm{Sp}_Y)\to Ho(\mathrm{Sp}_X)$ preserves the conditiones of nilpotence, boundedness from below, and finite rational type and so restricts to a functor
\[
\mathbf{R}f^\ast \colon 
Ho\big(\mathrm{Sp}_{Y}\big)^\mathbb{Q}_{\mathrm{f.t.,nil,bbl}}
\longrightarrow
Ho\big(\mathrm{Sp}_{X}\big)^\mathbb{Q}_{\mathrm{f.t.,nil,bbl}}\,.
\]
By (the adjoint of) Proposition \ref{prop:Spec2ModPNat}, the equivalences of Theorem \ref{thm:RatParamHomThry} fit into a diagram of derived functors that commutes up to natural isomorphism:
\[
\begin{tikzcd}
B\mathrm{-Mod}^\mathrm{op}_\mathrm{f.h.t.}
\ar[r,"\mathbf{R}\mathfrak{P}_B", "\simeq"']
\ar[d, "\mathbf{L}\phi_!"']
&
Ho\big(\mathrm{Sp}_{Y}\big)^\mathbb{Q}_{\mathrm{f.t.,nil,bbl}}
\ar[d, "\mathbf{R}f^\ast"]
\\
A\mathrm{-Mod}^\mathrm{op}_\mathrm{f.h.t.}
\ar[r, "\mathbf{R}\mathfrak{P}_A", "\simeq"']
&
Ho\big(\mathrm{Sp}_{X}\big)^\mathbb{Q}_{\mathrm{f.t.,nil,bbl}}
\end{tikzcd}
\]
The upshot is that if $N$ is a cofibrant $B$-module of finite homotopical type corresponding to the rational homotopy type of the $Y$-spectrum $P:= \mathfrak{P}_B (N)$, then the $A$-module 
\[
\mathbf{L}\phi_! N \cong \phi_! N\cong  A\bigotimes_B N
\]
corresponds to the rational homotopy type of the $X$-spectrum $f^\ast P$.

\begin{example}[Fibre spectra]
Under the standard assumptions on $A$ there is a canonical augmentation $a\colon A\to \mathbb{Q}$ corresponding to a point $x\colon \ast\to X$ of the spatial realisation.
If $M$ is a cofibrant $A$-module of finite homotopical type with $P = \mathfrak{P}_A(M)$, the fibre of $P$ at $x$ is $x^\ast P\cong \mathfrak{P}_\mathbb{Q}(a_! M)$.
\end{example}

\begin{remark}[Minimal model bootstrap]
\label{rem:MinModBoot}
If $A\otimes V$ is a minimal $A$-module corresponding to the rational homotopy type of the $X$-spectrum $P$, the previous example shows that
$
x^\ast P
\cong 
\mathfrak{P}_\mathbb{Q}(V)
$,
where $V$ has trivial differential (see Remark \ref{rem:BCMinimal}). 
By Lemma \ref{lem:FibofMinModSpec}, $x^\ast P$ is thus the Eilenberg--Mac Lane spectrum $HV^\vee$.
In particular, if $P$ is a nilpotent, bounded below $X$-spectrum of finite rational type for which the rational stable homotopy groups of $x^\ast P$ are known, then there is a minimal model for $M:= \mathbf{L}\mathfrak{M}_A(P)$ with underlying $A$-module
\[
A\otimes \big(\spi_\ast x^\ast P\otimes_\mathbb{Z}\mathbb{Q}\big)^\vee\,.
\]
Given knowledge of the graded rational vector space $\spi_\ast X_!P \otimes_\mathbb{Z}\mathbb{Q}\cong (H^\bullet(M))^\vee$, one is often able to completely determine a minimal model of $M$---see \cite[Proposition 2.49]{braunack-mayer_gauge_2019} for an implementation of this approach.
\end{remark}

Provided $f\colon X\to Y$ is homotopy equivalent to a nilpotent fibration, we can also give an algebraic description of the derived pushforward functor $\mathbf{L}f_!$.
The next result is phrased in terms of the equivalent algebraic characterisation on the map of cdgas $\phi$:
\begin{proposition}
Let $\phi\colon B\to A$ be a morphism of cofibrant connected cdgas of finite homotopical type such that $H^1(\phi)\colon H^1(B)\to H^1(A)$ is an injection.
Then there is a diagram of derived functors that commutes up to natural isomorphism:
\[
\begin{tikzcd}
Ho\big(\mathrm{Sp}_{X}\big)^\mathbb{Q}_{\mathrm{f.t.,nil,bbl}}
\ar[r,"\mathbf{L}\mathfrak{M}_B", "\simeq"']
\ar[d, "\mathbf{L}f_!"']
&
A\mathrm{-Mod}^\mathrm{op}_\mathrm{f.h.t.}
\ar[d, "\mathbf{R}\phi^\ast"]
\\
Ho\big(\mathrm{Sp}_{Y}\big)^\mathbb{Q}_{\mathrm{f.t.,nil,bbl}}
\ar[r, "\mathbf{L}\mathfrak{M}_A", "\simeq"']
&
B\mathrm{-Mod}^\mathrm{op}_\mathrm{f.h.t.}
\end{tikzcd}
\]
\end{proposition}
\begin{proof}
The assumption on $\phi$ guarantees the existence of a minimal model $B\to B\otimes \Lambda V \xrightarrow{\sim_\mathrm{qis}} A$ with $V$ concentrated in degrees $\geq 1$.
Taking spatial realisations, $X \xrightarrow{\sim_\mathrm{rat}} \mathfrak{S}(B\otimes \Lambda V ) \to Y$  is a nilpotent fibration by the Sullivan--de Rham equivalence.

We now aim to show that $\mathbf{L}f_!$ restricts to a functor 
\[
Ho\big(\mathrm{Sp}_{X}\big)^\mathbb{Q}_{\mathrm{f.t.,nil,bbl}}
\longrightarrow
Ho\big(\mathrm{Sp}_{Y}\big)^\mathbb{Q}_{\mathrm{f.t.,nil,bbl}}\,.
\]
Once this has been established, the result follows at once from Theorem \ref{thm:RatParamHomThry} and Proposition \ref{prop:Spec2ModPNat}.
The functor $\mathbf{L}f_!$ is exact and preserves connectivity (Corollary \ref{cor:PfwdConnectivity}), hence by
Corollary \ref{cor:NilBblFinRatType} it is sufficient to show that $\mathbf{L}f_! X^\ast H\mathbb{Q}$ is a nilpotent, bounded below $Y$-spectrum of finite rational type.
Omitting \lq\lq$\mathbf{L}$'' and \lq\lq$\mathbf{R}$'' for clarity, we have
\[
f_! X^\ast H\mathbb{Q} \cong
f_! (X^\ast S \wedge_X X^\ast H\mathbb{Q})
\cong
(f_! X^\ast S) \wedge_Y Y^\ast H\mathbb{Q}
\cong \Sigma^\infty_{Y+} X \wedge_Y Y^\ast H\mathbb{Q}\,.
\]
Since $X\to Y$ is (homotopy equivalent to) a nilpotent fibration and $X$ is nilpotent of finite rational type, the fibrewise stabilisation $\Sigma^\infty_{Y+} X$ is a nilpotent, bounded below $Y$-spectrum of finite rational type (see the discussion preceding Proposition \ref{prop:FibSuspendinMod}).
\end{proof}

\begin{example}[Collapse spectra]
For any cofibrant connected cdga $A$ of finite homotopical type, the unit $\mathbb{Q}\to A$ satisfies the conditions of the Proposition.
Let $P$ be a nilpotent, bounded below $X=\mathfrak{S}(A)$-spectrum of finite rational type with rational homotopy type corresponding to the $A$-module $M$.
The rational homotopy type of the spectrum $X_! P$ is thus modelled by the underlying cochain complex of $M$.

As special case we have that the spectrum $X_! \Sigma^\infty_{X+}X  \cong \Sigma^\infty_+ X$ is modelled in rational homotopy by the underlying cochain complex of $A$ (cf.~Proposition \ref{prop:TopInterp}).
\end{example}

\subsection{Connective covers and Postnikov sections}
\label{ss:Postnikov}
If $P$ is a nilpotent, bounded below $X$-spectrum of finite rational type it is trivial to check that for all integers $k$, the $X$-spectra $P_{\geq k}$ and $P_{\leq k}$ are also nilpotent, bounded below, and of finite rational type.
We now identify maps of $A$-modules corresponding to $P_{\geq k}\to P$ and $P\to P_{\leq k}$ in rational homotopy theory.

Fix a minimal model $A\otimes V$ for the rational homotopy type of $P$.
By minimality, $V$ is generated by a basis $\{v_\alpha\}_{\alpha\in \mathcal{I}}$ for some well-ordered set $\mathcal{I}$ such that (i) $dv_\beta \in A\otimes \mathrm{span}\{a_\alpha\mid \alpha < \beta\}$ and (ii) $\alpha\leq \beta \Rightarrow |v_\alpha|\leq |v_\beta|$.
Writing
\[
V^{\leq k} :=\mathrm{span}\{v_\alpha \mid |v_\alpha|\leq k\}
\qquad
\mbox{ and }
\qquad
V^{\geq k} :=\mathrm{span}\{v_\alpha \mid |v_\alpha|\geq k\}\,,
\]
the inclusion $V^{\leq k}\to V$ induces a map of minimal $A$-modules
$
A\otimes V^{\leq k}\to A\otimes V
$.
Similarly, there is a natural map of minimal $A$-modules
$
A\otimes V \to A\otimes V^{\geq k}
$
obtained by passing to the quotient $(A\otimes V)/(A\otimes V^{\leq (k-1)})\cong A\otimes V^{\geq k}$.
\begin{proposition}
Let $A\otimes V$ be a minimal model for the rational homotopy type of $P\in Ho(\mathrm{Sp}_X)^\mathbb{Q}_{\mathrm{f.t.,nil,bbl}}$.
Then the short exact sequence of minimal $A$-modules
\[
A\otimes V^{\leq k}
\longrightarrow 
A\otimes V
\longrightarrow
A\otimes V^{\geq (k+1)}
\]
corresponds  to the fibre-cofibre  sequence of $X$-spectra
\[
P_{\leq k}\longleftarrow P
\longleftarrow
P_{\geq (k+1)}
\]
under the equivalence of Theorem \ref{thm:RatParamHomThry}.
\end{proposition}
\begin{proof}
Applying $\mathfrak{P}_A$ to the short exact sequence of minimal $A$-modules
\[
A\otimes V^{\leq k}
\longrightarrow 
A\otimes V
\longrightarrow
A\otimes V^{\geq (k+1)}
\]
gives rise to the fibre-cofibre sequence of $X$-spectra
\[
\mathfrak{P}_A(A\otimes V^{\leq k})
\longrightarrow 
\mathfrak{P}_A(A\otimes V)
\longrightarrow
\mathfrak{P}_A(A\otimes V^{\geq (k+1)})\,.
\]
Taking fibres at $(x\colon \ast \to X)= \mathfrak{S}(A\to \mathbb{Q})$, by Lemma \ref{lem:FibofMinModSpec} we get the fibre-cofibre sequence of Eilenberg--Mac Lane spectra
$
H(V^{\leq k})^\vee
\leftarrow
HV
\leftarrow
H(V^{\geq (k+1)})^\vee
$.
Since $X$ is connected, this shows that $\mathfrak{P}_A(A\otimes V^{\leq k})$ is $(k+1)$-coconnective and $\mathfrak{P}_A(A\otimes V)\to \mathfrak{P}_A(A\otimes V^{\leq k})$ is an isomorphism on fibrewise stable homotopy groups in degrees $\leq k$.
As $P\sim_\mathrm{rat} \mathfrak{P}_A(A\otimes V)$ by assumption, it is clear that $\mathfrak{P}_A(A\otimes V)\to \mathfrak{P}_A(A\otimes V^{\leq k})$ is rationally equivalent to the truncation map $P\to P_{\leq k}$.
Similarly, $\mathfrak{P}_A(A\otimes V^{\geq (k+1)})\to \mathfrak{P}_A(A\otimes V)$ is rationally equivalent to the $(k+1)$-connective cover $P_{\geq(k+1)}\to P$.
\end{proof}

For a minimal $A$-module $A\otimes V$ as above,
the cofibre of $V^{\leq (k-1)}\to V^{\leq k}$ is a minimal $A$-module with underlying graded module of the form $A\otimes V^{=k}$ with $V^{=k} =\mathrm{span}\{v_\alpha\mid |v_\alpha| = k\}$.
In rational homotopy theory, $A\otimes V^{=k}$ corresponds to the $k$-th slice of an $X$-spectrum:

\begin{corollary}
\label{cor:kthSlice}
Let $A\otimes V$ be a minimal model for the rational homotopy type of $P\in Ho(\mathrm{Sp})^\mathbb{Q}_{\mathrm{f.t.,nil,bbl}}$.
Then $A\otimes V^{=k}$ is a minimal model for the $X$-spectrum $P_{=k}$.
\end{corollary}

\begin{remark}
\label{rem:ModularEM}
If the cdga $A$ is simply connected (so that $A^0 =\mathbb{Q}$ and $A^1 = 0$) then any minimal $A$-module of the form $A\otimes V^{=k}$ is free.
If $A$ is not simply connected, such minimal $A$-modules may have non-trivial differentials.
This situation is naturally interpreted as an algebraic version of Propsition \ref{prop:HeartofSpX}.
\end{remark}

\subsection{Monoidal products in rational homotopy theory}
We now explain how the rational homotopy theory equivalence of Theorem \ref{thm:RatParamHomThry} relates fibrewise smash products of parametrised spectra with derived tensor products of modules.
More precisely, for a cdga $A$ satisfying the standard hypotheses with spatial realisation $X=\mathfrak{S}(A)$ we show that the homotopy category $Ho(\mathrm{Sp}_X)^\mathbb{Q}_{\mathrm{f.t.,nil,bbl}}$ is closed under forming fibrewise smash products and moreover that for $P$, $Q$ nilpotent, bounded below $X$-spectra of finite rational type there is a natural isomorphism
\begin{equation}
\label{eqn:MonoidalComp}
\mathbf{L}\mathfrak{M}_A (P) \bigotimes^\mathbf{L}_A \mathbf{L}\mathfrak{M}_A(Q)
\xrightarrow{\;\;\cong\;\;}
\mathbf{L}\mathfrak{M}_A \big(P\wedge_X Q\big)
\end{equation}
in $Ho(A\mathrm{-Mod})$.

Our first aim is to construct the comparison map \eqref{eqn:MonoidalComp}.
In the following, we make use of various properties of the fibrewise smash product in $\mathrm{sSet}_{\dslash X}$ summarised in \cite[Section 1.1]{braunack-mayer_combinatorial_2020}.
In particular, we use that:
\begin{enumerate}[label=(\roman*)]
  \item $X^\ast S^0$ is the unit for $\wedge_X$;
  \item Pullbacks are strongly closed monoidal,  hence we have a projection formula for fibrewise smash products of retractive spaces; and
  \item $\mathrm{sSet}_\ast$-tensors of retractive spaces are expressed   in terms of the fibrewise smash product as $K\owedge_X Z\cong X^\ast K\wedge_X Z$ for $K\in \mathrm{sSet}_\ast$ and $Z\in \mathrm{sSet}_{\dslash X}$.
\end{enumerate}
Now, for retractive spaces $Y, Z$ over $X$ there is a commuting diagram in $\mathrm{sSet}_{\dslash X}$:
\begin{equation}
\label{eqn:ComponRetSp}
\begin{tikzcd}
Y\wedge_X Z
\ar[r, "\delta_{-1}"]
&
Y\wedge_X X^\ast X_! Z
\ar[r, shift left =4.5, "\delta_0"]
\ar[r, leftarrow, "\sigma_0"]
\ar[r, shift left = -4.5, "\delta_1"]
&
Y\wedge_X X^\ast (X_+ \wedge X_! Z)
\ar[r, shift left = 9, "\delta_0"]
\ar[r, leftarrow, shift left = 4.5, "\sigma_0"]
\ar[r, "\delta_1"]
\ar[r, leftarrow, shift left =-4.5, "\sigma_1"]
\ar[r, shift left = -9, "\delta_2"]
&
\dotsb
\end{tikzcd}
\end{equation}
In this diagram the augmentation $\delta_{-1}$ is the fibrewise smash product of $Y$ with the unit morphism $Z\to X^\ast X_! Z = X^\ast (Z/X)$.
The coface maps
\[
\delta_{0}\colon 
Y\wedge_X X^\ast (X_+^{\wedge(n-1)} \wedge Z/X)
\longrightarrow
Y\wedge_X X^\ast (X_+^{\wedge n} \wedge Z/X)
\]
are the results of forming fibrewise smash products of $Y\wedge_X X^\ast (X_+^{\wedge(n-1)} \wedge Z/X)$ with the component $X^\ast S^0 \to X^\ast X^\ast X_! X^\ast S^0 = X^\ast X_+$ of the  $(X_!\dashv X^\ast)$-unit at $X^\ast S^0$.
The remaining cosimplicial structure maps are obtained by applying the functor $Y\wedge_X X^\ast(-)$ to the diagram
\[
\begin{tikzcd}
Z/X
\ar[r, leftarrow, shift left = 1]
\ar[r, shift left = -1]
&
X_+ \wedge Z/X
\ar[r, leftarrow, shift left = 3]
\ar[r, shift left = 1]
\ar[r, leftarrow, shift left = -1]
\ar[r, shift left = -3]
&
X_+ \wedge X_+ \wedge Z/X
\ar[r, leftarrow, shift left = 5]
\ar[r, shift left = 3]
\ar[r, leftarrow, shift left = 1]
\ar[r, shift left = -1]
\ar[r, leftarrow, shift left = -3]
\ar[r, shift left = -5]
&
\dotsb
\end{tikzcd}
\]
constructed using the $X_+$-comodule structure maps of $Z/X$.

Applying $\mathfrak{M}^u_A \colon \mathrm{sSet}_{\dslash X}\to A\mathrm{-Mod}^\mathrm{op}$ to \eqref{eqn:ComponRetSp}, by Corollary \ref{cor:FinSmashModuleCompare} we get a commuting diagram of $A$-modules
\[
\begin{tikzcd}
&
\mathfrak{M}^u_A (Y)\otimes \mathfrak{M}^u_\mathbb{Q} (Z/X)
\ar[d]
\ar[r, leftarrow, shift left = 2]
\ar[r]
\ar[r, leftarrow, shift left = -2]
&
\mathfrak{M}^u_A (Y) \otimes \mathfrak{M}^u_\mathbb{Q} (X_+) \otimes \mathfrak{M}^u_\mathbb{Q} (Z/X)
\ar[d]
\ar[r, leftarrow, shift left = 4]
\ar[r, shift left = 2]
\ar[r, leftarrow]
\ar[r, shift left =-2]
\ar[r, leftarrow, shift left = -4]
&
\dotsb
\\
\mathfrak{M}^u_A (Y\wedge_X Z)
\ar[r, leftarrow]
&
\mathfrak{M}^u_A (Y\owedge_X Z/X)
\ar[r, leftarrow, shift left = 2]
\ar[r]
\ar[r, leftarrow, shift left = -2]
&
\mathfrak{M}^u_A (Y\owedge_X (X_+ \wedge Z/X))
\ar[r, leftarrow, shift left = 4]
\ar[r, shift left = 2]
\ar[r, leftarrow]
\ar[r, shift left =-2]
\ar[r, leftarrow, shift left = -4]
&
\dotsb
\end{tikzcd}
\]
in which the vertical arrows are quasi-isomorphisms.
The cochain complex $\mathfrak{M}^u_\mathbb{Q}(Z/X)$ coincides with the underlying cochain complex of the $A$-module $\mathfrak{M}^u_A(Z)$ and under our assumptions there is a quasi-isomorphism $A \to \mathfrak{A}(X) = \mathfrak{M}^u_\mathbb{Q}(X_+)$.
We thus have a commuting diagram of $A$-modules
\[
\begin{tikzcd}
\mathfrak{M}^u_A (Y)\otimes \mathfrak{M}^u_A(Z)
\ar[d]
\ar[r, leftarrow, shift left = 2]
\ar[r]
\ar[r, leftarrow, shift left = -2]
&
\mathfrak{M}^u_A (Y) \otimes A \otimes \mathfrak{M}^u_A(Z)
\ar[d]
\ar[r, leftarrow, shift left = 4]
\ar[r, shift left = 2]
\ar[r, leftarrow]
\ar[r, shift left =-2]
\ar[r, leftarrow, shift left = -4]
&
\dotsb
\\
\mathfrak{M}^u_A (Y)\otimes \mathfrak{M}^u_\mathbb{Q} (Z/X)
\ar[r, leftarrow, shift left = 2]
\ar[r]
\ar[r, leftarrow, shift left = -2]
&
\mathfrak{M}^u_A (Y) \otimes \mathfrak{M}^u_\mathbb{Q} (X_+) \otimes \mathfrak{M}^u_\mathbb{Q} (Z/X)
\ar[r, leftarrow, shift left = 4]
\ar[r, shift left = 2]
\ar[r, leftarrow]
\ar[r, shift left =-2]
\ar[r, leftarrow, shift left = -4]
&
\dotsb
\end{tikzcd}
\]
in which the vertical arrows are quasi-isomorphisms.
A careful analysis of the construction shows that all simplicial structure maps in the top diagram except for the zeroth face maps arise from the $A$-action on $\mathfrak{M}^u_A(Z)$ and the unit map $\mathbb{Q}\to A$, with the zeroth face maps given by the $A$-module structure on $\mathfrak{M}^u_A(Y)$.
We have thus constructed an augmented simplicial diagram of $A$-modules
\[
B_\bullet(\mathfrak{M}^u_A(Y);A;\mathfrak{M}^u_A(Z))\longrightarrow
\mathfrak{M}^u_A(Y\wedge_X Z)
\]
with $B_\bullet(\mathfrak{M}^u_A(Y);A;\mathfrak{M}^u_A(Z))$ the two-sided simplicial bar construction.
Passing to the homotopy colimit, we have proven the following
\begin{lemma}
\label{lem:UCompSmash}
Let $A$ be a cofibrant connected cdga of finite homotopical type, with $Y$ and $Z$ retractive spaces over $X =\mathfrak{S}(A)$.
In the homotopy category of $A$-modules, there is a morphism
\[
\mathbf{L}\mathfrak{M}^u_A (Y)\bigotimes^\mathbf{L}_A \mathbf{L}\mathfrak{M}^u_A (Z)
\longrightarrow
\mathbf{L}\mathfrak{M}^u_A (Y\wedge_X Z)
\]
that is natural in both arguments.
\end{lemma}

Having dispensed with the unstable setting, we now construct the comparison map for fibrewise smash products of $X$-spectra.
Fixing $X$-spectra $P$ and $Q$, by \cite[Lemma 2.19]{braunack-mayer_combinatorial_2020} there are sequences $\{Y_i\}$ and $\{Z_i\}$ of retractive spaces over $X$ such that
$
P\cong \mathrm{hocolim}_i \Sigma^{\infty-i}_X Y_i
$
and 
$Q\cong \mathrm{hocolim}_i \Sigma^{\infty-i}_X Z_i
$.
Examining the argument of \emph{loc.~cit.}~shows that this colimit description can be made sufficiently natural in $P$ and $Q$.
The functor $\mathfrak{M}_A$ turns homotopy colimits of $X$-spectra into homotopy limits of $A$-modules so by Lemmas \ref{lem:UCompSmash} and \ref{lem:SuspendInMod} we get a morphism
\[
\mathbf{L}\mathfrak{M}_A (P)
\bigotimes^\mathbf{L}_A
\mathbf{L}\mathfrak{M}_A (Q)
\longrightarrow
\underset{i,j}{\mathrm{holim}}\,\bigg(
\mathbf{L}\mathfrak{M}^u_A (Y_i)[i]
\bigotimes^\mathbf{L}_A
\mathbf{L}\mathfrak{M}^u_A (Z_j)[j]
\bigg)
\longrightarrow
\underset{i,j}{\mathrm{holim}}\,
\mathbf{L}\mathfrak{M}^u_A (Y_i\wedge_X Z_j)[i+j]\,.
\]
But $P\wedge_X Q \cong \mathrm{hocolim}_{i,j} \Sigma^{\infty-(i+j)}_X (Y_i\wedge_X Z_j)$ so that the codomain of the above morphism may be identified with $\mathbf{L}\mathfrak{M}_A(P\wedge_X Q)$.
This completes the proof of the following
\begin{lemma}
\label{lem:CompSmash}
Let $A$ be a cofibrant connected cdga of finite homotopical type, with $P$ and $Q$ parametrised spectra over $X =\mathfrak{S}(A)$.
In the homotopy category of $A$-modules, there is a morphism
\[
\mathbf{L}\mathfrak{M}_A (P)\bigotimes^\mathbf{L}_A \mathbf{L}\mathfrak{M}_A (Q)
\longrightarrow
\mathbf{L}\mathfrak{M}_A (P\wedge_X Q)
\]
that is natural in both arguments.
\end{lemma}

Next, we show that the homotopy category $Ho(\mathrm{Sp}_X)^\mathbb{Q}_{\mathrm{f.t.,nil,bbl}}$ is closed under forming fibrewise smash products.
Let $P$ be a nilpotent, bounded below $X$-spectrum of finite rational type and let 
\[
\mathcal{C}_P :=
\big\{J\in Ho(\mathrm{Sp}_X)^\mathbb{Q}_{\mathrm{f.t.,nil,bbl}} \mid P\wedge_X J \in Ho(\mathrm{Sp}_X)^\mathbb{Q}_{\mathrm{f.t.,nil,bbl}}\big\}\,.
\]
The goal is to show $\mathcal{C}_P$ contains all objects of $Ho(\mathrm{Sp}_X)^\mathbb{Q}_{\mathrm{f.t.,nil,bbl}}$.
Observe that:
\begin{itemize}
  \item $\mathcal{C}_P$ contains $X^\ast H\mathbb{Q}$, as there is a rational equivalence $P\wedge_X X^\ast H\mathbb{Q}\sim_\mathrm{rat} P$.
  
  \item $\mathcal{C}_P$ is closed under forming finite sums, for if $J_1, \dotsc, J_n\in \mathcal{C}_P$ then
  \[
  P\wedge_X\bigg(\bigoplus_{i=1}^n J_i\bigg) \cong \bigoplus_{i=1}^j P\wedge_X J_i
  \]
  and $Ho(\mathrm{Sp}_X)^\mathbb{Q}_{\mathrm{f.t.,nil,bbl}}$ is closed under finite sums (cf.~Corollary \ref{cor:NilSum}).
  
  \item $\mathcal{C}_P$ is closed under shifts; for any integer $k$ and $J\in\mathcal{C}_P$ we have $P\wedge_X \Sigma^k_X J\cong \Sigma^k_X(P\wedge_X J)$, which is in $Ho(\mathrm{Sp}_X)^\mathbb{Q}_{\mathrm{f.t.,nil,bbl}}$ since this category is closed under shifts.
  
  \item $\mathcal{C}_P$ is closed under forming fibre sequences. Indeed, suppose that $F\to J_1, \to J_2$ is a fibre(-cofibre) sequence of $X$-spectra with $J_1, J_2\in \mathcal{C}_P$.
  The fibrewise smash product is exact in each variable, so the we have a fibre(-cofibre) sequence $P\wedge_X F\to P\wedge_X J_1\to P\wedge_X J_2$.
  The $X$-spectrum $P\wedge_X F$ is bounded below and of finite rational type, and is moreover nilpotent by Proposition \ref{prop:NilXSpecSeq}.
  Hence $F\in \mathcal{C}_P$.
\end{itemize}
By Corollary \ref{cor:NilBblFinRatType}, these properties imply that $\mathcal{C}_P$ contains all \emph{bounded} objects of $Ho(\mathrm{Sp}_X)^\mathbb{Q}_{\mathrm{f.t.,nil,bbl}}$.
To remove the boundedness condition, fix some integer $l$ and suppose without loss of generality that $P$ is $m$-connective.
Connectivity is additive under forming smash products\footnote{The smash product of an $m$-connective spectrum $A$ with an $n$-connective spectrum $B$ is $(m+n)$-connective.}, hence taking fibres at $x\in X$ we have that
\[
x^\ast (P\wedge_X Q_{\geq (l+1)})
\cong 
x^\ast P
\wedge
(x^\ast Q)_{\geq (l+1)}
\]
is $(m+l+1)$-connective and hence $P\wedge_X Q_{\geq (l+1)}$ is an $(m+l+1)$-connective $X$-spectrum.
Forming the fibrewise smash product of $P$ with the fibre-cofibre sequence $Q_{\geq (l+1)}\to Q\to Q_{\leq l}$ yields the fibre-cofibre sequence
\[
P\wedge_X Q_{\geq (l+1)}
\longrightarrow
P\wedge_X Q
\longrightarrow
P\wedge_X Q_{\leq l}\,.
\]
The map $P\wedge_X Q\to P\wedge_X Q_{\leq l}$ induces fibrewise stable homotopy groups in dimensions $\leq (m+l)$ and hence $(P\wedge_X Q)_{\leq (m+l)}\cong (P\wedge_X Q_{\leq l})_{\leq (m+l)}$.
As $Q_{\leq l}$ is bounded we have $Q_{\leq l}\in \mathcal{C}_P$ and so $(P\wedge_X Q)_{\leq (m+l)}\cong (P\wedge_X Q_{\leq l})_{\leq (m+l)}$ is a nilpotent, bounded below $X$-spectrum of finite rational type.
Taking $l$ arbitarily large, Corollary \ref{cor:NilTrunc} implies that $P\wedge_X Q$ is nilpotent.
Since this $X$-spectrum is clearly bounded below and of finite rational type, we have proven the following
\begin{lemma}
\label{lem:NilSmashClosure}
$Ho(\mathrm{Sp}_X)^\mathbb{Q}_{\mathrm{f.t.,nil,bbl}}$ is closed under forming fibrewise smash products of $X$-spectra.
\end{lemma}

Similarly, in the algebraic setting we show that $A$-modules of finite homotopical type are closed under derived tensor products:
\begin{lemma}
\label{lem:finAmodTens}
The derived tensor product of $A$-modules restricts to a symmetric monoidal structure on $Ho(A\mathrm{-Mod})_{\mathrm{f.h.t.}}$.
\end{lemma}
\begin{proof}
$A\mathrm{-Mod}$ is a symmetric monoidal model category with respect to the Quillen bifunctor $(M, N)\mapsto M\bigotimes_A N$.
If $M$ and $N$ are cofibrant $A$-modules, it follows that the derived tensor product $M\bigotimes^\mathbf{L}_A N$ is (the homotopy type of) the ordinary tensor product $M\bigotimes_A N$.

Suppose that $M$ and $N$ are of finite homotopical type with minimal models $A\otimes V\to M$ and $A\otimes W\to N$.
The graded rational vector spaces $V, W$ are bounded below and finite dimensional in each degree and we have
\[
(A\otimes V)\bigotimes_A (A\otimes W)\cong A\otimes (V\otimes W)\,,
\]
with differential on the right-hand side induced by the differentials on $A\otimes V$ and $A\otimes W$ via the Leibniz rule.
Minimal models are cofibrant so that $A\otimes (V\otimes W)$ models the derived tensor product $M\bigotimes_A^\mathbf{L} N$.

%To complete the proof, we must show that $A\otimes(V\otimes W)$ is of finite homotopical type.
%This follows at once from the fact that $A\otimes(V\otimes W)$ is a minimal $A$-module, however to avoid getting bogged down in the combinatorics of proving minimality we follow a less direct path.
Fixing a minimal model $A\otimes U\to A\otimes (V\otimes W)$, taking extensions of scalars along the augmentation $a\colon A\to \mathbb{Q}$ gives a map of cochain complexes $U\to V\otimes W$.
As $A\otimes U$ and $A\otimes (V\otimes W)$ are cofibrant $A$-modules, $U\to V\otimes W$ is a quasi-isomorphism.
The differential on $U$ is trivial by minimality and hence, since $V\otimes W$ certainly has finite dimensional cohomology in each degree, $U$ is finite dimensional.
This shows that homotopy category of $A$-modules of finite homotopical type is closed under forming derived tensor products.
\end{proof}

We now turn to the main result of this section.
\begin{theorem}
\label{thm:SmashComp}
Let $A$ be a cofibrant connected cdga of finite homotopical type with $X= \mathfrak{S}(A)$.
The equivalence of rational homotopy theories 
\[
\begin{tikzcd}
Ho(\mathrm{Sp}_X)^\mathbb{Q}_{\mathrm{f.t.,nil,bbl}}
\ar[rr, shift left =2, "\mathbf{L}\mathfrak{M}_A"]
\ar[rr, shift left =-2, leftarrow, "\simeq", "\mathbf{R}\mathfrak{P}_A"']
&& 
Ho(A\mathrm{-Mod})^\mathrm{op}_\mathrm{f.h.t.}
\end{tikzcd}
\]
identifies the fibrewise smash product with the derived tensor product of $A$-modules.
That is, for any nilpotent, bounded below $X$-spectra $P, Q$ of finite rational type there is an equivalence of $A$-modules
\[
\mathbf{L}\mathfrak{M}_A (P)\bigotimes^\mathbf{L}_A \mathbf{L}\mathfrak{M}_A (Q)
\longrightarrow
\mathbf{L}\mathfrak{M}_A (P\wedge_X Q)
\]
which is natural in both arguments.
\end{theorem}
\begin{proof}
In view of Lemmas \ref{lem:CompSmash}, \ref{lem:NilSmashClosure}, and \ref{lem:finAmodTens}, it suffices to show that for all $P, Q\in Ho(\mathrm{Sp}_X)^\mathbb{Q}_{\mathrm{f.t.,nil,bbl}}$ the comparison map
\[
\kappa_{P,Q}\colon
\mathbf{L}\mathfrak{M}_A (P)\bigotimes^\mathbf{L}_A \mathbf{L}\mathfrak{M}_A (Q)
\longrightarrow
\mathbf{L}\mathfrak{M}_A (P\wedge_X Q)
\]
is an equivalence.
To this end we define
\[
\mathcal{D}_P :=
\left\{J\in Ho(\mathrm{Sp}_X)^\mathbb{Q}_{\mathrm{f.t.,nil,bbl}} 
\;\big|\; 
\kappa_{P,J}\text{ is an equivalence}\right\}
\,.
\]
We aim to show that $\mathcal{D}_P$ contains all nilpotent, bounded below $X$-spectra of finite rational type.
Writing $M := \mathbf{L}\mathfrak{M}_A(P)$ for brevity, observe that:
\begin{itemize}
  \item $\mathcal{D}_P$ contains $X^\ast H\mathbb{Q}$; in this case  $\kappa_{P,X^\ast H\mathbb{Q}}$ is equivalent the evident map
  $
  M\bigotimes^\mathbf{L}_A A
  \to
  M
  $
  which is an equivalence.
  
  \item $\mathcal{D}_P$ is closed under forming finite sums; if $J_1,\dotsc, J_n\in \mathcal{D}_P$ then
  \begin{align*}
  \mathbf{L}\mathfrak{M}_A\bigg( P\wedge_X \bigoplus_{i=1}^n J_i\bigg)
  &\cong
  \mathbf{L}\mathfrak{M}_A\bigg( \bigoplus_{i=1}^n P\wedge_X  J_i\bigg)
  \\
  &\cong
  \bigoplus_{i=1}^n \mathbf{L}\mathfrak{M}_A\big( P\wedge_X J_i\big)
  \\
  &\!\!\!\!\overset{J_i\in \mathcal{D}_P}{\cong}
  \bigoplus_{i=1}^n \bigg(M\bigotimes^\mathbf{L}_A \mathbf{L}\mathfrak{M}_A(J_i)\bigg)\\
  &\cong
  M\bigotimes^\mathbf{L}_A \bigg(\bigoplus_{i=1}^n \mathbf{L}\mathfrak{M}_A(J_i)\bigg)\,.
  \end{align*}
  
  \item $\mathcal{D}_P$ is closed under shifts; for any integer $k$ and $J\in \mathcal{D}_P$ we have 
  \begin{align*}
  \mathbf{L}\mathfrak{M}_A \big( P\wedge_X \Sigma^k_X J\big)
  &\cong
  \mathbf{L}\mathfrak{M}_A \big(\Sigma^k_X  P\wedge_X J\big)\\
  &\cong
  \mathbf{L}\mathfrak{M}_A \big(P\wedge_X J\big)[-k]\\
  &\!\!\!\!\overset{J\in \mathcal{D}_P}{\cong}
  \bigg(M\bigotimes^\mathbf{L}_A \mathbf{L}\mathfrak{M}_A(J)\bigg)[-k]\\
  &\cong
  M\bigotimes^\mathbf{L}_A \mathbf{L}\mathfrak{M}_A(J)[-k]\,.
  \end{align*}
  
  \item $\mathcal{D}_P$ is closed under forming fibre sequences; suppose that $F\to J_1\to J_2$ is a fibre-cofibre sequence of $X$-spectra with $J_1, J_2\in \mathcal{D}_P$.
  Forming fibrewise smash products with $P$ we get a (fibre)-cofibre sequence of $X$-spectra
  \[
  P\wedge_X  F
  \longrightarrow
  P\wedge_X J_1 
  \longrightarrow 
  P\wedge_X J_2
  \]
  that is sent by $\mathbf{L}\mathfrak{M}_A$ to the fibre(-cofibre) sequence comprising the bottom row of the commuting diagram of $A$-modules
  \[
  \begin{tikzcd}
  M\bigotimes^\mathbf{L}_A \mathbf{L}\mathfrak{M}_A(J_2)
  \ar[r]
  \ar[d, "\sim_\mathrm{qis}", "\kappa_{P,J_2}"']
  &
  M\bigotimes^\mathbf{L}_A \mathbf{L}\mathfrak{M}_A(J_1)
  \ar[r]
  \ar[d, "\sim_\mathrm{qis}", "\kappa_{P,J_1}"']
  &
   M\bigotimes^\mathbf{L}_A \mathbf{L}\mathfrak{M}_A(F)
  \ar[d, "\kappa_{P,K}"]
  \\
  \mathbf{L}\mathfrak{M}_A(P \wedge_X J_2)
  \ar[r]
  &
  \mathbf{L}\mathfrak{M}_A(P \wedge_X J_1)
  \ar[r]
  &
  \mathbf{L}\mathfrak{M}_A(P \wedge_X F)
  \,.
  \end{tikzcd}
  \]
  The left-hand and middle vertical arrows are equivalences in $Ho(A\mathrm{-Mod})$ by assumption on $J_1, J_2$.
  But $\mathbf{L}\mathfrak{M}_A(F)$ is the cofibre of $\mathbf{L}\mathfrak{M}_A(J_2)\to\mathbf{L}\mathfrak{M}_A(J_1)$, so that the top row of the above diagram is a fibre-cofibre sequence of $A$-modules and hence $\kappa_{P,F}$ is an equivalence.
\end{itemize}
By Corollary \ref{cor:NilBblFinRatType}, these properties imply that $\mathcal{D}_P$ contains all bounded objects of $Ho(\mathrm{Sp}_X)^\mathbb{Q}_{\mathrm{f.t.,nil,bbl}}$.

To remove the boundedness condition, we argue along the same lines as Lemma \ref{lem:CompSmash}.
Let $Q$ be a nilpotent, bounded below $X$-spectrum of finite rational type.
We may suppose without loss of generality that $P$ is $m$-connective, then for any integer $l$ the fibrewise smash product $P\wedge_X Q_{\geq (l+1)}$ is $(m+l+1)$-connective.
As $Q_{\leq l}$ is bounded, by Lemma \ref{lem:CompSmash} and the above, the fibre-cofibre sequence $Q_{\geq (l+1)}\to Q \to Q_{\leq l}$ determines a commuting diagram in $Ho(A\mathrm{-Mod})$
\[
  \begin{tikzcd}
  M\bigotimes^\mathbf{L}_A \mathbf{L}\mathfrak{M}_A(Q_{\leq l})
  \ar[r]
  \ar[d, "\sim_\mathrm{qis}", "\kappa_{P,Q_{\leq l}}"']
  &
  M\bigotimes^\mathbf{L}_A \mathbf{L}\mathfrak{M}_A(Q)
  \ar[r]
  \ar[d, "\kappa_{P,Q}"']
  &
   M\bigotimes^\mathbf{L}_A \mathbf{L}\mathfrak{M}_A(Q_{\geq(l+1)})
  \ar[d, "\kappa_{P,Q_{\geq(l+1)}}"]
  \\
  \mathbf{L}\mathfrak{M}_A(P \wedge_X Q_{\leq l})
  \ar[r]
  &
  \mathbf{L}\mathfrak{M}_A(P \wedge_X Q)
  \ar[r]
  &
  \mathbf{L}\mathfrak{M}_A(P \wedge_X Q_{\geq (l+1)})
  \,.
  \end{tikzcd}
\]
in which the left-hand vertical arrow is an equivalence.
The top and bottom horizontal rows are fibre-cofibre sequences, and the $A$-modules $M\bigotimes^\mathbf{L}_A \mathbf{L}\mathfrak{M}_A(Q_{\geq (l+1)})$ and $\mathbf{L}\mathfrak{M}_A(P\wedge_X Q_{\geq (l+1)})$ are both $(m+l+1)$-connective (by the properties of the derived tensor product and Corollary \ref{cor:ConnectiveUnderM}, respectively).
It follows that $\kappa_{P, Q}$ induces an isomorphism in cohomology in degrees $\leq (m+l)$.
Taking $l$ arbitrarily large shows that $\kappa_{P,Q}$ induces an isomorphism in cohomology and is thus an equivalence in $Ho(A\mathrm{-Mod})$.
\end{proof}

\begin{remark}
\label{rem:EMSS}
In the construction of the map \eqref{eqn:MonoidalComp} 
we used the fact that the derived tensor product of $A$-modules $M$ and $N$ is presented by the two sided simplicial bar construction; that is, $M\bigotimes^\mathbf{L}_A N$ is presented by the total complex of the normalised chain complex of $B_k(M;A;N) = M\otimes A^{\otimes k}\otimes N$.
The skeletal filtration of $B_\bullet(M;A;N)$ induces a spectral sequence of Eilenberg--Moore type with $E_2$-term
\[
E_2 = \mathrm{Tor}_{H^\bullet(A)} (H^\bullet(M), H^\bullet(N))\,. 
\]
If $A$ is simply connected ($A^0=\mathbb{Q}$ and $A^1 = 0$) the filtration is bounded and the spectral sequence converges strongly to $\mathrm{Tor}^\bullet_A(M,N)$.
\end{remark}

\begin{example}
Let $Y$ and $X$ be of finite rational type with nilpotent fibrations $Y\to X$ and $Z\to X$ be nilpotent fibrations over $X = \mathfrak{S}(A)$ with.
For $k, l\geq 0$ there is an equivalence of $X$-spectra
\[
\Sigma^{\infty-k}_{X+} Y \wedge_X \Sigma^{\infty-l}_{X+} Z 
\cong
\Sigma^{\infty-(k+l)}_{X+} (Y\times_X Z)\,. 
\]
By Proposition \ref{prop:FibSuspendinMod} and Theorem \ref{thm:SmashComp} the rational homotopy type of this $X$-spectrum is identified with the $A$-module
\[
\mathfrak{A}(Y)[k]\bigotimes^\mathbf{L}_A \mathfrak{A}(Z)[l]\cong
\bigg(\mathfrak{A}(Y)\bigotimes^\mathbf{L}_A \mathfrak{A}(Z)\bigg)[k+l]\,.
\]
In this case the spectral sequence of Remark \ref{rem:EMSS} is
\[
\mathrm{Tor}_{H^\bullet(X)} (H^\bullet(Y)[k], H^\bullet(Z)[l])
\Longrightarrow
H^\bullet(Y\times_X Z)[k+l]\,,
\]
the familiar Eilenberg--Moore spectral sequence, up to a shift.
\end{example}
  
\subsection{Rational homotopy classes of fibrewise stable maps}
\label{ss:RatFibStabMap}
Let $P$ and $Q$ be nilpotent, bounded below $X$-spectra of finite rational type. 
The $\mathbb{Z}$-graded rational vector space of fibrewise stable maps from $P$ to $Q$ is given in degree $k$ by
\begin{align*}
\{P, Q\}_X^k \otimes_\mathbb{Z}\mathbb{Q} 
&\cong
\spi_{-k} \big(X_\ast F_X(P,Q)\big)\otimes_\mathbb{Z}\mathbb{Q} 
\\
&\cong
Ho(\mathrm{Sp}_X)^\mathbb{Q}_{\mathrm{f.t.,nil,bbl}}(\Sigma^{-k}_X P , Q)\,.
\end{align*}
Under the equivalence of rational homotopy theories of Theorem \ref{thm:RatParamHomThry}, $P$ and $Q$ correspond to $A$-modules $M \cong \mathbf{L}\mathfrak{M}_A(P)$ and $N\cong \mathbf{L}\mathfrak{M}_A(Q)$ and we have equivalences
\[
Ho(\mathrm{Sp}_X)^\mathbb{Q}_{\mathrm{f.t.,nil,bbl}}(\Sigma^{-k}_X P , Q)
\cong 
Ho(A\mathrm{-Mod})_{\mathrm{f.h.t.}}\big(N, M[k]\big)
\cong \mathrm{Ext}^k_A (N,M)\,.
\]
Hence we obtain the following
\begin{proposition}
\label{prop:FibStapMapandExt}
Let $A$ be a cofibrant connected cdga of finite homotopical type with spatial realisation $X = \mathfrak{S}(A)$.
For nilpotent, bounded below $X$-spectra $P, Q$ of finite rational type corresponding in rational homotopy theory to the $A$-modules $M$ and $N$, respectively, there is an isomorphism of $\mathbb{Z}$-graded rational vector spaces
\[
\{P,Q\}^\ast_X \otimes_\mathbb{Z}\mathbb{Q}\cong \mathrm{Ext}^\ast_A (N,M)
\]
natural in both arguments.
\end{proposition}

Since the homotopy theory of $A$-modules is fairly simple we are able to go a fair way beyond Proposition \ref{prop:FibStapMapandExt}.
For any pair of $A$-modules $N, M$ there is a cochain complex of $A$-linear maps
\[
[N,M]_A^k = A\mathrm{-Mod}(N, M[k])
\]
with differential 
$
d\lambda = d_M  \circ \lambda - (-1)^k \lambda \circ d_N
$
for $\lambda \colon N \to M[k]$.
In the case of modules over $\mathbb{Q}$ we simply write $[N,M]$ instead of $[N,M]_\mathbb{Q}$.

The complexes $[N,M]_A$ enjoy good homotopical properties; for instance if $f\colon N'\to N$ is a cofibration and $g\colon M\to M'$ is a fibration then the induced map of hom complexes
\[
[N, M]_A
\longrightarrow
[N, M']_A \prod_{[N', M']_A} [N', M]_A
\]
is a fibration which is moreover a quasi-isomorphism if either of $f$ or $g$ is.
For $A$-modules $N, M$ the rational vector spaces $\mathrm{Ext}_A^\ast(N,M)$ can be computed in terms of the derived hom complexes as
\begin{equation}
\label{eqn:ExtDescr}
\mathrm{Ext}_A^k(N,M) \cong H^k [N^c, M]_A \cong H^0 [N^c, M[k]]_A
\end{equation}
with $N^c \to N$ any cofibrant replacement in $A\mathrm{-Mod}$.
The result is independent of the choice of cofibrant replacement and depends on $M$ and $N$ only up to quasi-isomorphism.

We consider two spectral sequences for $\mathrm{Ext}_A$ arising from two different cofibrant resolutions in $A\mathrm{-Mod}$.
We first  consider a \lq\lq hyper-Ext'' spectral sequence arising from the bar resolution.
For an $A$-module $N$, consider the simplicial $A$-module $B_k(A;A;N) = A\otimes A^{\otimes k} \otimes N$, with face and degeneracy maps induced by the $A$-module structures on $A$ and $N$.
Taking normalised chains yields a double complex whose total complex is the \emph{bar construction} $B(A;A;N)$.
The bar construction $B(A;A;N)$ is a cofibrant $A$-module equipped with a quasi-isomorphism of $A$-modules $B(A;A;N)\to N$.

Filtering the simplicial bar construction $B_\bullet(A;A;N)$ by its skeleta gives rise to an exhaustive filtration
\[
\mathcal{F}_0
\subset
\mathcal{F}_1
\subset
\dotsb
\subset
\mathcal{F}_p
\subset 
\dotsb
\subset B(A;A;N)\,.
\]
For each $p\geq 0$ the inclusion $\mathcal{F}_p\to \mathcal{F}_{p+1}$ is a cofibration of $A$-modules and there is a (homotopy) cofibre sequence of $A$-modules
\[
\begin{tikzcd}
\mathcal{F}_p
\ar[r, rightarrowtail]
&
\mathcal{F}_{p+1}
\ar[r]
&
A\otimes \overline{A}{}^{\otimes(p+1)}\otimes N [p+1]\,,
\end{tikzcd}
\]
with cofibre the free $A$-module on $\overline{A}{}^{\otimes(p+1)}\otimes N[p+1]$, for $\overline{A} = \ker(A\to \mathbb{Q})$ the augmentation ideal with respect to the canonical augmentation (cf.~Remark \ref{rem:StandingAssumptions}).
Forming cochain complexes of $A$-linear maps, we get a sequence of fibrations of cochain complexes
\begin{equation}
\label{eqn:BarResFibSeq}
\begin{tikzcd}[sep=small]
\dotsb
\ar[r, twoheadrightarrow]
&
{[\mathcal{F}_{p+1}, M]_A}
\ar[r, twoheadrightarrow]
&
{[\mathcal{F}_p, M]_A}
\ar[r, twoheadrightarrow]
&
\dotsb
%\ar[r, twoheadrightarrow]
%&
%{[\mathcal{F}_1, M]_A}
%\ar[r, twoheadrightarrow]
%&
%{[\mathcal{F}_0, M]_A}
\\
&
{[A\otimes \overline{A}{}^{\otimes(p+1)}\otimes N [p+1], M]_A}
\ar[u]
&
{[A\otimes \overline{A}{}^{\otimes p}\otimes N [p], M]_A}
\ar[u]
&
%&
%{[A\otimes \overline{A}^{\otimes(p+1)}\otimes N [-1], M]_A}
%\ar[u]
%&
\end{tikzcd}
\end{equation}
with (homotopy) fibres as indicated.
The complexes $[A\otimes \overline{A}{}^{\otimes p}\otimes N [p], M]_A \cong [\overline{A}{}^{\otimes p}\otimes N [p], M]$ have cohomology
\[
H^q [\overline{A}{}^{\otimes p}\otimes N[p], M]
\cong 
[H^\bullet (\overline{A}{}^{\otimes p}\otimes N),H^\bullet( M)[-p]]^q
\cong 
[H^\bullet (\overline{A}{}^{\otimes p}\otimes N),H^\bullet( M)]^{q-p}\,,
\] 
that is, the vector space of degree-$(q-p)$ maps $H^\bullet (\overline{A}{}^{\otimes p}\otimes N)\to H^\bullet( M)$.
Passing to cohomology, the sequence of fibrations \eqref{eqn:BarResFibSeq} gives an exact couple
\begin{equation}
\label{eqn:HyperExtExactCouple}
\begin{tikzcd}
{H^q [\mathcal{F}_p, M]_A}
\ar[r, "i"]
\ar[from = d, "k"]
&
{H^q [\mathcal{F}_{p-1}, M]_A}
\ar[dl, bend left = 10, dashed, "j"]
\\
{[H^\bullet (\overline{A}{}^{\otimes p}\otimes N),H^\bullet( M)]^{q-p}}
\end{tikzcd}
\end{equation}
and hence a spectral sequence with $E_1^{p,\bullet} = [H^\bullet (\overline{A}{}^{\otimes p}\otimes N),H^\bullet( M)]^{\bullet-p}$.
The differential $d_1$ is the composite
\[
\begin{tikzcd}[row sep = tiny]
H^\bullet {[A\otimes \overline{A}{}^{\otimes p} \otimes N[p], M]_A}
\ar[r, "k"]
\ar[d, equals,  "\wr"]
&
H^\bullet {[\mathcal{F}_p, A]_A}
\ar[r, "j"]
&
H^{\bullet+1} {[A\otimes \overline{A}{}^{\otimes (p+1)} \otimes N[p+1], M]_A}
\ar[d, equals, "\wr"]
\\
{[H(\overline{A}){}^{\otimes p} \otimes H(N), H(M)]^{\bullet - p}}
\ar[rr, "d_1^{p, \bullet}"']
&&
{[H(\overline{A})^{\otimes (p+1)} \otimes H(N), H(M)]^{\bullet - p}}
\end{tikzcd}
\]
and we identify the $E_2$-term as follows.
Regarding $H^\bullet(A)$ as a cdga with trivial differential, the simplicial bar construction 
\[
\underbrace{
\begin{tikzcd}[ampersand replacement = \&]
\dotsb
\ar[r, shift left = 4]
\ar[r, leftarrow, shift left = 2]
\ar[r]
\ar[r, leftarrow, shift left = -2]
\ar[r, shift left = -4]
\&
H(A)\otimes H(A)\otimes H(N)
\ar[r, shift left = 2]
\ar[r, leftarrow]
\ar[r, shift left = -2]
\&
H(A)\otimes H(N)
\end{tikzcd}
}_{\vphantom{\Big(}B_\bullet(H(A);H(A);H(N))}
\longrightarrow
H(N)
\]
is a Reedy cofibrant simplicial object of $H(A)\mathrm{-Mod}$.
The normalised complex
\[
\begin{tikzcd}
\dotsb
\ar[r]
&
H(A)\otimes H(\overline{A})\otimes H(N)
\ar[r]
&
H(A)\otimes H(N)
\ar[r]
&
H(N)
\end{tikzcd}
\]
is a free resolution of $H(N)$ as a (differential graded) $H(A)$-module and forming complexes of $H(A)$-linear maps to $H(M)$ yields the complex of cochain complexes
\[
\begin{tikzcd}
{[H(N),H(M)]_{H(A)}^\bullet}
\ar[r]
&
{[H(N), H(M)]^\bullet}
\ar[r, "\delta^{0,\bullet}"]
&
{[H(\overline{A})\otimes H(N), H(M))]^\bullet}
\ar[r, "\delta^{1,\bullet}"]
&
\dotsb
\end{tikzcd}
\]
in which each of the complexes $[H(\overline{A})^{\otimes s}\otimes H(N), H(M))]$ has trivial differential.
Forming the $\delta$-cohomology groups of this complex gives the bigraded vector spaces
\[
\mathrm{Ext}^{s,\bullet}_{H(A)} (H(N), H(M)) \cong \ker(\delta^{s,\bullet})/\mathrm{im}(\delta^{s-1,\bullet})\,.
\]
Both differentials $d_1^{\bullet, \bullet}$ and $\delta^{\bullet, \bullet}$ arise from skeletal filtrations of bar constructions (over $A$ and $H(A)$ respectively) and a careful analysis of the constructions shows that $d_1^{p,q} = \delta^{p,q-p}$.
We may therefore identify the $E_2$-term as
\[
E_2^{p,q} = \mathrm{Ext}_{H(A)}^{p, q-p}(H(N), H(M))\,.
\] 
We have so far obtained a half-plane spectral sequence with entering differentials.
To settle the issue of converge, we impose some additional conditions on $A$, $M$ and $N$.
Recall that a cdga $A$ is \emph{simply connected} if $A^0 =\mathbb{Q}$ and $A^1 = 0$; these conditions guarantee that the spatial realisation $X=\mathbf{R}\mathfrak{S}(X)$ is also simply connected.
\begin{proposition}[Hyper-Ext spectral sequence]
\label{prop:HyperExt}
Let $A$ be a simply connected cdga.
For $A$-modules $M$ and $N$ with cohomology groups that are bounded above and bounded below respectively, there is a strongly convergent spectral sequence
\[
\mathrm{Ext}^{p,q}_{H(A)}(H(N), H(M))
\Longrightarrow
\mathrm{Ext}^{p+q}_A(N,M)\,.
\]
\end{proposition}
\begin{proof}
Referring to the exact couple \eqref{eqn:HyperExtExactCouple}, let 
\[
Z^{s, \bullet}_r = k^{-1}\big(\mathrm{im} \{
\!\!
\begin{tikzcd}[sep = small]
H^\bullet {[\mathcal{F}_{s+r-1},M]_A }
\ar[r, "i"]
&
\dotsb
\ar[r, "i"]
&
H^\bullet {[\mathcal{F}_{s},M]_A }
\end{tikzcd}
\!\!
\}\big)
\subset
[H(\overline{A}^{\otimes s}\otimes N), H(M)]^{\bullet-s}\,.
\]
Now suppose that $N$ is $k$-connective and $M$ is $l$-coconnective.
Since $A$ is simply connected, $\overline{A}$ is $2$-connective and hence
$
H(\overline{A}{}^{\otimes s} \otimes N) \cong H(\overline{A})^{\otimes s} \otimes H(N)
$
vanishes in degrees $< 2s+ k$.
In particular, any map of graded rational vector spaces $H(\overline{A}{}^{\otimes s} \otimes N) \to H(M)$ of degree $q$ vanishes identically if $2s+k+q > l$.
The long exact sequence of cohomology groups
\[
\dotsb 
\longrightarrow
[H(\overline{A}{}^{\otimes (s+1)}\otimes N), H(M)]^{q - s -1}
\longrightarrow
H^q[\mathcal{F}_{s+1}, M]_A
\longrightarrow
H^q[\mathcal{F}_s, M]_A
\longrightarrow
\dotsb
\]
implies that $H^q [\mathcal{F}_{s+1}, M]_A\to  H^q [\mathcal{F}_{s}, M]_A$ is an isomorphism whenever $q> l-s-k$.
This lower bound decreases with $s$, which implies that the sequence 
\[
\dotsb \subset 
Z^{p,q}_{r+1} \subset Z^{p,q}_r
\subset \dotsb \subset Z^{p,q}_1
\]
satisfies the Mittag-Leffler condition for all $p,q$ and hence $\mathrm{lim}^1_r Z^{p,q}_r = 0$.
By \cite[Theorem 7.4]{boardman_conditionally_1999}, this is equivalent to the statements that (i) $\mathrm{lim}^1_r H^\bullet[\mathcal{F}_r, M]_A = 0$  and (ii) the spectral sequence converges strongly to $\mathrm{lim}_r H^\bullet[\mathcal{F}_r, M]_A$.
Since $\mathrm{colim}_r \mathcal{F}_r = B(A;A;N)$, we have $\mathrm{lim}_r [\mathcal{F}_r, M]_A \cong [B(A;A;N), M]_A$ and the Milnor-style short exact sequence
\[
0
\longrightarrow
\mathrm{lim}^1_r H^{\bullet-1}[\mathcal{F}_r, M]_A
\longrightarrow
H^\bullet[B(A;A;N), M]_A
\longrightarrow
\mathrm{lim}_r H^\bullet[\mathcal{F}_r, M]_A
\longrightarrow
0
\]
shows that $H^\bullet[B(A;A;N), M]_A \cong \mathrm{lim}_r H^\bullet[\mathcal{F}_r, M]_A$ by vanishing of the $\mathrm{lim}^1$ term.
The spectral sequence thus converges strongly to $H^\bullet[B(A;A;N), M]_A \cong \mathrm{Ext}^\bullet_{A}(N,M)$,
\[
E^{p,q}_2 \Longrightarrow
\mathrm{Ext}^q_A(N,M)\,,
\]
and the result follows from the identification $E_2^{p,q} \cong \mathrm{Ext}_{H(A)}^{p, q-p}(H(N), H(M))$.
\end{proof}

Interpreted in terms of the rational homotopy theory of parametrised spectra, the previous result gives rise to an \lq\lq Adams-style'' spectral sequence in topology:
\begin{proposition}
\label{prop:SSAdamsStyle}
Let $A$ be a cofibrant simply connected cdga of finite homotopical type with $X =\mathfrak{S}(A)$ and let $P, Q$ be bounded below $X$-spectra of finite rational type such that the rational cohomology of $X_! P$ is bounded above.
Then there is a strongly convergent spectral sequence
\[
\mathrm{Ext}^{p,q}_{H^\bullet(X)}\big(
H^\bullet (X_! Q), H^\bullet (X_! P)\big)
\Longrightarrow
\{P, Q\}^{p+q}_X\otimes_\mathbb{Z}\mathbb{Q}\,.
\] 
\end{proposition}
\begin{proof}
Let $P$ and $Q$ correspond to the $A$-modules $M$ and $N$ respectively under the rational homotopy theory equivalence of Theorem \ref{thm:RatParamHomThry} (note that $P$ and $Q$ are automatically nilpotent).
By Proposition \ref{prop:TopInterp}, the $A$-module $M$ models the $H^\bullet(X)$-action on $H^\bullet(X_!P)$ in cohomology.
Likewise, $N$ models the cohomology action on $H^\bullet(X_!Q)$ so that
\[
\mathrm{Ext}^{p,q}_{H^\bullet(X)}\big(H^\bullet(X_! Q),H^\bullet(X_! P)\big)
\cong 
\mathrm{Ext}^{p,q}_{H^\bullet(A)}\big(H^\bullet(N),H^\bullet(M)\big)\,.
\]
We now use Proposition \ref{prop:FibStapMapandExt} to identify the target of the spectral sequence of Proposition \ref{prop:HyperExt}.
\end{proof}

\begin{remark}
In the case that $P$, $Q$ are fibrewise suspension spectra of finite relative complexes we recover the spectral sequences of \cite[Proposition 15.25]{crabb_fibrewise_1998}.
\end{remark}

\begin{example}
For a cofibrant simply connected cdga $A$ of finite rational type with spatial realisation $X=\mathfrak{S}(A)$, consider a fibration $p\colon Y\to X$ where $Y$ is of finite rational type.
The fibrewise suspension spectrum $\Sigma^{\infty}_{X+} Y$ is a bounded below $X$ spectrum of finite rational type and, writing $F$ for the fibre of $p$, we have
\begin{align*}
\spi_{-k} F\otimes_\mathbb{Z} \mathbb{Q}  
&
\cong \{S^{-k}, \Sigma^\infty_+ F\}\otimes_\mathbb{Z} \mathbb{Q}
\\
&\cong \{x_! S^{-k}, \Sigma^\infty_{X+} Y\}_X \otimes_\mathbb{Z} \mathbb{Q}
\\
&
\cong
\mathrm{Ext}^k_{A} (\mathfrak{A}(Y), \mathbb{Q})
\end{align*}
by Proposition \ref{prop:FibStapMapandExt}.
This is a \lq\lq strict'' version of the isomorphism $\mathrm{Ext}^\bullet_{
C^\ast(X)}(C^\ast(Y), \mathbb{Q})\cong H_{-\bullet}(F)$.
By Proposition \ref{prop:SSAdamsStyle} we have a strongly convergent spectral sequence
\[
\mathrm{Ext}^{p,q}_{H^\bullet(X)}\big(
H^\bullet (Y), \mathbb{Q} \big)
\Longrightarrow
H_{-(p+q)} (F)
\]
in this case.
\end{example}

\begin{example}
Specialising the previous example to $\ast \to X$ for $X$ simply-connected, we get isomorphisms
\[
H_{-\bullet}(\Omega X) 
\cong 
\{\Sigma^{\infty}_{X+}\ast, \Sigma^{\infty}_{X+}\ast\}_X^\bullet\otimes_\mathbb{Z}\mathbb{Q}
\cong
\mathrm{Ext}^{\bullet}_{A}(\mathbb{Q},\mathbb{Q})\,.
\] 
Under this isomorphism, the Pontrjagin product on $H_\bullet(\Omega X)$ coincides with the Yoneda product on $\mathrm{Ext}$-groups.
\end{example}

Using minimal models we construct another spectral sequence for $\mathrm{Ext}_A$ which has slightly better convergence properties.
Let $N$ be an $A$-module with bounded below cohomology which has minimal model $A\otimes V\to N$.
By minimality, $V$ has a basis $\{v_\alpha\}_{\alpha\in \mathcal{I}}$ indexed by a well-ordered set $\mathcal{I}$ such that (i) $dv_\beta \in A\otimes \mathrm{span}\{v_\alpha\mid \alpha < \beta\}$ and (ii) $\alpha\leq \beta\Rightarrow |v_\alpha|\leq |v_\beta|$.
With notation as in Section \ref{ss:Postnikov}, consider the induced sequence of cofibrations of minimal $A$-modules (\lq\lq Postnikov tower'')
\[
\begin{tikzcd}
\dotsb 
\ar[r, rightarrowtail]
&
A\otimes V^{\leq (k-1)}
\ar[r, rightarrowtail]
&
A\otimes V^{\leq k}
\ar[r, rightarrowtail]
&
A\otimes V^{\leq (k+1)}
\ar[r, rightarrowtail]
&
\dotsb
\end{tikzcd}
\]
The colimit of this sequence is $A\otimes V$ and the (homotopy) cofibre of $A\otimes V^{\leq (k-1)}\to A\otimes V^k$ is the minimal $A$-module $A\otimes V^{=k}$.
Note that the differential on $A\otimes V^{=k}$ is not necessarily trivial if $A$ is not simply connected (cf.~Remark \ref{rem:ModularEM}).

Forming compelexes of $A$-linear maps to a fixed $A$-module $M$, we get a sequence of fibrations 
\[
\begin{tikzcd}
\dotsb
\ar[r, twoheadrightarrow]
&
{[A\otimes V^{\leq k}, M]_A}
\ar[r,  twoheadrightarrow]
&
{[A\otimes V^{\leq (k-1)}, M]_A}
\ar[r,  twoheadrightarrow]
&
\dotsb
\\
&
{[A\otimes V^{=k}, M]_A}
\ar[u]
&
{[A\otimes V^{=(k-1)}, M]_A}
\ar[u]
&
\end{tikzcd}
\]
with (homotopy) fibres as indicated.
Passing to cohomology, the sequence of fibrations gives rise to an exact couple
\begin{equation}
\label{eqn:MinExactCouple}
\begin{tikzcd}
{H^q [A\otimes V^{\leq p}, M]_A}
\ar[r, "i"]
\ar[from = d, "k"]
&
{H^q [A\otimes V^{\leq (p-1)}, M]_A}
\ar[dl, bend left = 10, dashed, "j"]
\\
{H^q [A\otimes V^{=p},M]_A}
\end{tikzcd}
\end{equation}
and hence a spectral sequence with $E_1^{p,q} = H^q [A\otimes V^{=p},M]_A$.
\begin{proposition}[Minimal spectral sequence]
\label{prop:minSpecSeq}
Let $N$ be a bounded below $A$-module with minimal model $A\otimes V\to N$.
If the $A$-module $M$ is bounded above, the spectral sequence of the exact couple \eqref{eqn:MinExactCouple} converges strongly:
\[
E^{p,q}_1 = H^q [A\otimes V^{=p},M]_A
\Longrightarrow
\mathrm{Ext}^{q}_A (N, M)\,.
\]
\end{proposition}
\begin{proof}
Since $N$ is bounded below there is some integer $k_0$ such that $V^{\leq k} = 0$ for $k< k_0$.
The spectral sequence of the exact couple \eqref{eqn:MinExactCouple} is thus a half-plane spectral sequence with entering differentials and we appeal to results of Boardman to deduce strong convergence to the limit.
Let
\[
Z^{s, \bullet}_r = k^{-1}\big(\mathrm{im} \{
\!\!
\begin{tikzcd}[sep = small]
H^\bullet {[A\otimes V^{\leq (s+r-1)},M]_A }
\ar[r, "i"]
&
\dotsb
\ar[r, "i"]
&
H^\bullet {[A\otimes V^{\leq s},M]_A }
\end{tikzcd}
\!\!
\}\big)
\subset
H^q [A\otimes V^{=p},M]_A\,.
\]
We claim that the sequence 
\[
\dotsb \subset Z^{s, t}_{r+1} \subset Z^{s,t}_r \subset \dotsb \subset Z^{s,t}_1
\]
satisfies the Mittag-Leffler condition for all $s,t$ and hence $\mathrm{lim}^1_r Z^{s,t}_r = 0$.
\medskip

\noindent \emph{Proof of claim.}
Suppose without loss of generality that $M$ is $l$-coconnective, i.e.~$H^{> l}(M) = 0$.
For any $k$, the $A$-module $A\otimes V^{=k}$ is minimal so that we may find a basis $\{w_\alpha\}_{\alpha \in \mathcal{J}}$ of $V^{=k}$ indexed by a well-ordered set $\mathcal{J}$ and such that $dw_\alpha\in A\otimes \mathrm{span}\{w_\beta\mid \beta< \alpha\}$ for all $\alpha\in \mathcal{J}$.
Writing $V^{= k}_{\leq \alpha} =\mathrm{span}\{w_\beta\mid \beta\leq \alpha\}$, there are cofibre sequences of minimal $A$-modules
\[
A\otimes V^{= k}_{\leq \alpha} \longrightarrow A\otimes V^{= k}_{\leq (\alpha+1)}
\longrightarrow
A\langle w_{\alpha +1}\rangle
\]
where $A\langle w_{\alpha +1}\rangle$ is the free $A$-module on $\mathrm{span}\{w_{\alpha+1}\}$.
We get a fibre sequence of cochain complexes of $A$-linear maps
\[
M[k]\cong [A\langle w_{\alpha +1}\rangle , M]_A
\longrightarrow
[A\otimes V^{= k}_{\leq (\alpha+1)}, M]_A
\longrightarrow
[A\otimes V^{= k}_{\leq \alpha}, M]_A\,,
\]
and
since $H^{q}M[k] = H^{q+k}M$ vanishes for $q+k > l$ the long exact sequence on cohomology implies that $H^q[A\otimes V^{= k}_{\leq (\alpha+1)}, M]_A
\to
H^q [A\otimes V^{= k}_{\leq \alpha}, M]_A$ is an isomorphism for $q> l-k$.

Let $\alpha_0$ be the least element of $\mathcal{J}$, then $H^q [A\otimes V^{=k}_{\leq \alpha_0}, M]_A \cong H^q M[-k]$ vanishes for $q>l-k$.
Using the above, transfinite induction over $\mathcal{J}$ shows that  $H^q [A\otimes V^{=k}_{\leq \alpha}, M]_A$ vanishes for $q> l-k$.
Observing that $[A\otimes V^{=k}, M]_A \cong \mathrm{lim}_\mathcal{J} [A\otimes V^{=k}_{\leq \alpha}, M]_A$, from the Milnor-style short exact sequence
\[
0
\longrightarrow
\mathrm{lim}^1_\mathcal{J} H^{\bullet -1 }[A\otimes V^{=k}_{\leq \alpha}, M]_A
\longrightarrow
H^\bullet [A\otimes V^{=k}, M]_A
\longrightarrow
\mathrm{lim}_\mathcal{J} H^\bullet [A\otimes V^{=k}_{\leq \alpha}, M]_A
\longrightarrow
0
\]
we have that $H^q [A\otimes V^{=k}, M]_A$ vanishes for $q > l+1-k$.

Examining the long exact sequence in cohomology induced by the fibre sequence
\[
[A\otimes V^{=k}, M]_A
\longrightarrow
[A\otimes V^{\leq k}, M]_A
\longrightarrow
[A\otimes V^{\leq (k-1)}, M]_A
\]
we find that $H^q [A\otimes V^{\leq k}, M]_A
\to H^q [A\otimes V^{\leq (k-1)}, M]_A$ is an isomorphism for $q > l+1-k$.
That the sequence $\dotsb \subset Z^{s, t}_{r+1} \subset Z^{s,t}_r \subset \dotsb \subset Z^{s,t}_1$ satisfies the Mittag-Leffler condition now follows from the fact that this lower bound decreases as $k$ increases.
\hfill
\qedsymbol
\medskip

By \cite[Theorem 7.4]{boardman_conditionally_1999}, $\mathrm{lim}^1_r Z^{s,t}_r = 0$ is equivalent to strong convergence of  spectral sequence to $\mathrm{lim}_r H^\bullet [A\otimes V^{\leq p}, M]_A$ and $\mathrm{lim}^1_r H^\bullet [A\otimes V^{\leq r}, M]_A = 0$.
Given the Milnor-style short exact sequence 
\[
0
\longrightarrow
\mathrm{lim}^1_r H^{\bullet-1}[A\otimes V^{\leq r}, M]_A
\longrightarrow
H^\bullet [A\otimes V, M]_A
\longrightarrow
\mathrm{lim}_r H^\bullet [A\otimes V^{\leq r}, M]_A
\longrightarrow
0
\]
and vanishing of the $\mathrm{lim}^1$ term, we conclude that the spectral sequence converges strongly to $H^\bullet [A\otimes V, M]_A \cong \mathrm{Ext}^\bullet_A (N, M)$.
\end{proof}

We now identify this spectral sequence in topology.
\begin{proposition}
\label{prop:minSpecSeqTop}
Let $A$ be a cofibrant connected cdga of finite homotopical type with $X=\mathfrak{S}(A)$ and let $P,Q $ be a bounded below $X$-spectra of finite rational type such that the cohomology of $X_!P$ is bounded above.
Then there is a strongly convergent spectral sequence
\[
E_1^{s,t} = \{P, Q_{=s}\}_{X}^t\otimes_\mathbb{Z}\mathbb{Q}
\Longrightarrow
\{P, Q\}_X^{t}\otimes_\mathbb{Z}\mathbb{Q}
\]
\end{proposition}
\begin{proof}
Suppose that the $A$-modules $M$, $N$ correspond to $P$ and $Q$, respectively, under the rational homotopy theory equivalence of Theorem \ref{thm:RatParamHomThry}.
Let $A\otimes V\to N$ be a minimal model for the rational homotopy type of $Q$.
By Proposition \ref{prop:TopInterp} the cohomology of $M$ is bounded above so we have a strongly convergent spectral sequence as in Proposition \ref{prop:minSpecSeq}.
Using Proposition \ref{prop:FibStapMapandExt} we identify the target of the spectral sequence as $\{P, Q\}_X^{t}\otimes_\mathbb{Z}\mathbb{Q} \cong \mathrm{Ext}_A^t (N, M)$ and the $E_1$-term is identified as $\{P_{=s}, Q\}_X^{t}\otimes_\mathbb{Z}\mathbb{Q} \cong \mathrm{Ext}_A^t (A\otimes V_{=s}, M)$ using Corollary \ref{cor:kthSlice}.
\end{proof}

\begin{remark}
In the case that $A$ (and hence $X$) is simply connected, any minimal $A$-module of the form $A\otimes W$ with $W$ concentrated in a single degree is necessarily free. 
The $E_1$-term in the spectral sequence arising from the minimal model (Proposition \ref{prop:minSpecSeq}) therefore simplifies: 
\begin{align*}
E_1^{s,t} &= H^t[A\otimes V^{=s}, M]_A
\\
&\cong H^{t}[V^{=s}, M]
\\
&\cong 
(V^{=s})^\vee \otimes H^{t+s}(M)\,.
\end{align*}
When the parameter space is simply connected, the spectral sequence of Proposition \ref{prop:minSpecSeqTop} thus becomes
\[
\mathrm{Hom}(\spi_{s+t} (X_! P), \spi_s(x^\ast Q))\otimes_\mathbb{Z}\mathbb{Q}\Longrightarrow
\{P, Q\}_X^t\otimes_\mathbb{Z}\mathbb{Q}\,,
\]
using $H^\bullet(X_!P) \cong \mathrm{Hom}(\spi_{\bullet}(X_! P),\mathbb{Q})$ and Lemma \ref{lem:FibofMinModSpec}.
\end{remark}

\begin{remark}
The topological spectral sequence of Proposition \ref{prop:minSpecSeqTop} is obtained by identifying the terms in an algebraic spectral sequence constructed using $A$-modules. 
That algebraic spectral sequence is itself obtained by working with an algebraic analogue of the Postnikov filtration.
A similar method can be used to give a more direct, topological derivation of the spectral sequence of Proposition \ref{prop:minSpecSeqTop}.
To see this, let $P, Q$ be parametrised spectra over an arbitrary base space $X$.
The tower of Postnikov sections $\dotsb \to Q_{\leq k}\to Q_{\leq (k-1)}\to \dotsb$ gives rise to a sequence of fibrations of spectra
\[
X_\ast F_X(P,Q)
\longrightarrow
\dotsb
\longrightarrow
X_\ast F_{X}(P,Q_{\leq k})
\longrightarrow
X_\ast F_{X}(P, Q_{\leq (k-1)})
\longrightarrow
\dotsb
\] 
in which the fibre of $F_{X}(P,Q_{\leq k})
\to
F_{X}(P, Q_{\leq (k-1)})$ at the zero map is $F_X (P, Q_{=k})$.
Upon taking stable homotopy groups, this sequence of fibrations determines an exact couple and hence a spectral sequence with $E_1^{s,t} = \spi_{-t} X_\ast F_X(P, Q_{=s}) \cong \{P, Q_{=s}\}_X^{t}$.
The convergence properties of this spectral sequence depend on $P$ and $Q$:
if $Q$ is bounded below we have a half-page spectral sequence with entering differentials and if $P$ is also bounded above, the spectral sequence converges strongly
\[
E_1^{s,t} = \{P, Q_{=s}\}^t_X 
\Longrightarrow
\{P, Q\}^t_X\,.
\]
The convergence argument uses Boardman's criterion \cite[Theorem 7.4]{boardman_conditionally_1999}.
\end{remark}

%%%%%%%%%%%%%%%%%
%%%%%%%%%%%%%%%%%

%%%%%%%%%%%%%%%%%
%%%%%%%%%%%%%%%%%

%%%%%%%%%%%%%%%%%

\end{document}